\documentclass[11pt,reqno]{amsart}
\usepackage{amsmath}
\usepackage{amsthm}
\usepackage{graphicx}
\usepackage{hyperref}
\usepackage{mathrsfs}
\usepackage{color}
\usepackage{bm}
\usepackage{amsfonts}
\usepackage{amssymb}
\usepackage{enumerate}
\usepackage{soul}
\usepackage[margin=1.1in]{geometry}
\usepackage[normalem]{ulem}

\makeatletter
\def\@settitle{\begin{center}%
  \baselineskip14\p@\relax
  \bfseries
  \uppercasenonmath\@title
  \@title
  \ifx\@subtitle\@empty\else
     \\[1ex]\uppercasenonmath\@subtitle
     \footnotesize\mdseries\@subtitle
  \fi
  \end{center}%
}
\def\subtitle#1{\gdef\@subtitle{#1}}
\def\@subtitle{}
\makeatother

\numberwithin{equation}{section}

\newtheorem{theorem}{Theorem}[section]
\newtheorem{lemma}[theorem]{Lemma}
\newtheorem{proposition}[theorem]{Proposition}
\newtheorem{corollary}[theorem]{Corollary}

\theoremstyle{definition}

\newtheorem{remark}[theorem]{Remark}

\newtheorem{innercustomthm}{Assumption}
\newenvironment{assumption}[1]
  {\renewcommand\theinnercustomthm{#1}\innercustomthm}
  {\endinnercustomthm}

\def\E{{\mathbb E}}
\def\EE{{\mathbb E}}
\def\R{{\mathbb R}}
\def\RR{{\mathbb R}}
\def\N{{\mathbb N}}
\def\FF{{\mathbb F}}
\def\PP{{\mathbb P}}
\def\P{{\mathcal P}}
\def\cP{{\mathcal P}}

\def\cH{{\mathcal H}}

\def\X{{\mathcal X}}
\def\Y{{\mathcal Y}}
\def\Z{{\mathcal Z}}

\def\L{{\mathcal L}}
\def\cL{{\mathcal L}}
\def\cG{{\mathcal G}}

\def\W{{\mathcal W}}
\def\cW{{\mathcal W}}
\def\A{{\mathcal A}}
\def\F{{\mathcal F}}

\def\bU{V}
\def\C{{\mathcal C}}

\definecolor{darkspringgreen}{rgb}{0.09, 0.45, 0.27}

\def\newD{{\lambda_d}}
\def\newDD{{\lambda}}

\newcommand{\bb}{\bar{b}}
\newcommand{\cc}{\bar{g}}

\def\newexp{{q}}

\title[A central limit theorem for mean field games]{From the master equation to mean field game limit theory: A central limit theorem}
\author{Fran\c{c}ois Delarue, Daniel Lacker, and Kavita Ramanan}

\begin{document}

\begin{abstract}
Mean field games (MFGs) describe the limit, as $n$ tends to infinity, of stochastic differential games with $n$ players interacting with one another through their common empirical distribution. Under suitable smoothness assumptions that guarantee uniqueness of the MFG equilibrium, a form of law of large of numbers (LLN), also known as propagation of chaos, has been established to show that the MFG equilibrium arises as the limit of the sequence of empirical measures of the $n$-player game Nash equilibria, including the case when player dynamics are driven by both idiosyncratic and common sources of noise. The proof of convergence relies on the so-called master equation for the value function of the MFG, a partial differential equation on the space of probability measures. In this work, under additional assumptions, we establish a functional central limit theorem (CLT) that characterizes the limiting fluctuations around the LLN limit as the unique solution of a linear stochastic PDE. The key idea is to use the solution to the master equation to construct an associated McKean-Vlasov interacting $n$-particle system that is sufficiently close to the Nash equilibrium dynamics of the $n$-player game for large $n$. We then derive the CLT for the latter from the CLT for the former. Along the way, we obtain a new multidimensional CLT for McKean-Vlasov systems. We also illustrate the broader applicability of our methodology by applying it to establish a CLT for a specific linear-quadratic example that does not satisfy our main assumptions, and we explicitly solve the resulting stochastic PDE in this case. 
\end{abstract}

\maketitle


\tableofcontents

\section{Introduction}

\textit{Finite games and mean field games.} Mean field games (MFG), introduced independently in \cite{lasrylions,lasry2006jeux1,lasry2006jeux2} and  \cite{huang2006large,huang2007large},
 are models of competition among a continuum of symmetric agents, each of whom dynamically controls a state variable; see the forthcoming books
 \cite{CarmonaDelarue_book_I,CarmonaDelarue_book_II} for an overview.  Equilibria of MFGs  were introduced as potentially more tractable approximations of Nash equilibria of  large finite systems of agents
 {interacting with one
  another through their common empirical distribution.}  
While much of the theoretical work of the past decade has focused on questions of existence and uniqueness of MFG equilibria, the probabilistic limit theory is less well understood.  
A law of large numbers has only recently come into focus \cite{cardaliaguet-delarue-lasry-lions,lacker2016general,fischer2017connection}, clarifying how a MFG arises as the limit of a suitable sequence of $n$-player games as $n\rightarrow\infty$, {in the presence of both
  idiosyncratic and common sources of noise.}  
In particular, given for each $n$ a Nash equilibrium for the $n$-player game, the limit of the sequence of
empirical distributions of state variables can be effectively characterized 
in terms of the equilibria of the MFG.  The goal of this paper is to complement this law of large numbers with a central limit theorem (CLT). Using similar techniques, a large deviation principle and non-asymptotic concentration 
bounds are obtained in a companion paper \cite{dellacram18b}. 

The main tool in our analysis is the \emph{master equation}, an infinite-dimensional partial differential equation (PDE), which plays a similar role to the Hamilton-Jacobi-Bellman (HJB) equation  in classical stochastic control theory. Lions \cite{lions-collegedufrance} demonstrated how the master equation can be used to construct an equilibrium for the MFG, and this construction was developed further in \cite{bensoussan2015master,bensoussan2017interpretation,carmona2014master}.
To solve the master equation is no simple matter, though some preliminary well-posedness results may be found in \cite{gangbo2015existence,chassagneux2014probabilistic} in the case without common noise and \cite{cardaliaguet-delarue-lasry-lions,CarmonaDelarue_book_II} in the case with common noise.

A major breakthrough came in \cite{cardaliaguet-delarue-lasry-lions} with the discovery  that, for a class of MFG with a unique equilibrium, when the associated master equation has a smooth enough solution, it can be used to prove  convergence of $n$-player Nash equilibria to the unique MFG equilibrium.  The essential idea is that the solution to the  master equation can be used to build a system of $n$ interacting diffusions of McKean-Vlasov type
(see \cite{McKean1907,sznitman1991topics}
for standard references on McKean-Vlasov equations)
 that  is  quantitatively ``close" to the true $n$-player Nash equilibrium system and 
 whose empirical measure converges to the MFG equilibrium. This immediately implies that the empirical measure of the $n$-player Nash equilibrium system converges to the MFG equilibrium.  In this paper we refine the ``closeness'' estimates  to show that the distance between the McKean-Vlasov and Nash systems decays rapidly enough that the fluctuations are the same.

The results of this paper and the companion \cite{dellacram18b} mark the first probabilistic limit theorems for MFGs beyond the law of large numbers, at least for diffusion-based models. The very recent simultaneous works \cite{cecchin2017probabilistic,cecchin2017convergence,bayraktar2017analysis} carry out a similar program for MFGs with finite state space (and without common noise), using the (finite-dimensional) master equation to connect the $n$-player equilibrium to a more classical interacting particle system, and then transferring limit theorems (a law of large numbers, CLT, and LDP) from the latter to the former. Loosely related ideas appeared also in \cite{ahuja2017asymptotic}, which uses FBSDE methods to study the convergence of MFGs as the common noise parameter vanishes.
\vspace{4pt}

\textit{Main results.} Let us explain the idea and main results more clearly. We consider $n$-player stochastic differential games, in which agent $i$ chooses a Markovian control $\alpha^i=\alpha^i(t,\bm{X}_t)$ lying in an action space $A$ to influence an $(\R^d)^n$-valued state process $\bm{X}=(X^1,\ldots,X^n)$ given by
\begin{align}
dX^i_t &= b(X^i_t,m^n_{\bm{X}_t},\alpha^i(t,\bm{X}_t))dt + \sigma dB^i_t + \sigma_0 dW_t. \label{intro:dynamics}
\end{align}
Throughout, $W$ and $B^1,\ldots,B^n$ are independent Wiener processes, and we write 
\[
m^n_{\bm{x}} = \frac{1}{n}\sum_{k=1}^n\delta_{x_k}
\]
to denote the empirical measure of a vector $\bm{x}=(x_1,\ldots,x_n)$ in $(\R^d)^n$.  Agent $i$ seeks to minimize the cost functional
\begin{align*}
\E\left[\int_0^Tf(X^i_t,m^n_{\bm{X}_t},\alpha^i(t,\bm{X}_t))dt + g(X^i_T,m^n_{\bm{X}_T})\right],
\end{align*}
where the time horizon $T \in (0,\infty)$ is fixed.
Define the Hamiltonian
\begin{align*}
H(x,m,y) = \inf_{a \in A}\bigl[b(x,m,a) \cdot y + f(x,m,a)\bigr],
\end{align*}
and let $\widehat{\alpha}(x,m,y)$ denote a minimizer, which we will always assume to exist.
Nash equilibria of this game (in closed-loop strategies), defined precisely in Section \ref{se:Nashsystems}, can be studied in terms of solutions $(v^{n,i})_{i=1}^n$ of the PDE system 
\begin{align}
\nonumber
\partial_tv^{n,i}(t,\bm{x}) &+ H\left(x_i,m^n_{\bm{x}},D_{x_i}v^{n,i}(t,\bm{x})\right) + \sum_{j \neq i}D_{x_j}v^{n,i}(t,\bm{x}) \cdot b\Bigl(x_j,m^n_{\bm{x}},\widehat{\alpha}\bigl(x_j,m^n_{\bm{x}},D_{x_j}v^{n,j}(t,\bm{x})\bigr)\Bigr) \\
&+ \frac{1}{2}\sum_{j = 1}^n\mathrm{Tr}\left[D^2_{x_j,x_j}v^{n,i}(t,\bm{x})\sigma \sigma^\top \right] + \frac{1}{2}\sum_{j,k=1}^n\mathrm{Tr}\left[D^2_{x_j,x_k}v^{n,i}(t,\bm{x})\sigma_0 \sigma_0^\top \right] = 0,
\label{intro:Nash:system}
\end{align}
for $(t,\bm{x}) \in [0,T] \times (\R^d)^n$, with terminal condition $v^{n,i}(T,\bm{x}) = g(x_i,m^n_{\bm{x}})$. 
The quantity  $v^{n,i}$ describes the value function of the $i$th player in the $n$-player game.
We assume this PDE system has a smooth solution, in which case a (closed-loop) Nash equilibrium is given by
\[
\alpha^{n,i}(t,\bm{X}_t) = \widehat{\alpha}(X^i_t,m^n_{\bm{X}_t},D_{x_i}v^{n,i}(t,\bm{X}_t)).
\]
Substituting these controls into the dynamics \eqref{intro:dynamics}, we identify the Nash equilibrium empirical measure  $(m^n_{\bm{X}_t})_{t \in [0,T]}$. 
The value function of a typical player in the associated MFG is described by the so-called 
master equation, whose definition we postpone to Section \ref{se:Nashsystems}.
Assuming, among other things, the existence of a sufficiently smooth solution to the master equation, we show the following:  
\begin{enumerate}
\item As $n\rightarrow\infty$, $(m^n_{\bm{X}_t})_{t \in [0,T]}$ converges in law to the unique solution $(\mu_t)_{t \in [0,T]}$ to the MFG, which is itself a stochastic flow of probability measures. This was already shown in \cite[Theorem 2.15]{cardaliaguet-delarue-lasry-lions}  when the state space is the torus. Our extension to $\R^d$ is straightforward provided the solution to the master equation 
together with its derivatives have an appropriate rate of growth and the system 
\eqref{intro:Nash:system} has a solution. Similar questions are addressed in the forthcoming reference 
\cite[Chapter 13, Section 6.3]{CarmonaDelarue_book_II}.
\item Our main contribution is to show that the fluctuation process $(\sqrt{n}(m^n_{\bm{X}_t} - \mu_t))_{t \in [0,T]}$ converges in law (in a suitable distribution space) to the unique solution to a certain linear stochastic PDE driven by a space-time Gaussian noise, the coefficients of which depend on the solution to the master equation and on $(\mu_t)_{t \in [0,T]}$.
\end{enumerate}
\vspace{2pt}
These results are presented precisely in Section \ref{se:statements}.

\textit{Strategy of proof.} The idea driving all of the results, {both in this work and the companion paper
\cite{dellacram18b}}, is to use the solution to the master equation, $U=U(t,x,m)$, to construct an interacting diffusion system $\bm{\overline{X}}=(\overline{X}^1,\ldots,\overline{X}^n)$ of McKean-Vlasov type:
\begin{align}
d \overline{X}^i_t = b\Bigl(\overline{X}^i_t,m^n_{\bm{\overline{X}}_t},\widehat{\alpha}\bigl(\overline{X}^i_t,m^n_{\bm{\overline{X}}_t},D_xU(t,\overline{X}^i_t,m^n_{\bm{\overline{X}}_t}) \bigr)\Bigr)dt + \sigma dB^i_t + \sigma_0 dW_t, \label{intro:particles}
\end{align}
starting from $\overline{X}^i_0=X^i_0$, and driven by the same Wiener processes. 
By showing that the functions 
$(u^{n,i}(t,\bm{x}) := U(t,x_i,m^n_{\bm{x}}))_{i=1}^n$ \emph{nearly} solve the $n$-player PDE system written for $(v^{n,i})_{i=1}^n$ above (see the crucial Proposition \ref{pr:master-nearly-nash}), we can find good estimates between $D_{x_i}v^{n,i}$ and $D_{x_i}u^{n,i}$ (see Theorem \ref{th:mainestimate}), which enable an estimate of the distance between the empirical measures $m^n_{\bm{X}}$ and $m^n_{\bm{\overline{X}}}$.  
This was indeed the strategy of \cite{cardaliaguet-delarue-lasry-lions}, where it is shown that 
\begin{align}
\E\left[\W_{1,\C^d}(m^n_{\bm{X}},m^n_{\bm{\overline{X}}})\right] = O(n^{-2}), \label{intro:O(1/n^2)}
\end{align}
where  $\W_{p,{\mathcal C}^d}$ denotes the $p$-Wasserstein distance on the space of probability
measures on the path space ${\mathcal C}^d:=C([0,T];\R^d)$ with finite $p^{\textrm{\rm th}}$ moment.
Equivalently, \eqref{intro:particles} must be regarded as an approximating $n$-particle system of the 
$n$-player Nash equilibrium.
From McKean-Vlasov theory (see, e.g., \cite{gartner1988mckean,oelschlager1984martingale,sznitman1991topics}{, or \cite{carmona2016probabilistic,coghi2016,Dawson1995}
and \cite[Chapter 2, Section 2.1]{CarmonaDelarue_book_II}
 for the 
common noise case}), we conclude that both empirical measure sequences converge in law to the unique solution $(\mu_t)_{t \in [0,T]}$ of the (conditional) McKean-Vlasov SDE
\begin{align*}
dX_t &= b\Bigl(X_t,\mu_t,\widehat{\alpha}\bigl(X_t,\mu_t,D_xU(t,X_t,\mu_t)\bigr)\Bigr)dt + \sigma dB_t + \sigma_0 dW_t, \\ 
\mu_t &= \text{Law}(X_t \, | \, (W_s)_{s \le t}).
\end{align*}
Moreover, it is known from early work on the master equation that this limit $(\mu_t)_{t \in [0,T]}$ is the unique equilibrium of the MFG, which is to say that 
$(X_{t})_{t \in [0,T]}$ solves the following optimization problem in the environment $(\mu_{t})_{t \in [0,T]}$:
\begin{align*}
\begin{cases}
\displaystyle \sup_\alpha \E\left[\int_0^Tf(X_t,\mu_t,\alpha_t)dt + g(X_T,\mu_T)\right], 
\vspace{5pt}
\\
\displaystyle \text{s.t. } dX_t = b(X_t,\mu_t,\alpha_t)dt + \sigma dB_t + \sigma_0 dW_t,
\end{cases}
\end{align*}
the supremum being taken over $A$-valued processes that are progressively measurable with respect to the filtration generated by $X_{0}$, $B$ and $W$. 
This leads to the convergence claimed in point (1) of the program detailed above.

Next, to prove the CLT described in point (2) above, we study first the McKean-Vlasov fluctuations $(\sqrt{n}(m^n_{\bm{\overline{X}}_t} - \mu_t))_{t \in [0,T]}$. The early results of \cite{tanaka1981central,sznitman1985fluctuation,meleard} only cover affine dependence on the mean field term, but requiring the interaction in \eqref{intro:particles} to be affine is far too restrictive, except perhaps in linear-quadratic models.  On the other hand, the later work \cite{kurtz-xiong} treats more general  mean field interactions to include the non-affine yet smooth interactions that we obtain from the solution to the master equation.  However, the state space in \cite{kurtz-xiong}  is only one-dimensional, and so 
we extend their result to our multi-dimensional setting; further, in order to track the influence of the dimension explicitly, we follow some of the arguments introduced in \cite{meleard}. While $O(n^{-1})$ would not suffice, the estimate \eqref{intro:O(1/n^2)} is sharp enough to show that the fluctuations of the McKean-Vlasov and Nash systems are then the same.

Under additional assumptions,  the estimate \eqref{intro:O(1/n^2)}  is further refined in the companion paper \cite{dellacram18b} to establish exponential closeness of the McKean-Vlasov and Nash systems, which is then used to obtain large deviations results and non-asymptotic concentration bounds.

\textit{Required assumptions and examples.} 
The above results are  proved under admittedly very strong hypotheses. The main {conditions}, Assumptions \ref{assumption:A}, \ref{assumption:B}, and \ref{assumption:B'}, are spelled out in Section \ref{se:assumptions} below, and {the CLT} requires some additional hypotheses ({see Assumption \ref{assumption:C}}). Most notably, we assume throughout that $U$ admits bounded derivatives up to order two 
in the measure argument.  For the  CLT, additional smoothness 
is required of the first and second order derivatives with respect to $m$, the precise form of which depends on the dimension $d$ of the state space. In addition, we require either that the Hamiltonian $H$ is Lipschitz in $y$ or that $\{v^{n,i} : n \in \N, \ i=1,\ldots,n\}$ are uniformly bounded. Although this rules out many natural examples, such as linear-quadratic models (see, however, the discussion of the next paragraph), the main well-posedness results of \cite{cardaliaguet-delarue-lasry-lions} ensure that our assumptions cover a broad class of coefficients.
More importantly, we avoid as much as possible imposing specific assumptions on the data $(b,f,g)$; for instance, in contrast with 
\cite{cardaliaguet-delarue-lasry-lions}, we do not impose any explicit monotonicity condition on $f$ and $g$.  Instead, we aim as much as possible to create a ``black box'' in the sense that our results should read as follows: If there exists a sufficiently regular solution to the master equation, then the results (1) and (2) described above are valid. We hope that our work  will apply not only to the setting of \cite{cardaliaguet-delarue-lasry-lions} but also to those illuminated by future research on the well-posedness of the master equation.

Despite the restrictive assumptions required for our general results, the same ideas can be successfully adapted to specific examples which fall outside the scope of our main theorems. 
In Section \ref{se:examples} we discuss the linear-quadratic model of \cite{carmona-fouque-sun}, which does not fit the assumptions 
of our main theorems but for which an explicit solution is known for both the $n$-player and MFG equilibria. 
These explicit solutions allow us to bypass the difficult estimates required in the general setting, and we again connect the $n$-player equilibrium and a corresponding McKean-Vlasov system. As before, this connection lets us transfer {the CLT}
for the latter system to the former. 
In other words, while a single unifying theorem seems out of reach for technical reasons, the strategy of our arguments seems much more broadly applicable.
In addition, we compute the explicit solution for the limiting SPDE in this specific model.

\textit{Organization of the paper.}
The paper is organized as follows. In Section \ref{se:Nash-Master} we carefully define the $n$-player game and the Nash system, the MFG and the the master equation, and the main sets of assumptions. 
Section \ref{se:statements} then states precisely the main results, with the LLN in Section \ref{se:LLN-statements} and the
CLT in Section \ref{se:CLT-statements}.  The key estimates relating the Nash system and master equation are developed in Section \ref{se:mainestimates}, to prepare for the proofs of the LLN in Section \ref{se:LLN} and the CLT in Section \ref{se:CLT-proofs}.
Section \ref{se:examples} treats the specific  model mentioned in the previous paragraph.
Relevant properties of derivatives of functionals on the Wasserstein space of 
probability measures are collected in 
Appendix \ref{ap:derivatives}.

\section{Nash systems and Master equations} \label{se:Nash-Master}

\subsection{Notation and model inputs}
\label{subs-notation}

For a normed space $(E,\|\cdot\|)$, let $\P(E)$ denote the set of Borel probability measures on $E$.
Throughout the paper we make use of the standard notation $\langle \mu,\varphi\rangle := \int_E\varphi\,d\mu$ for integrable functions $\varphi$ on $E$ and measures $\mu$ on $E$.  
 Given $p \in [1,\infty)$, we write $\P^p(E,\|\cdot\|)$, or simply $\P^p(E)$ if the norm is understood, for the set of $\mu \in \P(E)$ satisfying $\langle \mu, \|\cdot\|^p\rangle < \infty$. 
For a separable Banach space $(E,\|\cdot\|)$, we always endow $\P^p(E,\|\cdot\|)$ with the $p$-Wasserstein metric $\W_{p,E}$ defined by  
 \begin{align}
\label{Wassp}
\W_{p,E}(\mu,\nu) := \inf_\pi \left(\int_{E\times E}\|x-y\|^p\pi(dx,dy)\right)^{1/p},
\end{align}
where the infimum is over all probability measures $\pi$ on $E \times E$ with marginals $\mu$ and $\nu$. When the space $E$ is understood, we may omit it from the subscript in $\W_{p,E}$ by writing simply $\W_p$.

For a positive integer $k$, we always equip $\R^k$ with the Euclidean norm, denoted $|\cdot|$, unless stated otherwise. For fixed $T \in (0, \infty)$, we will make use of the path spaces
\[
\C^k := C([0,T];\R^k), 
\]
which are always endowed with the supremum norm $\|x\|_\infty = \sup_{t \in [0,T]}|x_t|$.
For $m \in \P(\C^k)$ and $t \in [0,T]$, we write $m_t$ for the time-$t$ marginal of $m$, i.e., the image of $m$ under  the map $\C^k \ni x \mapsto x_t \in \R^k$.

For a set $E$ and $n \in \N$, we often use boldface $\bm{x}=(x_1,\ldots,x_n)$ for an element of $E^n$, and we write
\[
m^n_{\bm{x}} := \frac{1}{n}\sum_{i=1}^n\delta_{x_i}
\]
for the associated empirical measure, which lies in $\P(E)$.

\subsection{Derivatives on Wasserstein space}
\label{subse:derivatives:m:P2}

To define the master equation, we must first introduce a suitable derivative for functions of probability measures. As we will be applying but not solving the master equation, we need very few properties of these derivatives, all of which are outlined here. Following \cite{cardaliaguet-delarue-lasry-lions}, we fix an exponent $\newexp \in [1,\infty)$, and we say that a function $\bU : \P^\newexp(\R^d) \rightarrow \R$ is ${{\mathscr C}^1}$ if there exists a continuous map $\frac{\delta \bU}{\delta m} : \P^\newexp(\R^d) \times \R^d \rightarrow \R$ satisfying
\begin{enumerate}[(i)]
\item For every $\W_{\newexp}$-compact set $K \subset \P^\newexp(\R^d)$, there exists $c <  \infty$ such that $\sup_{m \in K}|\frac{\delta \bU}{\delta m}(m,v)| \le c(1+|v|^\newexp)$ for all $v \in \R^d$.
\item For every $m,m' \in \P^\newexp(\R^d)$,
\begin{align}
\bU(m')-\bU(m) = \int_0^1\int_{\R^d}\frac{\delta \bU}{\delta m}((1-t)m + tm',v)\,(m'-m)(dv)\,dt. \label{def:measure-derivative}
\end{align}
\end{enumerate}
Note that the condition (i) is designed to make the integral in (ii) well-defined. Only one function $\frac{\delta \bU}{\delta m}$ can satisfy \eqref{def:measure-derivative}, up to a constant shift; that is, if $\frac{\delta \bU}{\delta m}$ satisfies \eqref{def:measure-derivative} then so does $\frac{\delta \bU}{\delta m} + c$ for any $c \in \R$. For concreteness we always choose the shift to ensure
\begin{align}
\int_{\R^d}\frac{\delta \bU}{\delta m}(m,v)\, m(dv)=0. \label{def:derivative-shift}
\end{align}

If $\frac{\delta \bU}{\delta m}(m,v)$ is continuously differentiable in $v$, we define its \emph{intrinsic derivative} $D_m\bU : \P^{\newexp}(\R^d) \times \R^d \rightarrow \R^d$ by
\begin{align}
\label{intrinsic}
D_m\bU(m,v) = D_v\left(\frac{\delta \bU}{\delta m}(m,v)\right),
\end{align} 
where we use the notation $D_v$ for the gradient in $v$. 
If, for each $v \in \R^d$, the map $m \mapsto \frac{\delta \bU}{\delta m}(m,v)$ is ${{\mathscr C}^1}$, then we say that $\bU$ is ${{\mathscr C}^2}$ and let $\frac{\delta^2 \bU}{\delta m^2}$ denote its derivative, or more explicitly, 
\begin{align*}
\frac{\delta^2\bU}{\delta m^2}(m,v,v') = \frac{\delta}{\delta m}\left(\frac{\delta \bU}{\delta m}(\cdot,v)\right)(m,v').
\end{align*}
We will also make some use of the derivative
\[
D_v D_m\bU(m,v) = D_v[D_m\bU(m,v)],
\]
when it exists, and we note that $D_vD_m\bU$ takes values in $\R^{d\times d}$; 
for some results, we will also consider higher order derivatives $D^k_{v} D_{m} \bU(m,v)$ with values in $\R^{d \times \ldots \times d}
\cong
\R^{d^{k+1}}$ for $k \in \N$.
Finally, if $\bU$ is ${{\mathscr C}^2}$ and if $\frac{\delta^2\bU}{\delta m^2}(m,v,v')$ is twice continuously differentiable in $(v,v')$, we let
\[
D_m^2\bU(m,v,v') = D^2_{v,v'}\frac{\delta^2\bU}{\delta m^2}(m,v,v')
\]
denote the $d \times d$ matrix of partial derivatives $(\partial_{v_i}\partial_{v'_j}
[\delta^2 \bU/\delta m^2](m,v,v'))_{i,j}$. Equivalently (see \cite[Lemma 2.4]{cardaliaguet-delarue-lasry-lions}),
\[
D_m^2\bU(m,v,v') = D_m(D_m\bU(\cdot,v))(m,v').
\]
Aside from these definitions, the only result we should state here is an important observation on how these derivatives interact with empirical measures, essentially known from \cite{cardaliaguet-delarue-lasry-lions}.
For completeness, the proof is given in Appendix \ref{ap:derivatives}.

\begin{proposition} \label{pr:empiricalmeasure}
Given $\bU : \P^{\newexp}(\R^d) \rightarrow \R$, define $u_n : (\R^d)^n \rightarrow \R$ by $u_n(\bm{x}) = \bU(m^n_{\bm{x}})$ for some fixed $n \ge 1$.
\begin{enumerate}[(i)]
\item If $\bU$ is ${{\mathscr C}^1}$ and if $D_m\bU$ exists and is bounded and jointly continuous, then $u_n$  is continuously differentiable, and
\begin{align}
D_{x_j}u_n(\bm{x}) = \frac{1}{n}D_m\bU(m^n_{\bm{x}},x_j), \text{ for } j=1,\ldots,n. \label{def:derivative-empirical-measure}
\end{align}
\item If $\bU$ is ${{\mathscr C}^2}$ and if $D_m^2\bU$ exists and is bounded and jointly continuous, then $u$ is twice continuously differentiable, and
\begin{align*}
D_{x_k}D_{x_j}u_n(\bm{x}) =  \frac{1}{n^2}D^2_m\bU(m^n_{\bm{x}},x_j,x_k) + \delta_{j,k}\frac{1}{n}D_vD_m\bU(m^n_{\bm{x}},x_j),
\end{align*}
where $\delta_{j,k} = 1$ if $j=k$ and $\delta_{j,k} = 0$ if $j\neq k$. 
\end{enumerate}
\end{proposition}

\subsection{Nash systems and $n$-player games} \label{se:Nashsystems}

We fix throughout the paper a filtered probability space $(\Omega,\F,\FF=(\F_t)_{t \in [0,T]},\PP)$, supporting independent $\FF$-Wiener processes $W$ of dimension $d_0$ and $(B^i)_{i=1}^\infty$ of dimension $d$
(we choose the dimension of the idiosyncratic noises equal to the dimension of the state space for convenience only), as well as a sequence of 
i.i.d.\ ${\mathcal F}_{0}$-measurable $\R^d$-valued initial states $(X^i_0)_{i=1}^{\infty}$ with distribution $\mu_0$.

We describe the $n$-player game and PDE systems first, deferring a precise statement of assumptions to Section \ref{se:assumptions}. 
We are given {an exponent $p^*$}, an action space $A$, assumed to be a Polish space, and Borel measurable functions
\begin{align*}
(b,f) &: \R^d \times \P^{{p^*}}(\R^d) \times A \rightarrow \R^d \times \R, \\
g &: \R^d \times \P^{{p^*}}(\R^d) \rightarrow \R,
\end{align*}
along with two matrices $\sigma \in \R^{d \times d}$ and $\sigma_0 \in \R^{d \times d_0}$.

In the $n$-player game, players $i=1,\ldots,n$ control state process $(\bm{X}_t=(X^1_t,\ldots,X^n_t))_{t \in [0,T]}$, given by
\begin{align}
dX^i_t = b(X^i_t,m^n_{\bm{X}_t},\alpha^i(t,\bm{X}_t))dt + \sigma dB^i_t + \sigma_0 dW_t, \label{def:SDEnplayer}
\end{align}
where we recall that $m^n_{\bm{X}_t}$ denotes the empirical measure of the vector $\bm{X}_t$.
Here $\alpha^i$ is the control chosen by player $i$ in feedback form. The objective of player $i$ is to try to choose $\alpha^i$ to minimize
\[
J^{n,i}(\alpha^1,\ldots,\alpha^n) = \E\left[\int_0^Tf(X^i_t,m^n_{\bm{X}_t},\alpha^i(t,\bm{X}_t))dt + g(X^i_T,m^n_{\bm{X}_T})\right].
\]
A (closed-loop) Nash equilibrium is defined in the usual way as a vector of feedback functions $(\alpha^1,\ldots,\alpha^n)$, where $\alpha^i : [0,T] \times (\R^d)^n \rightarrow A$ are such that the SDE \eqref{def:SDEnplayer} is unique in law, such that
\[
J^{n,i}(\alpha^1,\ldots,\alpha^n) \le J^{n,i}(\alpha^1,\ldots,\alpha^{i-1},\widetilde{\alpha},\alpha^{i+1},\ldots,\alpha^n),
\]
for any alternative choice of feedback control $\widetilde \alpha$.

From the work of \cite{bensoussan1983nonlinear}, we know that a Nash equilibrium can be built using a system of HJB equations.
Define the Hamiltonian $H : \R^d \times \P^{{p^*}}(\R^d) \times \R^d \rightarrow \R$ by
\begin{align*}
H(x,m,y) = \inf_{a \in A}\bigl[b(x,m,a) \cdot y + f(x,m,a)\bigr].
\end{align*}
Assume that this infimum is attained for each $(x,m,y)$, and
let $\widehat{\alpha}(x,m,y)$ denote a minimizer; we will place assumptions on the function $\widehat{\alpha}$ in the next section.
It is convenient to define the functionals $\widehat{b}$ and $\widehat{f}$ on $\R^d \times \P^{{p^*}}(\R^d) \times \R^d$ by
\begin{align}
\label{hatfns}
\widehat{b}(x,m,y) = b(x,m,\widehat{\alpha}(x,m,y) ) \quad \mbox{ and } \quad \widehat{f}(x,m,y) = f(x,m,\widehat{\alpha}(x,m,y) ), 
\end{align}
and note that then 
\begin{align} 
\label{exp-ham}
H(x,m,y) = \widehat{b}(x,m,y) \cdot y + \widehat{f}(x,m,y).
\end{align}
As already mentioned in the Introduction, see \eqref{intro:Nash:system}, the \emph{$n$-player Nash system} is a PDE system for $n$ functions, $(v^{n,i} : [0,T] \times (\R^d)^n \rightarrow \R)_{i=1}^{\infty}$, given by 
\begin{equation}
\label{def:Nashsystem}
\begin{split}
\partial_tv^{n,i}(t,\bm{x}) &+ H\left(x_i,m^n_{\bm{x}},D_{x_i}v^{n,i}(t,\bm{x})\right) + \sum_{j=1,j \neq i}^n D_{x_j}v^{n,i}(t,\bm{x}) \cdot \widehat{b}\left(x_j,m^n_{\bm{x}},D_{x_j}v^{n,j}(t,\bm{x})\right) \nonumber \\
&+ \frac{1}{2}\sum_{j = 1}^n\mathrm{Tr}\left[D^2_{x_j,x_j}v^{n,i}(t,\bm{x})\sigma \sigma^\top \right] + \frac{1}{2}\sum_{j,k=1}^n\mathrm{Tr}\left[D^2_{x_j,x_k}v^{n,i}(t,\bm{x})\sigma_0 \sigma_0^\top \right] = 0, 
\end{split}
\end{equation}
with terminal condition $v^{n,i}(T,\bm{x}) = g(x_i,m^n_{\bm{x}})$.

Using  (classical) solutions to the $n$-player Nash system, we may construct  an equilibrium for the $n$-player game. The $i^\text{th}$ agent uses the feedback control
\[
[0,T] \times (\R^d)^n \ni (t,\bm{x}) \mapsto \widehat{\alpha}\left(\bm{x},m^n_{\bm{x}},D_{x_i}v^{n,i}(t,\bm{x})\right).
\]
As a result, the equilibrium state process $\bm{X}=(X^1,\ldots,X^n)$ is governed by
\begin{align}
dX^i_t &= \widehat{b}(X^i_t,m^n_{\bm{X}_t},D_{x_i}v^{n,i}(t,\bm{X}_t))dt + \sigma dB^i_t + \sigma_0 dW_t.  \label{def:Nash-SDEsystem}
\end{align}
Under Assumption \ref{assumption:A} of Section \ref{se:assumptions} below, the SDE \eqref{def:Nash-SDEsystem} is uniquely solvable. 
Indeed,
due to assumption \ref{assumption:A}(4), the second derivatives of $v^{n,i}$ exist and are continuous, which ensures that $D_{x_i}v^{n,i}$ is locally Lipschitz. In light of Assumption \ref{assumption:A}(1) and the fact that $\bm{x} \mapsto m^n_{\bm{x}}$ is a Lipschitz function from $(\R^d)^n$ to $(\P^{{p^*}}(\R^d),\W_{{p^*}})$, this ensures that
the SDE system \eqref{def:Nash-SDEsystem} has a unique strong solution.

\subsection{The mean field game and master equation}
The master equation is a PDE for a function $U : [0,T] \times \R^d \times \P^{p*}(\R^d) \rightarrow \R$, given by
\begin{align}
0 = \ &\partial_tU(t,x,m) + H(x,m,D_xU(t,x,m)) \nonumber \\
	&+ \frac{1}{2}\mathrm{Tr}\left[(\sigma\sigma^\top+\sigma_0\sigma_0^\top) D_x^2U(t,x,m)\right] \nonumber \\
&+ \int_{\R^d}\widehat{b}(v,m,D_xU(t,v,m)) \cdot D_mU(t,x,m,v) \,dm(v) \nonumber \\
&+ \frac{1}{2}\int_{\R^d}\mathrm{Tr}\left[(\sigma\sigma^\top+\sigma_0\sigma_0^\top) D_vD_mU(t,x,m,v)\right]\,dm(v)   \label{def:masterequation}  \\
&+ \frac{1}{2}\int_{\R^d}\int_{\R^d}\mathrm{Tr}\left[\sigma_0 \sigma_0^\top D^2_mU(t,x,m,v,v')\right]\,dm(v)\,dm(v') \nonumber \\
&+ \int_{\R^d}\mathrm{Tr}\left[\sigma_0 \sigma_0^\top D_xD_mU(t,x,m,v)\right]\,dm(v), \nonumber
\end{align}
for $(t,x,m) \in (0,T) \times \R^d \times \P^{p*}(\R^d)$, with terminal condition $U(T,x,m) = g(x,m)$. The connection between the Nash system and the master equation will be made clear in Proposition \ref{pr:master-nearly-nash} below; roughly speaking, $v^{n,i}(t,\bm{x})$ is expected to be close to $U(t,x_{i},m^n_{\bm{x}})$ as $n$ tends to infinity. 

Just as the $n$-player Nash system was used to build an equilibrium for the $n$-player games, we will use the master equation to describe an equilibrium for the mean field game. First, consider the McKean-Vlasov equation
\begin{align}
d\X_t = \widehat{b}(\X_t,\mu_t,D_xU(t,\X_t,\mu_t))dt + \sigma dB^1_t + \sigma_0 dW_t, \quad \X_0 = X^1_0, \quad \mu = \L(\X | W), \label{def:MKV-conditional}
\end{align}
where $\L(\X | W)$ denotes the conditional law of $\X$ given (the path) $W$, viewed as a random element of $\P^{p^*}(\C^d)$.
Here, a solution $\X=(\X_t)_{t \in [0,T]}$ is required to be adapted to the filtration generated by the process $(X^1_0,W_t,B^1_t)_{t \in [0,T]}$. Notice that necessarily $\mu_t=\L(\X_t | W) = \L(\X_t | (W_s)_{s \in [0,t]})$ a.s., for each $t \in [0,T]$, because $(W_s-W_t)_{s \ge t}$ is independent of $(\X_s,W_s)_{s \le t}$.
Assumptions \ref{assumption:A}(1) and \ref{assumption:A}(5), stated in Section \ref{se:assumptions} below, ensure that there is a unique strong solution to \eqref{def:MKV-conditional}; this follows from a straightforward adaptation of the arguments of Sznitman \cite[Chapter 1]{sznitman1991topics} (cf. \cite[Section 7]{carmona2016probabilistic}
and {\cite[Chapter 2, Section 2.1]{CarmonaDelarue_book_II}}).
For the reader who is more familiar with the PDE formulation of mean field games, we emphasize that the process
$(\mu_t)_{t \in [0,T]}$ is a weak solution to the stochastic Fokker-Planck equation
\begin{equation*}
d \mu_{t}=
- \textrm{\rm div} \bigl( 
\widehat b(\cdot,\mu_{t},D_x  U(t,\cdot,\mu_{t})) \mu_{t}
\bigr) dt + \tfrac12
\textrm{\rm Tr}[ 
D^2_x \mu_{t}
( \sigma \sigma^\top 
+ \sigma_{0} \sigma_{0}^\top )  ]
 dt 
- 
\bigl( \sigma^{\top}_0 D_x \mu_{t}
\bigr) \cdot
dW_{t},  
\end{equation*}
for $t \in [0,T]$, which follows from a straightforward application of It\^o's formula to
the process 
$(\phi(X_{t}))_{t \in [0,T]}$ for smooth test functions $\phi$. 

Since $U$ is a classical solution to the master equation with bounded derivatives (see Assumptions \ref{assumption:A}(1) and \ref{assumption:A}(5) in Section \ref{se:assumptions} below), it is known that the measure flow $\mu$ constructed from the McKean-Vlasov equation \eqref{def:MKV-conditional} is the unique equilibrium of the mean field game; see for instance \cite[Proposition 5.106]{CarmonaDelarue_book_I}. A mean field game equilibrium is usually defined as a fixed point of the map $\Phi$ which sends a $W$-measurable random measure $\mu$ on $\C^d$ (such that $(\mu_t)_{t \in [0,T]}$ is adapted to the filtration generated by $W$) to a new random measure $\Phi(\mu)$, defined as follows:
\begin{enumerate}[(i)]
\item Solve the stochastic optimal control problem, with $\mu$ fixed:
\begin{align*}
\begin{cases}
\sup_\alpha \E\left[\int_0^Tf(X_t,\mu_t,\alpha_t)dt + g(X_T,\mu_T)\right], \\
\text{s.t. } dX_t = b(X_t,\mu_t,\alpha_t)dt + \sigma dB^1_t + \sigma_0 dW_t.
\end{cases}
\end{align*}
\item Letting $X^*$ denote the optimally controlled state process, set $\Phi(\mu) = \L(X^*|W)$.
\end{enumerate}
Note that if the optimization problem in step (i) has multiple solutions, the map $\Phi$ may be set-valued, and we seek $\mu$ such that $\mu \in \Phi(\mu)$.
The original formulation of Lasry and Lions \cite{lasrylions} is a forward-backward PDE system, which is essentially equivalent to this fixed point procedure, when $\sigma_0=0$. When $\sigma_0 \neq 0$, the forward-backward PDE becomes stochastic, but the same connection remains. For more details on the connection between the master equation and more common PDE or probabilistic formulations of mean field games, see \cite{bensoussan2015master,bensoussan2017interpretation,carmona2014master} or \cite[Section 1.2.4]{cardaliaguet-delarue-lasry-lions}. For our purposes, we simply take the McKean-Vlasov equation \eqref{def:MKV-conditional} as the definition of $\mu$.

\subsection{Assumptions} \label{se:assumptions}
The following standing assumption holds throughout the paper. In fact, more will be needed in Section 
\ref{se:statements}
to derive the central limit theorem,  
but  Assumptions \textbf{A}, \textbf{B} and \textbf{B'} below form the foundation behind the key estimates used this paper and as well as the 
companion \cite{dellacram18b}.

\begin{assumption}{\textbf{A}} \label{assumption:A} $\ $

\begin{enumerate}
\item The exponent $p^*$ lies in $[1,2]$. A minimizer $\widehat{\alpha}(x,m,y) \in \arg\min_{a \in A}\bigl[b(x,m,a) \cdot y + f(x,m,a)\bigr]$ exists for every $(x,m,y) \in \R^d \times \P^{{p^*}}(\R^d) \times \R^d$, 
such that the function $\widehat{b}(x,m,y)$ defined in \eqref{hatfns} is Lipschitz in all variables. That is, there exists $C <  \infty$ such that, for all $x,x',y,y' \in \R^d$ and $m,m' \in \P^{{p^*}}(\R^d)$, 
\begin{align*}
|\widehat{b}(x,m,y) - \widehat{b}(x',m',y')| \le C\left(|x-x'| + \W_{{p^*}}(m,m') + |y-y'|\right),
\end{align*}
where we recall that $\W_{p^*}$ is shorthand for $\W_{{p^*},(\R^d, |\cdot|)}$.
\item The $d \times d$ matrix $\sigma$ is non-degenerate. 
\item The initial states $(X^i_0)_{i=1}^\infty$ are i.i.d. with law $\mu_0 \in \P^{p'}(\R^d)$ for some $p' > 4$.
\item For each $n$, the $n$-player Nash system \eqref{def:Nashsystem} has a classical solution $(v^{n,i})_{i=1}^n$, in the sense that each function $v^{n,i}(t,\bm{x})$ is continuously differentiable in $t$ and twice continuously differentiable in $\bm{x}$. Moreover, $D_{x_j}v^{n,i}$ has linear growth and $v^{n,i}$ has quadratic growth, for each $n,i,j$. That is, there exist $L_{n,i} > 0$ and $L_{n,i,j} > 0$ such that, for all $t \in [0,T]$ and $\bm{x} \in (\R^d)^n$,
\begin{align*}
|D_{x_j}v^{n,i}(t,\bm{x})| &\le L_{n,i,j}\left(1 + |\bm{x}|\right), \\
|v^{n,i}(t,\bm{x})| &\le L_{n,i}\left(1 + |\bm{x}|^2\right),
\end{align*}
where $|\bm{x}|$ in \ref{assumption:A}(4) is the Euclidean norm of $\bm{x} \in (\RR^d)^n \cong \R^{dn}$.

\item The master equation admits a classical solution $U: [0,T] \times \RR^d \times \P^{p^*}(\RR^d) \ni (t,x,m) 
\mapsto U(t,x,m)$. The derivative $D_xU(t,x,m)$ exists and is Lipschitz in $(x,m)$, uniformly in $t$
{(using $\cW_{p^*}$ for the argument $m \in {\mathcal P}^{p^*}(\RR^d)$)}, and $U$ admits continuous derivatives $\partial_tU$, $D_xU$, $D_mU$, $D_x^2U$, $D_vD_mU$, $D_xD_mU$, and $D_{m}^2U$. Moreover, $D_xU$, $D_mU$, $D_xD_mU$, and $D_m^2U$ are assumed to be bounded.
\end{enumerate}
\end{assumption}
Assumptions (A.4-5) are heavy, owing to the difficulty of solving the master equation and $n$-player systems.
As for \ref{assumption:A}(4), earlier results on the solvability of systems of quasilinear parabolic equations may be found in \cite{ladyzhenskaia1988linear}. Those results suffice to guarantee the solvability of the Nash system when the Hamiltonian $H$ is globally Lipschitz,
see for instance 
\cite[Section 2.4.4]{cardaliaguet-delarue-lasry-lions}, 
 but a modicum of care is needed when the Hamiltonian is quadratic in the variable $y$. We refer to 
\cite{bensoussan1983nonlinear}
and to the forthcoming book \cite{CarmonaDelarue_book_II} (see Section 6.3.1 of Chapter 6 therein) for precise solvability results that fit our framework. 
Regarding the master equation,  the main examples of solvability  
we have in mind may be found in \cite[Theorems 2.8 and 2.9]{cardaliaguet-delarue-lasry-lions}. Results therein provide tractable assumptions on the data $(b,f,g)$ that ensure the existence and uniqueness of a classical solution to the master equation satisfying \ref{assumption:A}(5). Yet it must be stressed that the results of \cite{cardaliaguet-delarue-lasry-lions} are stated on the torus and under the rather demanding assumption that the Hamiltonian $H$ is globally Lipschitz. Again, we refer to the forthcoming book \cite{CarmonaDelarue_book_II} (specifically Section 5.4 of Chapter 5 therein) for an extension to the Euclidean setting that includes quadratic Hamiltonians.
We also refer to \cite[Theorems 5.3 and 5.5]{chassagneux2014probabilistic} for another result on existence and uniqueness for the master equation, although the boundedness requirements of the derivatives, as stated in \ref{assumption:A}(5), are not entirely verified there. 
Lastly, observe that, 
for some of the arguments below, 
the assumptions on the second order derivatives of $U$ in the direction of the measure can be relaxed
when $\sigma_{0} = 0$, i.e., when there is no common noise. 
This fact is made clear in \cite{cardaliaguet-delarue-lasry-lions}: Therein, the assumptions that are needed to prove the 
solvability of the master equation 
 are slightly weaker in the absence of a common noise (compare Theorems 2.8 and 2.9 in \cite{cardaliaguet-delarue-lasry-lions}). We refrain from distinguishing the two cases $\sigma_{0}=0$
and $\sigma_{0} \not =0$ in our paper, since 
the second-order derivatives of $U$ with respect to $m$ will be needed to establish the CLT, even in the case 
$\sigma_{0}=0$, see
Assumption 
\ref{assumption:C} in 
Section \ref{subse:add:assumtions:clt}.

As indicated above, our analysis heavily depends on properties of the Hamiltonian $H$. We thus need an additional assumption to control the growth of the function $\widehat{f}$, which is defined in \eqref{hatfns}, using of course the same function $\widehat{\alpha}$ from Assumption \ref{assumption:A}(1). We provide two alternatives in Assumptions \ref{assumption:B} and \ref{assumption:B'}. Assumption \ref{assumption:B} is more in the spirit of 
\cite{cardaliaguet-delarue-lasry-lions}, while Assumption \ref{assumption:B'} fits the framework addressed in \cite[Sections 5.4 and 6.3]{CarmonaDelarue_book_II}.

\begin{assumption}{\textbf{B}} \label{assumption:B} $\ $
$\widehat{f}(x,m,y)$ is Lipschitz in $y$, uniformly in $(x,m)$. That is, there exists $C > 0$ such that, for all $x,y,y' \in \R^d$ and $m \in \P^{{p^*}}(\R^d)$, 
\begin{align*}
|\widehat{f}(x,m,y) - \widehat{f}(x,m,y')| \le C|y-y'|.
\end{align*}
\end{assumption}

We will make use of one more set of assumptions, which weakens the Lipschitz assumption on $\widehat{f}$ in favor of uniform boundedness of solutions to the $n$-player Nash systems.
Note that boundedness of $D_xU$ and $D_mU$ in Assumption \ref{assumption:A}(5) ensures that $U$ has linear growth, by Lemma \ref{le:lineargrowth}. On the other hand, the following assumption will require $U$ to be bounded.

\begin{assumption}{\textbf{B'}} \label{assumption:B'} $\ $

\begin{enumerate}
\item The solution $U$ to the master equation is uniformly bounded.
\item The Nash system solutions 
$(v^{n,i})_{i=1}^{n}$ are bounded, uniformly in $n$ and $i$.
\item $\widehat{f}(x,m,y)$ is locally Lipschitz in $y$ with quadratic growth, uniformly in $(x,m)$. That is, there exists $C > 0$ such that, for all $x,y,y' \in \R^d$ and $m \in \P^{{p^*}}(\R^d)$, 
\begin{align*}
|\widehat{f}(x,m,y) - \widehat{f}(x,m,y')| \le C(1 + |y| + |y'|)|y-y'|.
\end{align*}
\end{enumerate}
\end{assumption}

\begin{remark}
It is typically difficult to find estimates on the Nash system solutions $v^{n,i}$ which are uniform in $n$, and this is a key challenge of MFG analysis. However, assumption \ref{assumption:B'}(2) can be directly verified in many situations. For instance, if $\widehat{f}$ is of the form $\widehat{f}(x,m,y) = f_1(x,m) + f_2(x,m,y)\cdot y$, where $f_2$ is of linear growth in $y$ uniformly in $(x,m)$, and if
{$f_{1}$ and $g$ are uniformly bounded}, then \ref{assumption:B'}(2) holds. {See \cite[Chapter 6, Section 6.3.1]{CarmonaDelarue_book_II} for details.}
\end{remark}

It may appear that we have not imposed any assumptions directly on the terminal cost function $g$. In fact, there are implicit requirements coming from Assumption \ref{assumption:A}(5) along with the boundary condition $U(T,x,m)=g(x,m)$.

\section{Statements of main results} \label{se:statements}

This section summarizes the main results of the paper on the $n$-player equilibrium empirical measures 
$(m^n_{\bm{X}})_{n \geq 1}$ and on their
marginal flows 
$((m^n_{\bm{X_{t}}})_{t \in [0,T]})_{n \geq 1}$, defined by the SDE \eqref{def:Nash-SDEsystem}. Proofs are deferred to later sections. 

\subsection{Law of large numbers}
\label{se:LLN-statements}

We first state the law of large numbers, regarding the convergence of $(m^n_{\bm{X}})_{n \geq 1}$ to $\mu$, where $\mu$ is defined by the McKean-Vlasov equation \eqref{def:MKV-conditional}. This is essentially the result of \cite{cardaliaguet-delarue-lasry-lions}, see also
\cite[Chapter 6]{CarmonaDelarue_book_II},
though we
formulate the result in Euclidean space instead of on the torus; we include the proof in Section \ref{se:LLN}.

\begin{theorem} \label{th:LLN}
Suppose Assumption \ref{assumption:A} holds, as well as either \ref{assumption:B} or \ref{assumption:B'}. Then 
\[
\lim_{n\rightarrow\infty}\E\left[\W_{2,\C^d}^2(m^n_{\bm{X}},\mu)\right] = 0.
\]
\end{theorem}

There are several ways to identify the rate of convergence in Theorem \ref{th:LLN}, and the interested reader is referred to the companion paper \cite{dellacram18b} for a detailed discussion.
Notably, the dimension $d$ appears in the rate, and we explain below the role of the dimension in the central limit theorem.

\subsection{Diffusively scaled fluctuations} 
\label{se:CLT-statements}
The central limit theorem is stated as a limit theorem for the sequence of rescaled differences 
$((S^n_{t} := \sqrt{n}(m^n_{\boldsymbol X_{t}} - \mu_{t}))_{t \in [0,T]})_{n \geq 1}$
between the  empirical measures and the 
solution to the MFG constructed above. 

To state a central limit theorem, we 
will regard each $(S^n_{t})_{t \in [0,T]}$ as a process with values in a distributional space. 
To do so, we use the same spaces as in \cite{meleard} or \cite{jourdain-meleard}. For an integer $j \geq 0$, a real $\alpha >0$,
and a smooth function $g$ on $\R^d$ with compact support, we let:
\begin{equation}
\label{eq:cH:j,alpha}
\| g \|_{j,\alpha}^2
:= \sum_{\vert {\boldsymbol k} \vert \leq j} \int_{\R^d} \frac{\vert D^{\boldsymbol k} g(x) \vert^2}{1+ \vert x \vert^{2 \alpha}} dx,
\end{equation}
where ${\boldsymbol k}$ denotes a multi-index ${\boldsymbol k}=(k_{1},\cdots,k_{d})$ and $\vert {\boldsymbol k} \vert 
:= k_{1} + \cdots + k_{d}$ is its length. We then call ${\mathcal H}^{j,\alpha}$ the completion of the set of smooth functions on $\R^d$ with compact support with respect to $\| \cdot \|_{j,\alpha}$. It is a Hilbert space when equipped with $\| \cdot \|_{j,\alpha}$. We denote by ${\mathcal H}^{-j,\alpha}$ its dual space. Below, we see $((S^n_{t})_{t \in [0,T]})_{n \geq 1}$ as a sequence with values in $\cH^{-(2+2\newD),\newD}$ for 
\[
\newD := \lfloor d/2 \rfloor +1.
\]

The limiting law of the sequence $((S^n_{t})_{t \in [0,T]})_{n \geq 1}$ will be identified as the solution to a properly stated SPDE. In order to formulate this SPDE, we define, for any $m \in \P^{p^*}(\R^d)$, the operator $\A_{t,m}$ by:
\begin{equation}
\label{eq:cA:clt}
\begin{split}
\bigl[ \A_{t,m}\phi \bigr](x) &:=
D_x\phi(x) \cdot \widehat{b}(x,m,D_xU(t,x,m))   + \frac12\mathrm{Tr}[D^2_x\phi(x)(\sigma\sigma^\top+\sigma_0\sigma_0^\top)] \\
	&\hspace{15pt} + \int_{\R^d}D_x\phi(v)\cdot	\frac{\delta}{\delta m}\left[\widehat{b}(v,m,D_xU(t,v,m))\right](x)\,m(dv),	\quad x \in \R^d,
\end{split}
\end{equation}	
for $\phi : \R^d \rightarrow \R$. 
Observe if needed that the derivative with respect to $m$ in the second line may be rewritten as:
\begin{align*}
&\frac{\delta}{\delta m}\left[\widehat{b}(v,m,D_xU(t,v,m))\right](x)	
\\
&=\frac{\delta}{\delta m}\widehat{b}(v,m,D_xU(t,v,m),x) + D_y\widehat{b}(v,m,D_xU(t,v,m))\frac{\delta}{\delta m}D_xU(t,v,m,x),
\end{align*}
where $D_y$ acts on the third argument of $\widehat{b}$.
We will see in  
\eqref{eq:embedding:1} that, for $\phi \in \cH^{4+2\newD,\newD}$, $D_x \phi(x)$ and $D^2_{x} \phi(x)$ exist pointwise and are bounded by $C \| \phi \|_{4+2\newD,\newD}(1+ \vert x \vert^\newD)$ for a universal constant $C$; in particular, 
$[ \A_{t,m}\phi](x)$ exist pointwise. If, as functions of $(t,x,m)$,
$\widehat{b}(x,m,D_x  U(t,x,m))$ 
and $[\delta / \delta m][\widehat{b}(y,m,D_x U(t,y,m))](x)$ are sufficiently smooth and $m \in \cP^{\newD}(\R^d)$, then 
${\mathcal A}_{t,m} \phi \in \cH^{2+2\newD,\newD}$.

The limiting SPDE is driven by three inputs: an initial condition and two noises. To state it properly, we assume that the probability space used to construct the mean field game is rich enough to carry a triple $(W,\theta_{0},\xi)$, where $W$ is the same common noise as before, $\theta_{0}$
is independent of $(W,\xi)$ and is a centered Gaussian random variable with values in $\cH^{-(2+2\newD),\newD}$ with covariance:
\begin{equation*}
\forall \phi_{1},\phi_{2} \in \cH^{2+2\newD,\newD},
\quad \E \bigl[ \theta_{0}(\phi_{1}) \theta_{0}(\phi_{2}) \bigr] = \langle \mu_{0},\phi_{1} \phi_{2} \rangle,
\end{equation*}
and where 
$\xi$ is, conditional on $W$, a continuous centered Gaussian process with values in $\cH^{-(2+2\newD),\newD}$ with 
covariance:
\begin{equation*}
\begin{split}
\forall \phi_{1},\phi_{2} \in \cH^{2+2\newD,\newD},
\ &\forall s,t \in [0,T], \quad
{\mathbb E} \bigl[ \xi_{t}(\phi_{1}) \xi_{s}(\phi_{2}) \, \vert \, W \bigr] 
= \int_{0}^{s \wedge t}
\bigl\langle \mu_{r}, \sigma \sigma^\top D_x \phi_{1} \cdot D_x\phi_{2} \bigr\rangle dr.
\end{split}
\end{equation*}
(Recall that $\mu$ is $W$-measurable.)
We will see in Section \ref{se:CLT-proofs} that both covariances are well-defined provided $p' \geq 2\newD$, with $p'$ as in Assumption {\bf A}. In both of them, 
$\langle \cdot,\cdot \rangle$ denotes the duality product.

We will see in Lemma 
\ref{lem:density} that, with probability 1, for all $t \in (0,T]$, $\mu_{t}$ has a density. As a consequence, $(\xi_{t})_{t \in [0,T]}$ may be represented 
as $(\xi_{t}(\phi) = \int_{0}
^t \sqrt{\mu_{s}} (\sigma^{\top} D_x\phi) \cdot d\beta_{s})_{0 \le t \le T}$, where
$\beta=(\beta^{1},\ldots,\beta^{d_{0}})$ is
a $d_{0}$-tuple of independent cylindrical Wiener processes with values in $L^2(\R^d)$, $\beta$ being independent of $(\theta_{0},W)$.
Since $\mu$ is $W$-measurable, 
$(W,\mu)$, 
$\theta_{0}$ and $\beta$ are independent.

We state the SPDE
for the weak limit $(S_{t})_{t \in [0,T]}$
of $((S_{t}^n)_{t \in [0,T]})_{n \geq 1}$ as an equation in $\cH^{-(4+2\newD),\newD}$:
\begin{equation}
\label{eq:spde:clt:intro}
\begin{split}
d \bigl\langle S_{t},\phi
\bigr\rangle 
&=     \bigl\langle
 S_{t}, {\mathcal A}_{t,\mu_{t}} \phi \bigr\rangle
 dt + \langle S_{t},{\sigma_{0}}^{\top} D_x \phi \rangle \cdot dW_{t} + d \xi_{t}(\phi),
\quad t \in [0,T], \quad {\phi \in \cH^{4+2\newD,\newD}},
\end{split}
\end{equation} 
where $\langle S_{t},{\sigma_{0}}^{\top} 
D_x  \phi \rangle \cdot dW_{t}$ is understood as 
$\sum_{i=1}^{d} \langle S_{t}, ({\sigma_{0}}^{\top} D_x \phi)_{i} \rangle \cdot 
dW_{t}^i$. 

As we will see in the example treated in Section \ref{se:examples}, an important feature of this SPDE is the fact that the operator 
${\mathcal A}_{t,\mu_{t}}$ is nonlocal. To clarify the challenges this presents, suppose for the moment that there is no common noise, $\sigma_{0}=0$. Then, if ${\mathcal A}_{t,\mu_{t}}$ were local, 
we could regard \eqref{eq:spde:clt:intro} as a (standard) parabolic partial differential equation
driven by the adjoint ${\mathcal A}_{t,\mu_{t}}^*$
and forced by an additive noise; then, we could  
use existing results on the fundamental solution of ${\mathcal A}_{t,\mu_{t}}^*$ to represent the solution.
In a sense, the same is true when $\sigma_{0}$ is non zero, because it suffices to shift $\phi$ into $\phi(\cdot - \sigma_{0} W_{t})$ to recover the case without common noise. 

Of course, the nonlocal term in \eqref{eq:cA:clt} arises because of the McKean-Vlasov interaction, and we must be prepared to 
face it in practice.
In the example in  
Section \ref{se:examples}, the
coefficients are linear-quadratic,
which makes  
the form of the interaction not too complicated. Ultimately, we 
manage to derive an explicit expression for the solution to 
\eqref{eq:spde:clt:intro}, by solving first for the interaction term and then constructing the entire solution by proceeding as if the structure were local.

\subsection{Additional assumptions and statement of the central limit theorem} 
\label{subse:add:assumtions:clt}
In order to state our CLT, we need the following additional set of assumptions: 
{\begin{assumption}{\textbf{C}} \label{assumption:C} $\ $
\begin{enumerate}
\item $\widehat b$ is locally bounded in $y$, uniformly in $(x,m)$.
\item $\widehat b$
and $D_x  U$ 
have derivatives up to the order $5+2\newD$
with respect to $(x,y)$ and $x$ respectively, and these derivatives are jointly continuous and uniformly bounded with respect to all arguments.
\item $\delta \widehat{b}/\delta m$ and 
$\delta (D_x  U)/\delta m$ exist, are bounded
and jointly continuous in all the arguments, and have derivatives in the variable $v$ up to the order $4+2\newD$ that are jointly continuous and uniformly bounded with respect to all arguments.
\item $\delta(D_{y} \widehat{b})/\delta m$ 
exists
and is jointly continuous and uniformly bounded in all the arguments.
$\delta (D_x  U)/\delta m$,
 $D_{y} \widehat{b}$ and 
 $\delta(D_{y} \widehat{b})/\delta m$ are Lipschitz continuous in $m$
with respect to $\W_{1}$ 
(understood here as $\W_{1,\R^d}$), uniformly in the other arguments. 
\item $\delta^2 \widehat{b}/\delta m^2$
and $\delta^2 (D_x  U)/\delta m^2$
exist, are bounded and are Lipschitz continuous in
$m$ with respect to $\W_{1}$, uniformly 
in the other arguments. 
\end{enumerate}
\end{assumption}}

\begin{theorem} 
\label{th:CLT:sec:3}
Let Assumptions \ref{assumption:A} and \ref{assumption:C} hold,
 as well as either Assumption 
\ref{assumption:B} or \ref{assumption:B'}, with with 
$p' > 12\newD$ and $p^*=1$.  
 Then, the sequence $({S}^n)_{n \geq 1}$ converges in law on the space $ C([0,T];\cH^{-(2+2\newD),\newD})$ to the unique solution to the SPDE \eqref{eq:spde:clt:intro} with $S_{0}=\theta_{0}$ as initial condition.  
\end{theorem}
The rationale for requiring the coefficients to be so regular may be explained as follows. In the proof of the CLT, we need to regard the coefficients (or their derivatives, depending on the step of the proof) as test functions, namely as elements of the space $\cH^{4+2\newD,\newD}$.
In this respect, the fact that the dimension $d$ of the state space enters the definition of the space of test functions
through the parameter $\newD$ should not come as a surprise. Indeed, the Wasserstein distance between a probability distribution and the empirical measure formed by $n$ i.i.d.\ samples is less and less accurate as the dimension of the space increases (see, e.g., \cite{fournier-guillin} or the refinement of Theorem \ref{th:LLN} above in our companion paper \cite[Theorem 3.1]{dellacram18b}). 
Still, we prove in Lemma 
\ref{lem:lln:l4:appendix} below the rather striking fact (which seems to be a new point) that the rate of convergence 
of the empirical distribution of a sample of i.i.d. random vectors with values $\RR^d$
can be made  
 dimension-free, provided the distance is measured using smooth enough test functions. 
Noting, for example, that the $1$-Wasserstein and total variation distances between two measures $m$ and $m'$ may each be written as a supremum of $\int\phi\,d(m-m')$ over a suitable family of test functions $\phi$, it is clear that the strength of a probability metric and thus the rate of convergence of empirical measures therein depends heavily on the regularity of the test functions. Hence, to preserve the $O(\sqrt{n})$ fluctuations, we must be prepared to refine our family of test functions. The two key difficulties are identifying the right level of smoothness and dealing with the supremum appearing in a metric of the aforementioned form. Invoking suitable forms of Sobolev embeddings seems to be the most efficient strategy for addressing both difficulties. As is  well known, Sobolev embeddings heavily depend on the dimension and this is precisely what dictates the choice of the space of test functions. 

Lastly, observe that requiring $p^*=1$ in 
Assumptions \ref{assumption:A}, \ref{assumption:B} or 
\ref{assumption:B'} is somewhat redundant with the contents of Assumption
\ref{assumption:C}. Indeed, the fact that the derivatives $\delta \hat{b}/\delta m$
and $\delta (D_{x} U)/\delta m$ have bounded derivatives in $v$ implies that
both $\hat{b}$ and $D_{x} U$ are Lipschitz continuous in $m$ with respect to ${\mathcal W}_{1}$.

As for the proof, we will follow the strategy put forth in the Introduction. Namely, we will first prove 
\eqref{intro:O(1/n^2)}
in Section 
\ref{se:mainestimates}
below. Then, extending earlier works on the same subject, we will prove the CLT for the McKean-Vlasov system \eqref{intro:particles} in Section \ref{se:CLT-proofs}. We will conclude by combining the two of them.

\section{Main estimates} \label{se:mainestimates}

All of the results {proved in this paper and in the companion one \cite{dellacram18b}} rely on the crucial estimates developed in this section.  
In the following results and proofs, $U$ is the classical solution to the master equation \eqref{def:masterequation}. The letter $C$ denotes a generic positive constant, which may change from line to line but is universal in the sense that it never depends on $i$ or $n$, though it may of course depend on model parameters, including, e.g., the bounds on {the growth and the regularity of $U$ and its derivatives}, the Lipschitz constants of $\widehat{b}$ and $\widehat{f}$, and the time horizon $T$.

\subsection{A preliminary estimate}
\label{subse:4:premiliminary}
For $(t,\bm{x}) \in [0,T] \times (\R^d)^n$, define 
\[
u^{n,i}(t,\bm{x}) = U(t,x_i,m^n_{\bm{x}}). 
\]
The following result, largely borrowed from Proposition 6.3 in \cite{cardaliaguet-delarue-lasry-lions}, shows that the functions $(u^{n,i})_{i=1}^n$ \emph{nearly} solve the $n$-player Nash system defined in Section \ref{se:Nashsystems}. We provide the proof in Appendix \ref{ap:derivatives}.

\begin{proposition} \label{pr:master-nearly-nash}
Suppose Assumption \ref{assumption:A} holds as well as either Assumption \ref{assumption:B} or \ref{assumption:B'}.
There exist a constant $C < \infty$ and, for each $n \in \N$,  continuous functions $r^{n,i} : [0,T] \times (\R^d)^n \rightarrow \R$, for $1 \le i \le n$, with 
\[ 
\|r^{n,i}\|_\infty \le C/n,
\]
such that
\begin{align*}
\partial_tu^{n,i}(t,\bm{x}) &+ H\left(x_i,m^n_{\bm{x}},D_{x_i}u^{n,i}(t,\bm{x})\right) + \sum_{j =1, j \neq i}^n D_{x_j}u^{n,i}(t,\bm{x}) \cdot \widehat{b}\left(x_j,m^n_{\bm{x}},D_{x_j}u^{n,j}(t,\bm{x})\right) \\
&+ \frac{1}{2}\sum_{j = 1}^n\mathrm{Tr}\left[\sigma \sigma^\top D^2_{x_j,x_j}u^{n,i}(t,\bm{x}) \right] + \frac{1}{2}\sum_{j,k=1}^n\mathrm{Tr}\left[\sigma_0 \sigma_0^\top  D^2_{x_j,x_k}u^{n,i}(t,\bm{x})\right] = -r^{n,i}(t,\bm{x}),
\end{align*}
for $(t,\bm{x}) \in (0,T) \times (\R^d)^n$. 
\end{proposition}

To proceed, we define an $n$-particle SDE system of McKean-Vlasov type, which we will compare to the true Nash system. Precisely,  let $\bm{\overline{X}} = (\overline{X}^1,\ldots,\overline{X}^n)$ solve the {approximating $n$-particle} system
\begin{align}
\label{eq:sec4:overlineX}
d\overline{X}^i_t &= \widehat{b}
\bigl(\overline{X}^i_t,m^n_{\bm{\overline{X}}_t},D_xU(t,\overline{X}^i_t,m^n_{\bm{\overline{X}}_t})
\bigr)dt + \sigma dB^i_t + \sigma_0 dW_t, \quad \overline{X}^i_0=X^i_0.
\end{align}
Because of Assumptions \ref{assumption:A}(1) and \ref{assumption:A}(5), this SDE system admits a unique strong solution.
We make the following convenient abbreviations: for $i, j = 1, \ldots, n$, 
\begin{align}
\label{Y-Z}
Y_{t}^i = v^{n,i}(t,\bm{X}_t), \quad\quad &Z_{t}^{i,j} = D_{x_{j}} v^{n,i}(t,\bm{X}_t), \\
\label{slantY-Z}
\Y_{t}^i = U(t,X_{t}^i,m^n_{\bm{X}_t}) = u^{n,i}(t,\bm{X}_t), \quad\quad &\Z_{t}^{i,j}= D_{x_j}u^{n,i}(t,\bm{X}_t).
\end{align}
Recalling the definition of $u^{n,i}$, we can apply  Proposition \ref{pr:empiricalmeasure} to get
\begin{align}
\Z^{i,j}_t = D_{x_j}[U(t,X^i_t,m^n_{\bm{X}_t})] = \begin{cases}
\frac{1}{n}D_mU(t,X^i_t,m^n_{\bm{X}_t},X^j_t) &\text{if } j \neq i, \\
D_xU(t,X^i_t,m^n_{\bm{X}_t}) + \frac{1}{n}D_mU(t,X^i_t,m^n_{\bm{X}_t},X^i_t) &\text{if } j = i.
\end{cases} \label{def:Zexpression}
\end{align}
In particular, recalling that $D_xU$ and $D_mU$ are bounded, we have the crucial bounds 
\begin{align}
|\Z^{i,j}_t| &\le \frac{C}{n},  \text{ for } j \neq i, \quad\quad\quad |\Z^{i,i}_t| \le C. \label{pf:mainestimate-zbounds}
\end{align}
Also, in what follows, 
for $i=1,\ldots,n$, define: 
\begin{align}
\label{def-M}
M^i_t &= \int_0^t\left(\sum_{j=1}^n \left(Z^{i,j}_s - \Z^{i,j}_s\right)  \cdot \sigma dB^j_s + \sum_{j=1}^n (Z^{i,j}_s-\Z^{i,j}_s) \cdot  \sigma_0 dW_s\right), \\
\label{def-N}
N^i_t &= \int_0^t (Y_{s}^i - \Y_{s}^i)  dM^i_{s}.
\end{align}

\subsection{Main estimate}
\label{subse:4:main}

We state and prove our main estimate first under assumptions \ref{assumption:A} and \ref{assumption:B}, and then under assumptions \ref{assumption:A} and \ref{assumption:B'}.

\begin{theorem} \label{th:mainestimate}
Suppose Assumptions \ref{assumption:A} and \ref{assumption:B} hold.
Then, there exists $C <  \infty$ such that, for each $n$,
\begin{align}
\frac{1}{n}\sum_{i=1}^n\E\left[\int_0^T\left|D_{x_i}v^{n,i}(t,\bm{X}_t) - D_xU(t,X^i_t,m^n_{\bm{X}_t})\right|^2dt\right] &\le \frac{C}{n^2}, \label{def:estimate-Dx} \\
\E\left[\frac{1}{n}\sum_{i=1}^n\|X^i - \overline{X}^i\|_\infty^2\right] &\le \frac{C}{n^2}. \label{def:estimate-|X-Xbar|}
\end{align}
Moreover, it holds almost surely that
\begin{align}
&{\frac{1}{n}\sum_{i=1}^n \|X^i - \overline{X}^i \|_{\infty}^2 \le \frac{C}{n}\sum_{i=1}^n\int_0^T\left|D_{x_i}v^{n,i}(t,\bm{X}_t) - D_xU(t,X^i_t,m^n_{\bm{X}_t})\right|^2dt}, \label{pf:mainestimates1}
\\
&{\frac{1}{n}\sum_{i=1}^n \int_0^T\left|D_{x_i}v^{n,i}(t,\bm{X}_t) - D_xU(t,X^i_t,m^n_{\bm{X}_t})\right|^2dt \le \frac{C}{n}\sum_{i=1}^n \int_0^T|Z^{i,i}_t-\Z^{i,i}_t|^2dt + \frac{C}{n^2}},
\label{eq:Dv:Du:Z:cZ}
\end{align}
{and for all $t \in [0,T]$,}
\begin{align}
&\frac{1}{n}\sum_{i=1}^n[N^i]_t \le \frac{C}{n^3}\sum_{i=1}^n[M^i]_t, \quad \text{ and } \quad \frac{1}{n}\sum_{i=1}^n[M^i]_T \le \frac{C}{n^2} + \frac{C}{n}\sum_{i=1}^n|N^i_T|. \label{def:estimates-N,M}
\end{align}
\end{theorem}

\begin{remark}
\label{rem:use:estimates}
In this paper, we make no use of \eqref{def:estimates-N,M}, which is used only in the companion paper \cite{dellacram18b}. Since the proofs follow from the same computations, we find  it more convenient to prove here all estimates at once. This permits us not to repeat the arguments in \cite{dellacram18b} and to instead  directly invoke Theorem \ref{th:mainestimate}.
\end{remark}

\begin{remark}
As we will see in Section \ref{se:LLN} below, the law of large numbers of Theorem \ref{th:LLN} follows quickly from Theorem \ref{th:mainestimate}.  
Interestingly, these theorems together imply uniqueness for the master equation, within the class described in Assumption \ref{assumption:A}, and we sketch the argument here. Let $\widetilde{U}$ denote another smooth solution, and let $(\widetilde{\X},\widetilde{\mu})$ solve the corresponding McKean-Vlasov equation 
 \eqref{def:MKV-conditional}, driven by the same Wiener processes $B^1$ and $W$, but with $D_xU$ replaced by $D_x\widetilde{U}$. Then Theorem \ref{th:LLN} shows that $(m^n_{\bm{X}_t})_{t \in [0,T]}$ converges in probability in $C([0,T];\P^{p^*}(\R^d))$ to both $(\mu_t)_{t \in [0,T]}$ and $(\widetilde{\mu}_t)_{t \in [0,T]}$, so $\mu = \widetilde{\mu}$ and $ \X=\widetilde{\X}$ a.s.

Combining 
\eqref{def:estimate-Dx} and \eqref{def:estimate-|X-Xbar|} and using in addition the smoothness of 
$D_x  U$ and $D_x  \widetilde{U}$, we get
\begin{align*}
0 &= \lim_{n \rightarrow \infty}
\frac1n \sum_{i=1}^n\E \int_{0}^T \bigl\vert D_x  U(t,\overline X_{t}^i,m_{\overline{\boldsymbol X}_{t}}^n) -
D_x  \widetilde{U}(t,\overline X_{t}^i,m_{\overline{\boldsymbol X}_{t}}^n) 
\bigr\vert^2 dt \\
	&= \EE \int_{0}^T \bigl\vert D_x  U(t,\X_{t},\mu_{t}) -
D_x  \widetilde{U}(t,\X_{t},\mu_{t}) 
\bigr\vert^2 dt.
\end{align*}
By expanding 
$(U(t,\X_{t},\mu_{t}))_{t \in [0,T]}$
and 
$(\widetilde U(t,\X_{t},\mu_{t}))_{t \in [0,T]}$
using It\^o's formula for functions defined on $\cP^2(\RR^d)$ (see for instance \cite{chassagneux2014probabilistic} for the case $\sigma_{0}=0$ and 
\cite[Chapter 4, Section 4.3]{CarmonaDelarue_book_II} for the general case), and by invoking the fact that both $U$ and $\widetilde{U}$ solve the master equation, we finally obtain that, for any $t \in [0,T]$, $U(t,\X_{t},\mu_{t})$ and $\widetilde{U}(t,\X_{t},\mu_{t})$ are almost surely equal. 
Choosing arbitrarily the initial condition of the game ({i.e., choosing arbitrarily the initial time and the initial distribution}), we deduce that $U(t,\cdot,m)$ and $\widetilde U(t,\cdot,m)$ are equal for any $t \in [0,T]$ and any $m$ with full support. 
By a standard mollification argument, the same holds true for any $m$, whether the support of $m$ is full or not.
\end{remark}

\subsection*{Proof of Theorem \ref{th:mainestimate}}
We begin by proving that \eqref{def:estimate-|X-Xbar|} follows from \eqref{def:estimate-Dx}. 
Abbreviate 
\begin{equation}
\label{widehat:widetilde:b:n:i}
\begin{split}
\widehat{b}^{n,i}(t,\bm{x}) &= \widehat{b}\bigl(x_i,m^n_{\bm{x}},D_{x_i}v^{n,i}(t,\bm{x})\bigr), 
\\
\widetilde{b}^{n,i}(t,\bm{x}) &= \widehat{b}\bigl(x_i,m^n_{\bm{x}},D_xU(t,x_i,m^n_{\bm{x}})\bigr), 
\end{split}
\end{equation}
for $(t,\bm{x}) \in [0,T] \times (\R^d)^n$. 
Note first that \eqref{def:Nash-SDEsystem}, \eqref{eq:sec4:overlineX} and the Lipschitz assumption on $\widehat{b}$ from Assumption \ref{assumption:A}(1) yields:  
\begin{align*}
&|X^i_t-\overline{X}^i_t|^2 
\\
&\le T\int_0^t\left|\widehat{b}^{n,i}(s,\bm{X}_s) - \widetilde{b}^{n,i}(s,\bm{\overline{X}}_s)\right|^2ds \\
	&\le 2T\int_0^t\left|\widehat{b}^{n,i}(s,\bm{X}_s) - \widetilde{b}^{n,i}(s,\bm{X}_s)\right|^2ds + 2T\int_0^t\left|\widetilde{b}^{n,i}(s,\bm{X}_s) - \widetilde{b}^{n,i}(s,\bm{\overline{X}}_s)\right|^2ds \\
	&\le C\int_0^t\left|D_{x_i}v^{n,i}(s,\bm{X}_s) - D_xU(s,X^i_s,m^n_{\bm{X}_s})\right|^2ds + C\int_0^t\left(|X^i_s-\overline{X}^i_s|^2 + \W_2^2(m^n_{\bm{X}_s},m^n_{\bm{\overline{X}}_s})\right)ds.
\end{align*}
By Gronwall's inequality,
\begin{align}
|X^i_t-\overline{X}^i_t|^2 &\le C\int_0^t\left|D_{x_i}v^{n,i}(s,\bm{X}_s) - D_xU(s,X^i_s,m^n_{\bm{X}_s})\right|^2ds + C\int_0^t\W_2^2(m^n_{\bm{X}_s},m^n_{\bm{\overline{X}}_s})ds. \label{pf:gronwall1}
\end{align}
Noting that
\begin{align*}
\W_2^2(m^n_{\bm{X}_s},m^n_{\bm{\overline{X}}_s}) &\le \frac{1}{n}\sum_{i=1}^n|X^i_s-\overline{X}^i_s|^2,
\end{align*}
we may average \eqref{pf:gronwall1} over $i=1,\ldots,n$ and apply Gronwall's inequality once again to obtain 
\begin{align*}
\W_2^2(m^n_{\bm{X}_t},m^n_{\bm{\overline{X}}_t}) &\le \frac{C}{n}\sum_{i=1}^n\int_0^t\left|D_{x_i}v^{n,i}(s,\bm{X}_s) - D_xU(s,X^i_s,m^n_{\bm{X}_s})\right|^2ds.
\end{align*}
Using this bound in \eqref{pf:gronwall1} yields
\eqref{pf:mainestimates1}.
Also, \eqref{def:estimate-|X-Xbar|} will follow from \eqref{def:estimate-Dx}, which we now prove.

First, use It\^o's formula,  \eqref{exp-ham}, the fact that $v^{n,i}$ solves the Nash system, and \eqref{Y-Z} to obtain:
\begin{align}
dY^i_t = & \ \partial_tv^{n,i}(t,\bm{X}_t) + \sum_{j=1}^nD_{x_j}v^{n,i}(t,\bm{X}_t) \cdot \widehat{b}^{n,j}(t,\bm{X}_t)dt + \frac{1}{2}\sum_{j=1}^n\mathrm{Tr}\left[ D_{x_j,x_j}^2v^{n,i}(t,\bm{X}_t) \sigma\sigma^\top\right]dt \nonumber \\
	&+ \frac{1}{2}\sum_{j,k=1}^n\mathrm{Tr}\left[  D_{x_j,x_k}^2v^{n,i}(t,\bm{X}_t) \sigma_0\sigma_0^\top \right]dt + \sum_{j=1}^nZ^{i,j}_t	\cdot	\sigma dB^j_t + \sum_{j=1}^nZ^{i,j}_t	\cdot	\sigma_0 dW_t 	\nonumber
	\\
	= &- \widehat{f}\left(X^i_t,m^n_{\bm{X}_t},Z^{i,i}_t\right)dt + \sum_{j=1}^nZ^{i,j}_t\cdot \sigma dB^j_t + \sum_{j=1}^nZ^{i,j}_t \cdot \sigma_0 dW_t. \nonumber
\end{align}
On the other hand, use It\^{o}'s formula, \eqref{exp-ham},  Proposition \ref{pr:master-nearly-nash}, and \eqref{slantY-Z} to 
get:  
\begin{align*}
d\Y^i_t = \ &\partial_tu^{n,i}(t,\bm{X}_t) + \sum_{j=1}^nD_{x_j}u^{n,i}(t,\bm{X}_t) \cdot \widehat{b}\left(X^j_t,m^n_{\bm{X}_t},D_{x_j}v^{n,j}(t,X^j_t,m^n_{\bm{X}_t})\right)dt 
\\
	&+ \frac{1}{2}\sum_{j=1}^n\mathrm{Tr}\left[ D_{x_j,x_j}^2u^{n,i}(t,\bm{X}_t)\sigma\sigma^\top	\right]dt + \frac{1}{2}\sum_{j,k=1}^n\mathrm{Tr}\left[ D_{x_j,x_k}^2u^{n,i}(t,\bm{X}_t) \sigma_0\sigma_0^\top\right]dt \\
	&+ \sum_{j=1}^n\Z^{i,j}_t \cdot \sigma dB^j_t + \sum_{j=1}^n\Z^{i,j}_t \cdot\sigma_0 dW_t \\
		= \ &-\left[\widehat{f}\left(X^i_t,m^n_{\bm{X}_t},\Z^{i,i}_t\right) + r^{n,i}(t,\bm{X}_t)\right]dt + \sum_{j=1}^n\Z^{i,j}_t \cdot\sigma dB^j_t + \sum_{j=1}^n\Z^{i,j}_t \cdot \sigma_0 dW_t \\
		& + \sum_{j=1}^n\Z^{i,j}_t \cdot \left( \widehat{b}\left(X^j_t,m^n_{\bm{X}_t},Z^{j,j}_t\right) -  \widehat{b}\left(X^j_t,m^n_{\bm{X}_t},\Z^{j,j}_t\right)\right)dt.
\end{align*}  
Use It\^o's formula and the terminal conditions for the Nash system and the master equation to get, for $t \in [0,T]$,
\begin{align}
(Y^i_t-\Y^i_t)^2 = \ &2\int_t^T(Y^i_s-\Y^i_s)\left\{\vphantom{\sum_{j=1}^n}\widehat{f}\left(X^i_s,m^n_{\bm{X}_s},Z^{i,i}_s\right) - \widehat{f}\left(X^i_s,m^n_{\bm{X}_s},\Z^{i,i}_s\right) \right.
\nonumber
 \\
 &\left. +  \sum_{j=1}^n\Z^{i,j}_s \cdot \left( \widehat{b}\left(X^j_s,m^n_{\bm{X}_s},Z^{j,j}_s\right) -  \widehat{b}\left(X^j_s,m^n_{\bm{X}_s},\Z^{j,j}_s\right)\right) -  r^{n,i}(s,\bm{X}_s)\right\}ds 
 \nonumber
 \\
	&- \int_t^T\sum_{j=1}^n	\Bigl\vert \sigma^\top \bigl( Z^{i,j}_s-\Z^{i,j}_s \bigr) 
	\Bigr\vert^2 ds - \int_t^T\left|\sum_{j=1}^n
	\sigma_0^\top\bigl(Z^{i,j}_s-\Z^{i,j}_s\bigr)\right|^2ds 
	\nonumber
	\\
	&- 2\int_t^T(Y^i_s-\Y^i_s)\sum_{j=1}^n(Z^{i,j}_s-\Z^{i,j}_s)
	\cdot
	\left(\sigma dB^j_s + \sigma_0 dW_s\right). \label{eq:diff:Y:cY:f:lip:2}
\end{align}
Use the Lipschitz assumption on $\widehat{b}$ along with the estimates \eqref{pf:mainestimate-zbounds} to get:
\begin{align*}
\left|\sum_{j=1}^n\Z^{i,j}_s \cdot \left( \widehat{b}\left(X^j_s,m^n_{\bm{X}_s},Z^{j,j}_s\right) -  \widehat{b}\left(X^j_s,m^n_{\bm{X}_s},\Z^{j,j}_s)\right)\right)\right| & \le C\sum_{j=1}^n|\Z^{i,j}_s||Z^{j,j}_s -  \Z^{j,j}_s| \\
	& \le C|Z^{i,i}_s-\Z^{i,i}_s| + \frac{C}{n}\sum_{j=1,j \neq i}^n |Z^{j,j}_s -  \Z^{j,j}_s|.
\end{align*}
Substitute the last inequality back into \eqref{eq:diff:Y:cY:f:lip:2}, 
and use the Lipschitz assumption on $\widehat{f}$ from Assumption \ref{assumption:B'},  the elementary inequality $2ab \le \epsilon^{-1}a^2 + \epsilon b^2$,  and the estimate $|r^{n,i}| \le C/n$ from Proposition \ref{pr:master-nearly-nash} to obtain, for $\epsilon > 0$,  
\begin{align}
(Y^i_t-\Y^i_t)^2 \le \ & \frac{C}{n^2} + C\int_t^T\left(\epsilon^{-1}(Y^i_s-\Y^i_s)^2 + \epsilon|Z^{i,i}_s-\Z^{i,i}_s|^2 + \frac{\epsilon}{n}\sum_{j \neq i}|Z^{j,j}_s -  \Z^{j,j}_s|^2\right)ds \nonumber
\\
	&- \int_t^T\sum_{j=1}^n	\bigl\vert	\sigma^\top	\bigl(Z^{i,j}_s-\Z^{i,j}_s\bigr) \bigr\vert^2 ds - \int_t^T\left|\sum_{j=1}^n \sigma_{0}^\top	\bigl(Z^{i,j}_t-\Z^{i,j}_s\bigr) \right|^2ds	\nonumber	 \\
	&-  2 \int_t^T(Y^i_s-\Y^i_s)\sum_{j=1}^n(Z^{i,j}_s-\Z^{i,j}_s) \cdot \left(\sigma dB^j_s + \sigma_0 dW_s\right).
	\label{eq:diff:Y:cY:f:lip}
	\end{align}
Here $C$ does not depend on $\epsilon$. Noting that $\sigma$ is non-degenerate, we may choose $\epsilon$ small enough and average over $i=1,\ldots,n$ to get
\begin{align}
\frac{1}{n}&\sum_{i=1}^n\left((Y^i_t-\Y^i_t)^2 + \int_t^T\sum_{j=1}^n|Z^{i,j}_s-\Z^{i,j}_s|^2ds\right)  \label{pf:BSDEestimate}
\\
	&\le \frac{C}{n^2} + \frac{C}{n}\sum_{i=1}^n\int_t^T(Y^i_s-\Y^i_s)^2ds - \frac{\widetilde{C}}{n}\sum_{i=1}^n\int_t^T(Y^i_s-\Y^i_s)\sum_{j=1}^n(Z^{i,j}_s-\Z^{i,j}_s) \cdot \left(\sigma dB^j_t + \sigma_0 dW_t\right),
	\nonumber 
\end{align}
where the constant $\widetilde{C} < \infty$ depends only on $\sigma$.
It follows from Lemma \ref{le:YZfinite-moments} below that the stochastic integrals are martingales.
Take conditional expectations to get, for $r \le t$,
\begin{align*}
\frac{1}{n}\sum_{i=1}^n\E[(Y^i_t-\Y^i_t)^2|\F_r] &\le \frac{C}{n^2} + \frac{C}{n}\sum_{i=1}^n\int_t^T\E[(Y^i_s-\Y^i_s)^2|\F_r]dt.
\end{align*}
Using Gronwall's inequality and taking $r=t$ we conclude that:
\begin{align}
\frac{1}{n}\sum_{i=1}^n(Y^i_t-\Y^i_t)^2 &\le \frac{C}{n^2}.
\label{eq:y:cy:B:average}
\end{align}
Also, taking conditional 
expectations in \eqref{pf:BSDEestimate}, and using \eqref{eq:y:cy:B:average}, we find:
\begin{align}
\label{eq:diff:Y:cY:f:lip:3}
\frac{1}{n}\sum_{i=1}^n\sum_{j=1}^n\E\left[\int_t^T|Z^{i,j}_s-\Z^{i,j}_s|^2ds 
\big\vert 
\, \F_{t} \right] \le \frac{C}{n^2}.
\end{align}
Then, returning to 
\eqref{eq:diff:Y:cY:f:lip}, taking the conditional expectation given $\F_t$, and applying Gronwall's lemma yields
\begin{align}
\max_{i=1,\ldots,n} (Y^i_t-\Y^i_t)^2 &\le \frac{C}{n^2}.
\label{eq:y:cy:B}
\end{align} 
This implies that, for $M$ and $N$ defined in \eqref{def-M}-\eqref{def-N},  
\begin{align*}
\frac{1}{n}\sum_{i=1}^n[N^i]_T &= \frac{1}{n}\sum_{i=1}^n\int_0^T(Y^i_t-\Y^i_t)^2\,d[M^i]_t \le \frac{C}{n^3}\sum_{i=1}^n[M^i]_T,
\end{align*}
which proves the first part of \eqref{def:estimates-N,M}.  
On the other hand, applying  \eqref{eq:y:cy:B:average} to  \eqref{pf:BSDEestimate} shows that
\begin{align*}
\frac{1}{n}\sum_{i=1}^n(Y^i_t-\Y^i_t)^2 + \frac{1}{n}\sum_{i=1}^n([M^i]_T-[M^i]_t) &\le \frac{C}{n^2} + \frac{C}{n}\sum_{i=1}^n\int_t^T(Y^i_s-\Y^i_s)^2ds - \frac{\widetilde{C}}{n}\sum_{i=1}^n(N^i_T-N^i_t) \\
	&\le \frac{C}{n^2} - \frac{\widetilde{C}}{n}\sum_{i=1}^n(N^i_T-N^i_t).
\end{align*}
Take $t=0$ to derive the second part of \eqref{def:estimates-N,M}.

Finally,  using the definition \eqref{Y-Z} of $Z$ and recalling from \eqref{def:Zexpression} and \eqref{slantY-Z}  that $D_xU(t,X^i_t,m^n_{\bm{X}_t}) = \Z^{i,i}_t - \frac{1}{n}D_mU(t,X^i_t,m^n_{\bm{X}_t},X^i_t)$ and that $D_mU$ is bounded, we get:
\begin{align*}
\frac{1}{n}\sum_{i=1}^n \int_0^T\left|D_{x_i}v^{n,i}(t,\bm{X}_t) - D_xU(t,X^i_t,m^n_{\bm{X}_t})\right|^2dt
 &\le \frac{C}{n}\sum_{i=1}^n \int_0^T|Z^{i,i}_t-\Z^{i,i}_t|^2dt + \frac{C}{n^2},
\end{align*}
which is \eqref{eq:Dv:Du:Z:cZ}. Taking expectations and using 
\eqref{eq:diff:Y:cY:f:lip:3} with $t=0$, we get \eqref{def:estimate-Dx}.
\hfill\qedsymbol

{\ } \\

We close the section by proving standard integrability results that were used in the above proof to conclude that a certain local martingale was in fact a martingale.

\begin{lemma} \label{le:YZfinite-moments}
Suppose Assumption \ref{assumption:A} holds, as well as either Assumption \ref{assumption:B} or \ref{assumption:B'}. Then, for all $i,j$,
\begin{align*}
\E\left[\|Y^i\|_{{\infty}}^2 + \|Z^{i,j}\|_{{\infty}}^2\right] < \infty.
\end{align*}
\end{lemma}
\begin{proof}
By the SDE \eqref{intro:dynamics} for $X^i$, we have:
\begin{align*}
|X^i_t| &\le |X^i_0| + \int_0^t|D_{x_i}v^{n,i}(t,\bm{X}_s)|ds + |\sigma||B^i_t| + |\sigma_0||W_t| \\
	&\le |X^i_0| + L_{n,i,i}\left(T + \int_0^t|\bm{X}_s|ds\right) + |\sigma||B^i_t| + |\sigma_0||W_t|, 
\end{align*}
where $L_{n,i,i}$ is the constant from Assumption \ref{assumption:A}(4).
This holds for all $i=1,\ldots,n$, and Gronwall's inequality yields:
\begin{align*}
\|\bm{X}\|_{{\infty}} \le C\left(1 + \sum_{i=1}^n|X^i_0| + \sum_{i=1}^n\|B^i\|_{{\infty}} + \|W\|_{{\infty}}\right),
\end{align*}
where we note that the constant $C <  \infty$ may depend on $n$ here.
Recalling from Assumption \ref{assumption:A}(3) that $X^i_0$ are i.i.d.\ with finite moments of order $p' > 4$, we conclude that $\E[\|\bm{X}\|_{{\infty}}^4] < \infty$. Using the definitions of $Y^i$ and $Z^{i,j}$ from \eqref{Y-Z} along with assumption \ref{assumption:A}(4), we find
\begin{align*}
|Y^i_t|^2 = |v^{n,i}(t,\bm{X}_t)|^2 \le 2L_{n,i}^2 (1+|\bm{X}_t|^4 ), \quad\quad\quad |Z^{i,j}_t|^2 \le2 L_{n,i,j}^2 (1+|\bm{X}_t|^2 ).
\end{align*}
\end{proof}

\subsection{An alternative estimate}

Under Assumption \ref{assumption:B'}, instead of Assumption \ref{assumption:B}, we have somewhat different estimates, based on an exponential transformation.

\begin{theorem} \label{th:mainestimate2}
Suppose Assumptions \ref{assumption:A} and \ref{assumption:B'} hold.
Then, \eqref{pf:mainestimates1} and \eqref{eq:Dv:Du:Z:cZ} hold, and, for sufficiently large $n$, the estimates \eqref{def:estimate-Dx} and \eqref{def:estimate-|X-Xbar|} hold.
For $i=1,\ldots,n$ and a constant $\eta > 0$, define $M^i$ as in
\eqref{def-M}  and $Q^i$ by 
\begin{align*}
Q^i_t &= \int_0^t\left[2(Y^i_s-\Y^i_s) + \eta\sinh(\eta(Y^i_s-\Y^i_s))\right]dM^i_s.
\end{align*}
Then, for sufficiently large $n$ and $\eta$, we have for all $t \in [0,T]$, 
\begin{align}
\frac{1}{n}\sum_{i=1}^n[Q^i]_t &\le \frac{C}{n^3}\sum_{i=1}^n[M^i]_t, \quad \text{ and } \quad \frac{1}{n}\sum_{i=1}^n[M^i]_T \le \frac{C}{n^2} + \frac{C}{n}\sum_{i=1}^n|Q^i_T|. \label{def:estimates-Q,M}
\end{align}
\end{theorem}

\begin{proof}
{The proofs of 
\eqref{pf:mainestimates1}
and
\eqref{eq:Dv:Du:Z:cZ}
are the same as in 
Theorem \ref{th:mainestimate}. In particular,} \eqref{def:estimate-|X-Xbar|} follows from \eqref{def:estimate-Dx}.
The proof of \eqref{def:estimate-Dx} given below is adapted from \cite[Theorem 6.32]{CarmonaDelarue_book_II}.

{\ }

\noindent \textbf{Step 1:}
We first estimate $(Y^i_s-\Y^i_s)^2$, for $s \in [0,T]$.
Our starting point is the equation \eqref{eq:diff:Y:cY:f:lip:2} from the proof of Theorem \ref{th:mainestimate}.
The local Lipschitz assumption on $\widehat{f}$ along with the uniform boundedness of $v^{n,i}$, $u^{n,i}$, and $\Z^{i,i}_t = D_{x_i}u^{n,i}(t,\bm{X}_t)$ implied by  Assumption \ref{assumption:B'} shows that
\begin{align*}
&\left|(Y^i_r-\Y^i_r)\left(\widehat{f}\left(X^i_r,m^n_{\bm{X}_r},Z^{i,i}_r\right) - \widehat{f}\left(X^i_r,m^n_{\bm{X}_r},\Z^{i,i}_r\right)\right)\right| \\
	&\quad \le C\left(1 + |Z^{i,i}_r| + |\Z^{i,i}_r|\right)|Y^i_r-\Y^i_r||Z^{i,i}_r-\Z^{i,i}_r| \\
	&\quad \le C\left(|Y^i_r-\Y^i_r||Z^{i,i}_r-\Z^{i,i}_r| + |Z^{i,i}_r -\Z^{i,i}_r|^2\right) \\
	&\quad \le C\left(|Y^i_r-\Y^i_r|^2 + |Z^{i,i}_r-\Z^{i,i}_r|^2\right).
\end{align*}
Substitute this back into equation \eqref{eq:diff:Y:cY:f:lip:2}, using the Lipschitz continuity of $\widehat{b}$,  Proposition \ref{pr:master-nearly-nash}, the estimate \eqref{pf:mainestimate-zbounds}, and the elementary inequality $2 ab \leq a^2 + b^2$ to get
\begin{align*}
(Y^i_s-\Y^i_s)^2 \le \ &\frac{C}{n^2} + C\int_s^T\left(|Y^i_r-\Y^i_r|^2 + |Z^{i,i}_r-\Z^{i,i}_r|^2 + \frac{1}{n}\sum_{j=1}^n|Z^{j,j}_r-\Z^{j,j}_r|^2\right)dr \\
	&- \int_s^T\sum_{j=1}^n	\Bigl\vert \sigma^\top	\bigl( 	Z^{i,j}_r-\Z^{i,j}_r\bigr) \Bigr|^2 dr - \int_s^T\left|\sum_{j=1}^n	\sigma_{0}^\top	\bigl(Z^{i,j}_r-\Z^{i,j}_r\bigr)\right|^2dr \\
	&- 2\int_s^T(Y^i_r-\Y^i_r)\sum_{j=1}^n(Z^{i,j}_r-\Z^{i,j}_r) \cdot\left(\sigma dB^j_r + \sigma_0 dW_r\right).
\end{align*}

\noindent \textbf{Step 2:}
The $|Z^{i,i}_r-\Z^{i,i}_r|^2$ term poses some problems, as it is not canceled by the quadratic variation term. To deal with this, our next step is to estimate $\cosh(\eta(Y^i_s-\Y^i_s))$ for a fixed $\eta > 0$, to be chosen later. We note that the constant $C < \infty$ will change from line to line but will not depend on the value of $\eta$.
Use It\^o's formula to get:
\begin{align*}
&
d\cosh(\eta(Y^i_s-\Y^i_s)) \nonumber
\\
&= \  \eta\sinh(\eta(Y^i_s-\Y^i_s))\left[-\sum_{j=1}^n\Z^{i,j}_s \cdot \left( \widehat{b}\left(X^j_s,m^n_{\bm{X}_t},Z^{j,j}_s\right) -  \widehat{b}\left(X^j_s,m^n_{\bm{X}_s},\Z^{j,j}_s\right)\right)  \right. \\
	&\quad\quad\quad\quad\quad \left. \vphantom{\sum_{j=1}^n}+  r^{n,i}(s,\bm{X}_s) - \widehat{f}\left(X^i_s,m^n_{\bm{X}_s},Z^{i,i}_s\right) + \widehat{f}\left(X^i_s,m^n_{\bm{X}_s},\Z^{i,i}_s\right)\right]ds \\
	&\quad + \frac{\eta^2}{2}\cosh(\eta(Y^i_s-\Y^i_s))\left(\sum_{j=1}^n
	\Bigl|\sigma^\top\bigl(Z^{i,j}_s-\Z^{i,j}_s\bigr) \Bigr|^2 + \left|\sum_{j=1}^n
	\sigma_{0}^\top	\bigl(Z^{i,j}_s-\Z^{i,j}_s\bigr)\right|^2\right)ds \\
	&\quad + \eta\sinh(\eta(Y^i_s-\Y^i_s))\left(\sum_{j=1}^n(Z^{i,j}_s-\Z^{i,j}_s)	\cdot	\sigma dB^j_s + \sum_{j=1}^n(Z^{i,j}_s-\Z^{i,j}_s) \cdot \sigma_0 dW_s\right).
\end{align*}
As before, use the fact from \eqref{pf:mainestimate-zbounds} that $\Z^{i,i}_r$ is uniformly bounded  and the local Lipschitz continuity of $\widehat{f}$ to get:
\begin{align*}
&\left|\eta\sinh(\eta(Y^i_r-\Y^i_r))\left(\widehat{f}\left(X^i_r,m^n_{\bm{X}_r},Z^{i,i}_r\right) - \widehat{f}\left(X^i_r,m^n_{\bm{X}_r},\Z^{i,i}_r\right)\right)\right| \\
	&\quad \le C\left(1 + |Z^{i,i}_r| + |\Z^{i,i}_r|\right)\left|\eta\sinh(\eta(Y^i_r-\Y^i_r))\right||Z^{i,i}_r-\Z^{i,i}_r| \\
	&\quad \le C\left|\eta\sinh(\eta(Y^i_r-\Y^i_r))\right||Z^{i,i}_r-\Z^{i,i}_r| + C\left|\eta\sinh(\eta(Y^i_r-\Y^i_r))\right||Z^{i,i}_r-\Z^{i,i}_r|^2 \\
	&\quad \le C\left|\eta\sinh(\eta(Y^i_r-\Y^i_r))\right|^2 + C\left(1 + \left|\eta\sinh(\eta(Y^i_r-\Y^i_r))\right|\right)|Z^{i,i}_r-\Z^{i,i}_r|^2 \\
	&\quad \le C\left|\eta\sinh(\eta(Y^i_r-\Y^i_r))\right|^2 + C\left(1 + \eta\cosh(\eta(Y^i_r-\Y^i_r))\right)|Z^{i,i}_r-\Z^{i,i}_r|^2,
\end{align*}
where the last line follows from the elementary inequality $|\sinh(x)| \le \cosh(x)$. Thus, use the Lipschitz continuity of $\widehat{b}$, the inequality $2ab \leq a^2 + b^2$, and \eqref{pf:mainestimate-zbounds} to obtain  
\begin{align*}
&\cosh(\eta(Y^i_s-\Y^i_s)) 
\\
&\le \ C\int_s^T\left[\vphantom{\sum_{j=1}^n}\eta^2\sinh^2(\eta(Y^i_r-\Y^i_r)) + (1 + \eta\cosh(\eta(Y^i_r-\Y^i_r)))|Z^{i,i}_s-\Z^{i,i}_r|^2 \right. \\
	&\quad\quad\quad\quad\quad \left. + \frac{\eta}{n}|\sinh(\eta(Y^i_r-\Y^i_r))|\sum_{j=1}^n|Z^{j,j}_r-\Z^{j,j}_r| + |r^{n,i}(r,\bm{X}_r)|^2\right]dr  \\
		&- \frac{\eta^2}{2}\int_s^T\cosh(\eta(Y^i_r-\Y^i_r))\left(\sum_{j=1}^n
		\Bigl\vert\sigma^\top \bigl( Z^{i,j}_r-\Z^{i,j}_r\bigr) \Bigr\vert^2 + \left|\sum_{j=1}^n \sigma_{0}^\top (Z^{i,j}_r-\Z^{i,j}_r)\right|^2\right)dr  
		\\
	&- \eta\int_s^T\sinh(\eta(Y^i_r-\Y^i_r))\sum_{j=1}^n(Z^{i,j}_r-\Z^{i,j}_r) \cdot (\sigma dB^j_r + \sigma_0 dW_r).
\end{align*}
Finally, since $\sigma$ is non-degenerate,  use the estimate $||r^{n,i}||_\infty \leq C/n$ from Proposition \ref{pr:master-nearly-nash}  
to conclude that 
\begin{align*}
\frac{\eta^2}{2}\int_s^T&\cosh(\eta(Y^i_r-\Y^i_r))\sum_{j=1}^n|Z^{i,j}_r-\Z^{i,j}_r|^2 dr + \cosh(\eta(Y^i_s-\Y^i_s))  \\
	\le \ &\frac{C}{n^2} + C\int_s^T\left[\vphantom{\sum_{j=1}^n}\eta^2\sinh^2(\eta(Y^i_r-\Y^i_r)) + (1 + \eta\cosh(\eta(Y^i_r-\Y^i_r)))|Z^{i,i}_r-\Z^{i,i}_r|^2 \right. dr 
	\\
	&\quad\quad\quad\quad\quad \left. + \frac{\eta}{n}|\sinh(\eta(Y^i_r-\Y^i_r))|\sum_{j=1}^n|Z^{j,j}_r-\Z^{j,j}_r|\right]dr 
	\\
	&- \widetilde{C}\eta\int_s^T\sinh(\eta(Y^i_r-\Y^i_r))\sum_{j=1}^n(Z^{i,j}_r-\Z^{i,j}_r) \cdot (\sigma dB^j_r + \sigma_0 dW_r).
\end{align*}
where the constant $\widetilde{C} < \infty$ depends only on $\sigma$.

\noindent \textbf{Step 3:} We now combine the previous estimates. Use the facts that $\sinh$ is locally Lipshitz and $Y^i_s$ and $\Y^i_s$ are uniformly bounded to find a constant $c_\eta <  \infty$ such that:
\begin{align*}
(Y^i_s-\Y^i_s)^2 &+ \frac{\eta^2}{2}\int_s^T\cosh(\eta(Y^i_r-\Y^i_r))\sum_{j=1}^n|Z^{i,j}_r-\Z^{i,j}_r|^2 dr \\
	\le \ & \frac{C}{n^2} + c_\eta\int_s^T|Y^i_r-\Y^i_r|^2dr +  \frac{C (1+\eta)}{n}\int_s^T\sum_{j=1}^n|Z^{j,j}_r-\Z^{j,j}_r|^2dr  \\
	&- \int_s^T\sum_{j=1}^n
	\Bigl| \sigma^\top \bigl( Z^{i,j}_r-\Z^{i,j}_r \bigr) \Bigr|^2dr  - \int_s^T\left|\sum_{j=1}^n\sigma_0^\top \bigl(Z^{i,j}_r-\Z^{i,j}_r\bigr)\right|^2dr 
	\\
	& + C\int_s^T(1 + \eta\cosh(\eta(Y^i_r-\Y^i_r)))|Z^{i,i}_r-\Z^{i,i}_r|^2dr  
	\\
	&- \int_s^T\left[2(Y^i_r-\Y^i_r) + \widetilde{C}\eta\sinh(\eta(Y^i_r-\Y^i_r))\right]\sum_{j=1}^n(Z^{i,j}_r-\Z^{i,j}_r) \cdot \left(\sigma dB^j_r + \sigma_0 dW_r\right).  
\end{align*}
Now, for sufficiently large $\eta$, we have:
\[
C(1+\eta\cosh(\eta x)) \le 2C\eta\cosh(\eta x) \le \frac{\eta^2}{4}\cosh(\eta x),
\]
for any $x \in \R$. Hence, for large $\eta$, we find 
\begin{align}
(Y^i_s-\Y^i_s)^2 &+ \frac{\eta^2}{4}\int_s^T\cosh(\eta(Y^i_r-\Y^i_r))\sum_{j=1}^n|Z^{i,j}_r-\Z^{i,j}_r|^2 dr  \nonumber \\
	\le \ & \frac{C}{n^2} +  c_\eta\int_s^T|Y^i_r-\Y^i_r|^2dr +   \frac{C\eta}{n} \int_s^T\sum_{j=1}^n|Z^{j,j}_r-\Z^{j,j}_r|^2dr  
	\nonumber
	\\
	&- \int_s^T\sum_{j=1}^n
	\Bigl| \sigma^\top \bigl( Z^{i,j}_r-\Z^{i,j}_r\bigr) \Bigr|^2 dr - \int_s^T\left|\sum_{j=1}^n
	\sigma_0^\top
	\bigl(Z^{i,j}_r-\Z^{i,j}_r\bigr)\right|^2dr \nonumber
	\\
	&- \int_s^T\left[2(Y^i_r-\Y^i_r) + \widetilde{C}\eta\sinh(\eta(Y^i_r-\Y^i_r))\right]\sum_{j=1}^n(Z^{i,j}_r-\Z^{i,j}_r) \cdot \left(\sigma dB^j_r + \sigma_0 dW_r\right).
	\label{eq:assumption:C:y:cy}
\end{align}
Averaging \eqref{eq:assumption:C:y:cy} over $i=1,\ldots,n$ and increasing the value of $\eta$  if needed, we obtain  
\begin{align}
\frac{1}{n}&\sum_{i=1}^n(Y^i_s-\Y^i_s)^2 + \frac{\eta^2}{2n}\sum_{i=1}^n\int_s^T\cosh(\eta(Y^i_r-\Y^i_r))\sum_{j=1}^n|Z^{i,j}_r-\Z^{i,j}_r|^2 dr \nonumber \\
	&\quad + \frac{1}{n}\sum_{i=1}^n\sum_{j=1}^n\int_s^T
	\Bigl| \sigma^\top \bigl( Z^{i,j}_r-\Z^{i,j}_r\bigr) \Bigr|^2 dr + \frac{1}{n}\sum_{i=1}^n\int_s^T\left|\sum_{j=1}^n
	\sigma_{0}^\top \bigl(Z^{i,j}_r-\Z^{i,j}_r\bigr)\right|^2 dr \nonumber \\
	&\le \frac{C}{n^2} + \frac{c_\eta}{n}\sum_{i=1}^n\int_s^T|Y^i_r-\Y^i_r|^2dr  \label{pf:bigestimate1} \\
	&\quad - \frac{1}{n}\sum_{i=1}^n\int_s^T\left[2(Y^i_r-\Y^i_r) + \widetilde{C}\eta\sinh(\eta(Y^i_r-\Y^i_r))\right]\sum_{j=1}^n(Z^{i,j}_r-\Z^{i,j}_r) \cdot \left(\sigma dB^j_r + \sigma_0 dW_r\right). \nonumber
\end{align}
Noting that the stochastic integral is a martingale thanks to Lemma \ref{le:YZfinite-moments} and the boundedness of $Y^i$ and $\Y^i$, we take conditional expectations and apply Gronwall's inequality as in \eqref{eq:y:cy:B:average} to find
\begin{align}
\frac{1}{n}\sum_{i=1}^n(Y^i_s-\Y^i_s)^2 \le \frac{C}{n^2}. \label{pf:Ybound1}
\end{align}
Moreover, substituting \eqref{pf:Ybound1} back into \eqref{pf:bigestimate1}, using non-degeneracy of $\sigma$, and taking conditional expectations, we see that 
\begin{align}
\frac{1}{n}\sum_{i=1}^n\sum_{j=1}^n\E\left[\int_s^T|Z^{i,j}_r-\Z^{i,j}_r|^2dr 
\big\vert \F_{s}
\right] \le \frac{C}{n^2}.
\label{eq:z:cz:C}
\end{align}
Returning to \eqref{eq:assumption:C:y:cy} and following the proof of \eqref{eq:y:cy:B}, we obtain:
\begin{align}
\max_{i=1,\ldots,n} (Y^i_s-\Y^i_s)^2 &\le \frac{C}{n^2}. \label{eq:y:cy:C}
\end{align}
Recall now that $\sinh$ is locally Lipschitz and $Y^i$ and $\Y^i$ are uniformly bounded. Using \eqref{eq:y:cy:C} we conclude that 
\[
[Q^i]_t \le \frac{C}{n^2}[M^i]_t,
\]
for all $i \in \{1,\ldots,n\}$ and $t \in [0,T]$. On the other hand, from \eqref{pf:bigestimate1} and \eqref{eq:y:cy:C} we get:
\[
\frac{1}{n}\sum_{i=1}^n[M^i]_T \le \frac{C}{n^2} + \frac{C}{n}\sum_{i=1}^n|Q^i_T|.
\]
This completes the proof of \eqref{def:estimates-Q,M}.

{\ }

\noindent \textbf{Step 4:} Finally, to prove \eqref{def:estimate-Dx}, it suffices to take expectations in \eqref{eq:Dv:Du:Z:cZ} and to use \eqref{eq:z:cz:C}.
\end{proof}

\subsection{Law of large numbers} \label{se:LLN}
As in \cite{cardaliaguet-delarue-lasry-lions}, the estimates of Theorems \ref{th:mainestimate} and \ref{th:mainestimate2} imply a law of large numbers for the equilibrium empirical measure $m^n_{\bm{X}}$. That is, we can now prove Theorem \ref{th:LLN}:

\begin{proof}[Proof of Theorem \ref{th:LLN}]
Note first that \eqref{def:estimate-|X-Xbar|} implies
\begin{align}
\E\left[\W_{2,\C^d}^{2}(m^n_{\bm{X}},m^n_{\bm{\overline{X}}})\right] \le \frac{C}{n^2}. \label{pf:LLN1}
\end{align}
On the other hand, a standard McKean-Vlasov limit theorem (e.g., \cite[Section 7]{carmona2016probabilistic}
or 
\cite[Chapter 2, Section 2.1]{CarmonaDelarue_book_II}) implies 
\[
\lim_{n\rightarrow\infty}\E\left[\W_{2,\C^d}^{2}(m^n_{\bm{\overline{X}}},\mu)\right] = 0.
\]
This proves the theorem.  
\end{proof}

\section{Proof of the central limit theorem} \label{se:CLT-proofs}

As we already mentioned in the  Introduction, the proof of the CLT is split into two parts. The first part is to control the distance between the two particle systems ${\boldsymbol X}$
and $\overline{\boldsymbol X}$, which is in fact the purpose of Theorem 
\ref{th:mainestimate}. The second part consists of deriving a CLT for the flow of empirical distributions  $(m^n_{\overline{\boldsymbol X}_{t}})_{0 \le t \le T}$. The first part is addressed in Section \ref{se:MFGfluctuationproof}, after we have developed the requisite machinery, and we begin with the second part. 

Generally speaking, the proof of the CLT goes along the lines of earlier works on the subject, see for instance \cite{hitsuda-mitoma}, \cite{meleard}, \cite{jourdain-meleard} and \cite{kurtz-xiong}. However, among these four references, only the last one allows for a common noise in the particle system or a nonlinear mean field term. 
In the approach we implement below, we follow the main steps of \cite{kurtz-xiong}, except that the space of distributions in which the CLT is proven is taken from \cite{meleard,jourdain-meleard}. 

A crucial point throughout the proof is that, under the assumption of Theorem \ref{th:CLT:sec:3} and for the same $p'$ as therein, one has 
${\mathbb E}[ \sup_{t \in [0,T]} \int_{\R^d}
\vert x \vert^{p'} d \mu_{t}(x)]< \infty$
and 
$\sup_{1 \leq i \leq N}{\mathbb E}[ \sup_{t \in [0,T]} \vert \overline X_{t}^i \vert^{p'}]< \infty$,
with $\overline{\boldsymbol X}$ as in 
\eqref{eq:sec4:overlineX}.

\subsection{Tightness}
The first step for proving a CLT for 
$\overline{\boldsymbol X}$ is to show that the difference $(\overline S^n_{t}=\sqrt{n}(m^n_{\overline{\boldsymbol X}_{t}}- \mu_{t}))_{t \in [0,T]}$ is tight in a suitable space of distributions. This is precisely the point where we prefer to follow \cite{meleard,jourdain-meleard} instead of \cite{hitsuda-mitoma,kurtz-xiong}. There are several reasons for that. First, the  framework used in \cite{meleard,jourdain-meleard} allows for weaker integrability conditions of the initial distribution. Second, the analysis therein is explicitly fitted to the higher-dimensional setting, whilst the one performed in \cite{hitsuda-mitoma,kurtz-xiong} is in dimension 1 only; an additional effort would be needed to make the latter  adapted to the case $d \geq 2$, since, as we shall see below, the dimension plays a crucial role in the definition of the space of distributions used in the statement of the CLT. In this regard, it seems that the space we use below permits to track the dependence on the dimension in a more explicit way.

Here, we use the same space of distributions as in 
\cite{meleard,jourdain-meleard}, as defined in \eqref{eq:cH:j,alpha}. 
Roughly speaking, we shall prove {under the assumptions of Theorem \ref{th:CLT:sec:3}} that the sequence $(\overline S^n_{t})_{t \in [0,T]}$ is tight in the space $C([0,T],{\mathcal H}^{-(2+2\newD),2\newD})$, where $\newD = \lfloor d/2 \rfloor + 1$. See Theorem \ref{thm:clt:tightness} below for a more precise statement.

We only provide the main steps of the proof, as it is quite similar to the aforementioned references. There are two key ingredients,  described in the following two subsections. 

\subsubsection{Coupling with the limiting process.} The first step to get tightness is to prove that the particle system $\overline{\boldsymbol X}$ is at most at distance $n^{-1/2}$ of the auxiliary particle system $\widehat{\boldsymbol X}$ defined as the solution to the following SDE system:
\begin{equation}
\label{eq:hat:coupling}
d \widehat{X}_{t}^i = \widehat{b}\bigl( \widehat{X}_{t}^i,\mu_{t}, D_x  U(t,\widehat X_{t}^i,\mu_{t}) \bigr) dt 
+ \sigma dB_{t}^i + \sigma_{0} dW_{t}, \quad 
\widehat{X}_{0}^i = \overline{X}_{0}^i, 
\end{equation}
for $i \in \{1,\ldots,n\}$. 
It is customary to introduce $\widehat{\boldsymbol X}$ in the analysis of mean field systems, see for instance \cite{sznitman1991topics}. A key feature is that, conditional on $W$, the particles $\widehat{X}^1,\ldots,\widehat{X}^n$ are independent. 
Below, we use the shortened notation $\widetilde{b}(t,x,m) = \widehat{b}(x,m,D_x  U(t,x,m))$, 
{which is similar to \eqref{widehat:widetilde:b:n:i}}.
We have the following bound, analogous to \cite[Theorem 2.4]{kurtz-xiong}.
 
\begin{lemma}
\label{lem:clt:n1/2}
Under the assumptions of Theorem 
\ref{th:CLT:sec:3}, we can find a constant $C$ such that, for any $n \geq 1$ and any $i \in \{1,\ldots,n\}$,
\begin{equation*}
\E\left[\sup_{t \in [0,T]}\bigl\vert \overline{X}_{t}^i - \widehat{X}_{t}^i \bigr\vert^{4}
\right]^{1/4} \leq C n^{-1/2},
\end{equation*}
\end{lemma}

\begin{proof} {\ } \\
\noindent\textbf{Step 1:}
The key point is to show the bound
\begin{equation}
\label{eq:clt:n12:w2}
\E\left[ \sup_{t \in [0,T]}{\mathcal W}_{2}\bigl( m^n_{\overline{\boldsymbol X}_{t}},m^n_{\widehat{\boldsymbol X}_{t}}\bigr)^{4} \right]^{1/4} \leq C n^{-1/2}.
\end{equation}
Although the derivation of \eqref{eq:clt:n12:w2} may seem standard, as the rate $n^{-1/2}$ is the one we expect from the law of large numbers, there is in fact a subtle point.
The proof would be fairly straightforward if the mean field interaction was of order 1, in the sense that $\widetilde{b}(t,x,m)$ was of the form of 
$\int_{\RR^d} \overline{b}(t,x,x') d m(x')$, for some Lipschitz function $\overline{b} : [0,T] \times \RR^d \times \RR^d \rightarrow \RR^d$ (cf. \cite{meleard,jourdain-meleard}). Obviously, there is no reason for this to be the case here, as $U$ is expected to inherit the nonlinear structure of the master equation. 
As shown in \cite{kurtz-xiong} (see in particular the assumption (S3) therein), \eqref{eq:clt:n12:w2} holds true provided
that for any $n \in \N$, any $m_0 \in \P^{p'}(\R^d)$, any i.i.d.\ samples $\bm{\zeta}=(\zeta_1,\ldots,\zeta_n)$ with law $m_0$, any $i \in \{1,\ldots,n\}$, and any $t \in [0,T]$ we have:
\begin{equation}
\label{eq:S3}
\begin{split}
&{\mathbb E} \Bigl[ \bigl\vert \widetilde{b}\bigl(t,\zeta_i,m^n_{{\boldsymbol \zeta}} \bigr) 
- 
\widetilde{b}\bigl(t,\zeta_i,m_0 \bigr)
\bigr\vert^{4} \Bigr]^{1/4} 
 \leq C \Bigl( 1  +\int_{\RR^d} \vert x \vert^4 \, m_0(dx) 
\Bigr) n^{-1/2}.
\end{split}
\end{equation}
The role of \eqref{eq:S3} is well-understood. Indeed, the proof of 
\eqref{eq:clt:n12:w2} is a straightforward application of Gronwall's lemma, provided we can estimate the term
\begin{equation*}
\int_{0}^T {\mathbb E} \Bigl[
 \bigl\vert \widetilde{b}\bigl(t,\widehat X_{t}^i,m^n_{\widehat{\boldsymbol X}_{t}} \bigr) 
- 
\widetilde{b}\bigl(t,\widehat{X}_{t}^i,{\mu_t} \bigr)
\bigr\vert^{4} \Bigr]^{1/4} dt.
\end{equation*}
Now, it suffices to invoke \eqref{eq:S3} but under 
$\EE[ \, \cdot \, \vert W]$, as the particles 
$(\widehat X^1_{t},\ldots,\widehat X^n_{t})$ are {i.i.d.\ with law $\mu_t$,} conditional on $W$. 
\vskip 4pt

\noindent\textbf{Step 2:}
We turn to the proof of \eqref{eq:S3}.
Using the Lipschitz property of $\widetilde{b}$ in $m$, we can easily replace $m^n_{\bm{\zeta}}$ by $m^{n,-i}_{\bm{\zeta}}$, the empirical measure of $(\zeta_k)_{k \neq i}$. Indeed, since $\widetilde{b}$ is Lipschitz,
the left-hand side in \eqref{eq:S3} can be bounded by:
\begin{equation*}
\begin{split}
&C n^{-1/2} {\mathbb E} \bigl[  \vert \zeta_i \vert^{4} \bigr]^{1/4} 
+ {\mathbb E} \Bigl[
 \bigl\vert \widetilde{b}\bigl(t,\zeta_i,m^{n,-i}_{\bm{\zeta}} \bigr) 
- 
\widetilde{b}\bigl(t,\zeta_i,m_0 \bigr)
\bigr\vert^{4} \Bigr]^{1/4}.
\end{split}
\end{equation*}
By conditioning on $\zeta_{i}$ in the second term and by changing the variable 
$n$ into $n-1$,
it is in fact sufficient to prove: 
\begin{equation}
\label{eq:S3:bis}
\begin{split}
&{\mathbb E} \Bigl[ \bigl\vert \widetilde{b}\bigl(t,x,m^n_{\bm{\zeta}} \bigr) 
- 
\widetilde{b}\bigl(t,x,m_0 \bigr)
\bigr\vert^{4} \Bigr]^{1/4} 
 \leq C \Bigl( 1  +\int_{\RR^d} \vert x' \vert^4 \, m_0(dx') 
\Bigr) n^{-1/2},
\end{split}
\end{equation}
for any $x \in \RR^d$, the constant $C$ being independent of $x$.  
So, the only point is to find sufficient conditions on $\widetilde{b}$ to guarantee 
\eqref{eq:S3:bis}. This is exactly the purpose of 
Lemma \ref{lem:lln:l4:appendix} in appendix. 
It requires that $m \mapsto \widetilde{b}(t,x,m)$ is ${{\mathscr C}^2}$ in the sense defined in Subsection \ref{subse:derivatives:m:P2} with bounded derivatives of order 1 and 2 and that the derivative of order 2 is Lipschitz continuous in 
$m$ with respect to the $1$-Wasserstein distance, uniformly in the other parameters. This holds true under the assumption of Theorem 
\ref{th:CLT:sec:3}.
\end{proof} 

\subsubsection{A preliminary bound toward tightness} A second important point in the proof is that the following (Sobolev) embeddings are continuous:  
\begin{eqnarray}
\label{eq:embedding:1}
&{\mathcal H}^{\ell+j,\alpha} \hookrightarrow {\mathfrak C}^{j,\alpha}, &\quad \ell > \frac{d}{2}, \ j \geq 0, \ \alpha \geq 0,
\\
&{\mathfrak C}^{j,0}  \hookrightarrow {\mathcal H}^{j,\alpha}, 
&\quad \alpha > \frac{d}{2}, \ j \geq 0,
\label{eq:embedding:2}
\\
&{\mathcal H}^{\ell+j,\alpha}
\hookrightarrow_{\textrm{\rm H.S.}}
{\mathcal H}^{j,\alpha+\beta}, &\quad
\ell > \frac{d}{2}, \ j \geq 0, \ \alpha \geq 0, \ \beta > \frac{d}{2},
\label{eq:embedding:3}
\end{eqnarray}
with similar Sobolev embeddings for the dual spaces (with converse order), where ${\mathfrak C}^{j,\alpha}$ is the space of functions $g$ with continuous partial derivatives up to order $j$ such that:
\begin{equation*}
\lim_{\vert x \vert \rightarrow \infty}
\frac{\vert D^{\boldsymbol k} g(x) \vert}{1+ \vert x \vert^{\alpha}} =0, \quad \textrm{\rm for  all multi-indices}
\ {\boldsymbol k} \ \textrm{\rm such that} \ \vert  {\boldsymbol k} \vert \leq j.
\end{equation*}
In the embeddings (\ref{eq:embedding:1}--\ref{eq:embedding:2}), the space ${\mathfrak C}^{j,\alpha}$ is equipped with with the norm:
\begin{equation*}
\| g \|_{ {{\mathfrak C}}^{j,\alpha}}
= \sum_{\vert {\boldsymbol k} \vert \leq j}
\sup_{x \in \RR^d}
\frac{ \vert D^{\boldsymbol k} g(x) \vert}{1+ \vert x \vert^{\alpha}}.
\end{equation*}
In \eqref{eq:embedding:3}, the symbol H.S. means that the embedding is Hilbert-Schmidt.

As a consequence of Lemma 
\ref{lem:clt:n1/2}, we have the following result, the proof of which is a mere adaptation of Proposition 5.3 in \cite{meleard}.

\begin{proposition}
\label{prop:tightness:1}
Under the assumptions of Theorem \ref{th:CLT:sec:3}, for any $\epsilon  \in (0,1)$ 
such that $p' > 8\newD (1+\epsilon)/(1-\epsilon)$ with 
$p'$ as in Assumption \textbf{\bf A}, it holds:
\begin{equation*}
\sup_{n \geq 1} \sup_{t \in [0,T]}
\EE \bigl[ \|\overline S_{t}^n \|_{-(1+\newD),2\newD}^{2(1+\epsilon)}
\bigr] < \infty.
\end{equation*}
\end{proposition}

\begin{proof} For a test function $\phi \in {\mathcal H}^{1+\newD,2\newD}$, we let:
\begin{equation*}
\overline K^n_{t}(\phi) = n^{-1/2}
\sum_{i=1}^n \bigl( \varphi( \overline X_{t}^i ) - \varphi( \widehat X_{t}^i ) \bigr), 
\quad \overline L^n_{t}(\phi) = n^{-1/2}
\sum_{i=1}^n \bigl(    \phi( \widehat X_{t}^i ) \bigr) - \EE [ \phi(\widehat X_{t}^1)  \, | \, W] 
\bigr),
\end{equation*}
noting that $\overline{K}^n_t(\phi)+\overline{L}^n_t(\phi)=\overline{S}^n_t(\phi)$.
As for the first term, we use the fact that 
\begin{equation*}
\begin{split}
\| \overline{K}_{t}^n \|^{2(1+\epsilon)}_{-(1+\newD),2\newD}
&\leq n^{1+\epsilon} n^{-1}
\sum_{i=1}^n \| d_{\overline{X}^i_{t},\widehat{X}^i_{t}} \|^{2(1+\epsilon)}_{-(1+\newD),2\newD} = n^{\epsilon}     
\sum_{i=1}^n \| d_{\overline{X}^i_{t},\widehat{X}^i_{t}} \|^{2(1+\epsilon)}_{-(1+\newD),2\newD},
\end{split}
\end{equation*}
and we conclude as 
 in the derivation of 
\cite[(5.3)]{meleard}, using H\"older instead of Cauchy-Schwarz inequality with exponents $2/(1+\epsilon)$ and $2/(1-\epsilon)$. 
Here we define $d_{x,y}(\phi) = \phi(x) - \phi(y)$ and $\| d_{x,y} \|_{-(1+\newD ),2\newD }
\leq C \vert x- y \vert ( 1+ \vert x \vert^{2\newD } 
+ \vert y \vert^{2\newD })$, see 
\cite[Lemma 6.1]{meleard}.

Regarding the second term, we apply Rosenthal's inequality for Hilbert-valued  independent random variables, see \cite[Theorem 5.2]{pinelis1994}. 
Indeed, we observe that 
\begin{equation*}
\overline L^n_{t} = n^{-1/2} \sum_{i=1}^n 
\bigl( d_{\widehat X_{t}^i}
- {\mathbb E}[ d_{\widehat {X}_{t}^1} \, | \, W]
\bigr),
\end{equation*}
with 
$d_{x}(\phi) = \phi(x)$ and 
$\| d_{x}\|_{-(1+\newD ),2\newD } \leq 
C (1+ \vert x \vert^{2\newD })$,
see again 
\cite{meleard}. In particular,  
$\| d_{\widehat X_{t}^i}
- {\mathbb E}[ d_{\widehat {X}_{t}^1}]
\|_{-(1+\newD ),2\newD } \leq C ( 1+ \vert X_{t}^i \vert^{2\newD } + {\mathbb E}[ \vert 
\widehat {X}_{t}^1 \vert^{2\newD }])$. 
The result follows.
\end{proof}

\subsubsection{Statement and proof of tightness}
The proof of tightness relies on the following expansion, which holds true for any smooth test function 
$\phi$ with compact support. 
\begin{equation}
\label{eq:expansion:1}
\begin{split}
d \langle \overline S^n_{t},\phi \rangle 
&= \sqrt{n}\Bigl[\bigl\langle
m^n_{\overline{\boldsymbol X}_{t}},
 D_x \phi(\cdot) \cdot \widetilde b(t,\cdot,m^n_{\overline{\boldsymbol X}_{t}}) \bigr\rangle - \bigl\langle \mu_{t}, D_x \phi(\cdot) \cdot \widetilde b(t,\cdot,\mu_{t})
  \bigr\rangle \Bigr] dt 
 \\
&\hspace{15pt} + \tfrac12
  \bigl\langle \overline S^n_{t},  
\mathrm{Tr}[(\sigma\sigma^\top+\sigma_0\sigma_0^\top)D_x^2\phi]
   \bigr\rangle dt 
 \\
&\hspace{15pt} +
\bigl\langle \overline S^n_{t}, \sigma_{0}^\top D_x \phi
\rangle \cdot dW_{t}
 + \frac{1}{\sqrt{n}} \sum_{i=1}^n 
\bigl( \sigma^\top D_x \phi(\overline{X}_{t}^i)
\bigr)  \cdot dB_{t}^i. 
\end{split}
\end{equation}
Following \cite[(5.6)]{meleard}, we let:
\begin{equation*}
M^{n,\phi}_{t} =
\frac{1}{\sqrt{n}} \sum_{i=1}^n \int_0^t 
\bigl( \sigma^\top
D_x \phi(\overline X_{s}^i) \bigr) \cdot dB_{s}^i.
\end{equation*}
If we call $(d_{x} D_{x_\ell})(\phi) =
(\partial \phi/\partial x_{\ell})(x)$ for $\ell \in \{1,\cdots,d\}$, we have
$\| d_{x} D_{x_\ell} \|_{-(1+\newD ),2\newD }
\leq C (1+ \vert x \vert^{2\newD })$, see the proof of 
\cite[Lemma 6.1]{meleard}. 
Hence, if $(e_{k})_{k \geq 1}$ denotes a complete orthonormal system in ${\mathcal H}^{1+\newD ,2\newD }$, we use Parseval's identity to get
\begin{equation*}
\begin{split}
\E\left[
\sum_{k \geq 1} \sup_{t \in [0,T]}
\vert M^{n,e_{k}}_{t} \vert^2
\right] &\leq C
\sum_{k \geq 1}
 \int_{0}^T {\mathbb E}
\bigl[ \vert D_x e_{k}(\overline X_{t}^1) \vert^2 \bigr] dt
\leq C
 \int_{0}^T {\mathbb E}
\bigl[ \bigl( 1+ \vert \overline X_{t}^1 \vert^{4\newD }\bigr)  \bigr] dt < \infty.
\end{split} 
\end{equation*}
Therefore, we can define:
\begin{equation*}
M^n_{t} = \sum_{k \geq 1}
M^{n,e_{k}}_{t} e_{k}^*, \quad t \in [0,T],
\end{equation*}
where $e_{k}^* \in {\mathcal H}^{-(1+\newD ),2\newD }$ is defined by $e_{k}^*(\phi) = \langle e_{k}, \phi \rangle$. 
The process $(M^n_{t})_{t \in [0,T]}$ forms a continuous martingale with values in ${\mathcal H}^{-(1+\newD ),2\newD }$. 
Obviously, $(M^n(\phi))_{t \in [0,T]}$ coincides almost surely with $M^{n,\phi}$. Actually, following \cite[Proposition 5.7]{meleard}, the above bound can be made stronger, namely
\begin{equation}
\label{eq:cl:b}
\sup_{n \geq 1} \E\left[ \sup_{t \in [0,T]} \| M^n_{t} \|^2_{-(1+\newD ),2\newD }
\right] < \infty.
\end{equation}
Here is now the main claim of this subsection.

\begin{proposition}
\label{thm:clt:tightness}
Under the assumptions of Theorem 
\ref{th:CLT:sec:3}, 
 the sequences $((\overline S^n_{t})_{t \in [0,T]})_{n \geq 1}$ and $((M^n_{t})_{t \in [0,T]})_{n \geq 1}$ are tight on the space $C([0,T],{\mathcal H}^{-(2+2\newD ),\newD })$, where again $\newD := \lfloor d/2 \rfloor + 1$.
\end{proposition}

\begin{proof} {\ } \\
\noindent\textbf{Step 1:}
We write the first term in the right hand side of \eqref{eq:expansion:1} in the form:
\begin{equation}
\label{eq:expansion:2}
\begin{split}
&{\sqrt{n}}
\bigl\langle
m^n_{\overline{\boldsymbol X}_{t}},
 D_x \phi(\cdot) \cdot \widetilde b(t,\cdot,m^n_{\overline{\boldsymbol X}_{t}}) \bigr\rangle - \bigl\langle \mu_{t}, D_x \phi(\cdot) \cdot \widetilde b(t,\cdot,\mu_{t})
  \bigr\rangle
  \\
  &\quad\quad=
    \bigl\langle
\overline{S}^n_{t},
 D_x \phi(\cdot) \cdot \widetilde b(t,\cdot,
 m^n_{\overline{\boldsymbol X}_{t}}
 ) \bigr\rangle
 + \Delta^n_{t}(\phi),
 \end{split}
\end{equation}
where 
\begin{equation*}
\begin{split}
 \Delta^n_{t}(\phi)
 &=  \sqrt{n}\bigl\langle \mu_{t}, 
 D_x \phi(\cdot) \cdot \widetilde b(t,\cdot,m^n_{\overline{\boldsymbol X}_{t}}) - D_x \phi(\cdot) \cdot \widetilde b(t,\cdot,\mu_{t})
  \bigr\rangle \\
  &=\int_{\RR^d} d\mu_{t}(x) \int_{0}^1 d\lambda \int_{\RR^d} D_x \phi(x) \cdot \frac{\delta \widetilde b}{\delta m}\bigl(t,x,\lambda m^n_{\overline{\boldsymbol X}_{t}} + (1-\lambda) \mu_{t}\bigr)(v) d \overline{S}^n_{t} (v) \\
&=  
\int_{\RR^d}d \overline{S}^n_{t} (v)\biggl[\int_{\RR^d} D_x \phi(x) \cdot \biggl(\int_{0}^1\frac{\delta \widetilde b}{\delta m}\bigl(t,x,\lambda m^n_{\overline{\boldsymbol X}_{t}} + (1-\lambda) \mu_{t}\bigr)(v) d \lambda\biggr)d\mu_{t}(x) \biggr].
\end{split}
\end{equation*}
Following  \cite[(5.8)]{meleard}, we write $ \Delta^n_{t}(\phi)$ as
\begin{equation}
\label{eq:expansion:3}
 \Delta^n_{t}(\phi)
 =  \Bigl\langle 
 \overline{S}^n_{t},\bigl[
 {\mathcal G}\bigl(t,m^n_{\overline{\boldsymbol X}_{t}} ,\mu_{t}\bigr) \phi \bigr](\cdot)
 \Bigr\rangle,
 \end{equation}
 with
 \begin{equation}
 \label{eq:def:cal:G}
\bigl[ {\mathcal G}\bigl(t,m^n_{\overline{\boldsymbol X}_{t}} ,\mu_{t}\bigr) \phi \bigr](v)
= 
\biggl\langle 
\mu_{t},
D_x \phi(\cdot) \cdot 
\biggl(
\int_{0}^1
\frac{\delta \widetilde b}{\delta m}\bigl(t,\cdot,
\lambda m^n_{\overline{\boldsymbol X}_{t}}  + (1-\lambda) \mu_{t}\bigr)(v) d \lambda 
\biggr)
\biggr\rangle.
\end{equation}
We then have the analogue of
\cite[Lemma 5.6]{meleard}, namely
\begin{equation}
\label{eq:bound:2}
\begin{split}
&\bigl\|
\bigl[ {\mathcal G}\bigl(t,m^n_{\overline{\boldsymbol X}_{t}},\mu_{t}\bigr) \phi \bigr](\cdot)
\bigr\|_{(1+\newD ),2\newD } \leq C \bigl\| \phi \bigr\|_{2+2\newD ,\newD }
\biggl( \int_{\RR^d} \bigl( 1+ \vert y \vert^{2\newD } \bigr) 
d\mu_{t}(y) \biggr)^{1/2}
\\
&\bigl\|
D^2_{x_\ell,x_{\ell'}} \phi (\cdot)
\bigr\|_{1+\newD,2\newD} 
+
\bigl\|
D_{x_\ell}\phi (\cdot)
\bigr\|_{1+\newD ,2\newD }
\leq C \bigl\| \phi \bigr\|_{3+\newD ,2\newD }, \quad \ell,\ell' =1,\cdots,d, \phantom{\int}
\\
&\bigl\|
\bigl[ D_x \phi \cdot \widetilde b(t,\cdot,m^n_{\overline{\boldsymbol X}_{t}})  \bigr](\cdot)
\bigr\|_{1+\newD ,2\newD } \leq C \bigl\| \phi \bigr\|_{3+\newD ,2\newD }.\phantom{\int}
\end{split}
\end{equation}
Indeed, under assumption \ref{assumption:C}, $\widetilde b$ and its derivatives in $x$ up to order $\newD +1$ are bounded, and the last two lines follow easily. After noticing that 
$\| \phi \|_{3+\newD ,2\newD } \leq C
\| \phi \|_{2+2\newD ,\newD }$ since $\newD  \geq 1$,
the first line follows from the fact that $\delta \widetilde b/\delta m$ and its derivatives in $v$ up to order $\newD +1$ are bounded, see \cite[Lemma 5.6]{meleard}.
Notice that the right-hand side in 
the first line is random, which is a special feature due to the presence of the common noise. It implies in particular that
\begin{equation}
\label{eq:conclusion:3rdstep}
\begin{split}
&\bigl\|
{\mathcal G}\bigl(t,m^n_{\overline{\boldsymbol X}_{t}},\mu_{t}\bigr)^* \overline{S}^n_{t} 
\bigr\|_{-(2+2\newD ),\newD }^2 
\leq C  
\biggl( 1 +\int_{\RR^d} \vert y \vert^{2\newD } d\mu_{t}(y)
\biggr)
\|\overline S^n_{t}\|_{-(1+\newD ),2\newD }^2.
\end{split}
\end{equation}
 \vspace{4pt}

\noindent\textbf{Step 2:} We now have the 
analogue of \cite[Proposition 5.9]{meleard}. We claim that, for a complete orthonormal system $(\psi_{p})_{p \geq 1}$ in ${\mathcal H}^{2 + 2\newD ,\newD }$, we have
\begin{equation}
\label{eq:bound:4}
\sup_{n \geq 1} 
\E\left[ \sum_{p \geq 1}
\sup_{t \in [0,T]} \langle \overline{S}^n_{t},\psi_{p}
\rangle^2 \right] < \infty. 
\end{equation}
In comparison with the proof of \cite[Proposition 5.9]{meleard}, there are two main differences, concerning the two terms (see 
\eqref{eq:expansion:1}):
\begin{equation} 
\label{eq:clt:two:difficult:terms}
\begin{split}
&\int_{0}^T \Bigl\vert 
\Bigl\langle \overline{S}_{s}^n,
\bigl[ {\mathcal G}(s,m^n_{\overline{\boldsymbol X}_{s}},\mu_{s}) \psi_{p} \bigr](\cdot)
\Bigr\rangle
\Bigr\vert^2 ds,
\quad \sup_{t \in [0,T]} \biggl\vert \int_{0}^t
\bigl\langle \overline{S}^n_{s}, \sigma_{0}^{\top}D_x \psi_{p}
\bigr\rangle \cdot dW_{s}
\biggr\vert^2.
\end{split}
\end{equation}
As for the first term, the difference with \cite{meleard} comes from the fact that, in \eqref{eq:conclusion:3rdstep}, the integral with respect to $\mu_{t}$ is random. A careful inspection of the proof of  \cite[Proposition 5.9]{meleard} combined with the statement 
of our Proposition \ref{prop:tightness:1}
 shows that (with $\epsilon$ as therein)
\begin{equation*}
\begin{split}
\E\left[ \sum_{p \geq 1} \bigl\langle \overline S^n_{t},\bigl[ {\mathcal G}\bigl(t,m^n_{\overline{\boldsymbol X}_{t}},\mu_{t}\bigr) \psi_{p} \bigr](\cdot) \bigr\rangle^2\right]
&=
\EE \Bigl[ 
\bigl\|
{\mathcal G}(t,m^n_{\overline{\boldsymbol X}_{t}},\mu_{t})^* \overline{S}^n_{t} 
\bigr\|_{-(2+2\newD ),\newD }^2 \Bigr]
\\
&\leq C \EE
\biggl[ 
\biggl( 1 +\int_{\RR^d} \vert y \vert^{2\newD } d\mu_{t}(y)
\biggr)
\|\overline S^n_{t}\|_{-(1+\newD ),2\newD }^2
\biggr]
\\
&\leq C \EE \biggl[ \biggl( 
1 + \int_{\R^d} \vert y \vert^{2\newD } d\mu_{t}(y)
\biggr)^{\frac{1+\epsilon}{\epsilon}}
\biggr]^{{\frac{\epsilon}{1+\epsilon}}}.
\end{split}
\end{equation*}
Since $p'>12 \newD $, we can choose 
$\epsilon = 1/5$, then $ 8\newD  (1+ \epsilon) / (1-\epsilon) = 12\newD $ and $2\newD  (1+\epsilon)/\epsilon =12 \newD $, and then   
Proposition
\ref{prop:tightness:1} holds true and the above right hand side is finite, uniformly in $t$.

We now handle the second term in \eqref{eq:clt:two:difficult:terms}. This term is new. By Doob's inequality,
\begin{equation*}
\begin{split}
\sum_{p \geq 1}
\E\left[ \sup_{t \in [0,T]} \biggl\vert \int_{0}^t
\bigl\langle \overline{S}^n_{s}, \sigma_{0}^{\top} D_x \psi_{p} 
\bigr\rangle
\cdot dW_{s}
\biggr\vert^2
\right]
&\leq C \sum_{p \geq 1} \sum_{i=1}^d  
\EE \biggl[ \int_{0}^T
\bigl\langle \overline{S}^n_{s}, D_{x_i} \psi_{p}(\cdot) \bigr\rangle^2
ds \biggr]
\\
&\leq C   
\sum_{i=1}^d
\EE \left[\int_{0}^T\| D_{x_i}^* \overline{S}_{s}^n \|_{-(2+2\newD ),\newD }^2 ds\right]
\\
&\leq C   
\sum_{i=1}^d
\EE \left[\int_{0}^T\|  \overline{S}_{s}^n \|_{-(1+2\newD ),\newD }^2 ds\right]
\end{split}
\end{equation*}   
which completes the proof of 
\eqref{eq:bound:4} since ${\mathcal H}^{-(1+\newD ),2\newD }$ embeds (continuously) in 
${\mathcal H}^{-(1+2\newD ),\newD }$. 
We then complete the proof with an easy adaptation of 
\cite[Proposition 5.10]{meleard} and 
\cite[Theorem 5.12]{meleard}.
\end{proof}

\subsection{Identification of limit points}

We now identify the (weak) limit points of the sequence $(\overline S^n,M^n)_{n \geq 1}$. 

\subsubsection{Limit of $(M^n)_{n \geq 1}$} The first step is to identify the limit (in law) of the sequence $(M^n)_{n \geq 1}$. We follow the proof of \cite[Theorem 5.14]{meleard}. The only difference with the framework investigated in \cite{meleard} comes from the fact that the limit process $(\mu_{t})_{t \in [0,T]}$ is random, which prompts us to work conditionally its realization.

To state the result, we first construct, on the same probability space as the one used to define the mean field game, a random element  $(\theta_{0},\xi)$ of $\cH^{-(2+2\newD ),\newD } \times  C([0,T],{\mathcal H}^{-(2+2\newD ),\newD })$ which, as in Subsection 
\ref{se:CLT-statements}, satisfies the following:
$\theta_{0}$ and $(W,\mu,\xi)$ are independent,
$\theta_{0}$
is a Gaussian random variable with covariance
\begin{equation}
\label{eq:clt:covariance:limit:initial:condition}
\forall \phi_{1}, \phi_{2} \in \cH^{2+2\newD ,\newD },
\quad
\EE \bigl[ 
\theta_{0}(\phi_{1}) \theta_{0}(\phi_{2})
\bigr] = \langle \mu_{0}, \phi_{1} \phi_{2}
\rangle,
\end{equation}
and $\xi$ is, conditional on $(W,\mu)$,
a continuous Gaussian process with covariance
\begin{equation}
\label{eq:clt:covariance:limit}
\begin{split}
\forall \phi_{1},\phi_{2} \in \cH^{2+2\newD ,\newD },
\ &\forall s,t \in [0,T], \quad
\\
&{\mathbb E} \bigl[ \xi_{t}(\phi_{1}) \xi_{s}(\phi_{2}) \, \vert \, W,\mu \bigr] 
=
 \int_{0}^{s \wedge t}
\bigl\langle \mu_{r},
\sigma \sigma^\top D_x \phi_{1} \cdot 
D_x \phi_{2}
 \bigr\rangle dr.
\end{split}
\end{equation}
Notice
from 
\eqref{eq:embedding:1}
 that, for all $\phi \in \cH^{1+2\newD ,\newD }$,  
the function $\RR^d \ni x \mapsto \phi(x)/(1+ \vert x \vert^\newD )$ is bounded. Since $p'$ in Assumption 
{\bf A} is greater than $2\newD $, we have 
$\EE[ \sup_{t \in [0,T]} \langle \mu_{t},\phi^2 \rangle] < \infty$, which shows that the above covariances make sense. Recall also that $\mu$ is $W$-measurable, so that conditioning on $\mu$ in \eqref{eq:clt:covariance:limit} is redundant. 
We claim:
\begin{lemma} 
\label{lem:weak:input:cv}
The sequence $(W,\mu,\overline S^n_{0},M^n)_{n \geq 1}$ converges in law to $(W,\mu,\theta_{0},\xi)$ on the space 
$ C([0,T],\R^{d_0}) \times  C([0,T],\cP^{1}(\R^d))
\times \cH^{-(2+2\newD ),\newD }
\times  C([0,T],{\mathcal H}^{-(2+2\newD ),\newD })$.
\end{lemma}

\begin{proof}
Because the first two components $(W,\mu)$ are fixed, it follows from Proposition \ref{thm:clt:tightness} that the sequence in question is tight. Hence, we characterize the limit points.

We first observe that, for any $n \geq 1$, $(W,\mu)$ is $W$-measurable and $\overline{S}^n_{0}$ is measurable with respect to $(X_{0}^1,\ldots,X_{0}^N)$. As a consequence, $(W,\mu)$ and $\overline{S}^n_{0}$ are independent and,
for any weak limit $(\varpi,\nu,\vartheta_{0})$ 
of $(W,\mu,\overline{S}^n_{0})_{n \geq 1}$, 
$(\varpi,\nu)$ and $\vartheta_{0}$ are also independent. For any 
$\phi \in \cH^{2+2\newD ,\newD }$, the law of $\vartheta_{0}(\phi)$ is easily identified by means of the standard version of the central limit theorem. As a byproduct, we easily identify the law of $\vartheta_{0}$, regarded as a random variable with 
values in $\cH^{-(2+2\newD ),\newD }$, with the law of $\theta_{0}$.

We finally show that any weak limit 
$(\varpi,\nu,\vartheta_{0},\zeta)$
of the quadruple $(W,\mu,\overline{S}^n_{0},M^n)_{n \geq 1}$ coincides in law with 
 $(W,\mu,\theta_{0},\xi)$.
The proof is as follows. Fix $\phi,\psi \in \cH^{2+2\newD ,\newD }$. We know that, for any 
$n \geq 1$, $M^n(\phi)$ can be written as
\begin{equation*}
M^n_{t}(\phi) = \int_{0}^t \frac1{\sqrt{n}} \sum_{i=1}^n 
\bigl(
\sigma^{\top} D_x \phi(\overline X_{s}^i) \bigr) \cdot dB_{s}^i, \quad t \in [0,T].
\end{equation*}
As $(W,\mu,\overline{S}^n_{0})$ is independent of $B^1,\ldots,B^n$, we easily deduce that, conditional on $(W,\mu,\overline{S}^n_{0})$, the three processes $(M^n_t(\phi))_{t \in [0,T]}$, $(M^n_t(\psi))_{t \in [0,T]}$,
and
\begin{align}
M^n_t(\phi)M^n_t(\psi) - \int_{0}^{t} \langle m^n_{\overline{\boldsymbol X}_s}, \sigma\sigma^\top D_x  \phi \cdot D_x\psi \rangle ds, \quad t \in [0,T], \label{def:mtg-crossterm}
\end{align}
are martingales. That is, for any real-valued bounded continuous functions $h_{1}$ on $ C([0,T],\RR^{d_0})
\times  C([0,T],{\mathcal P}^{2}(\RR^d))
\times \cH^{-(2+2\newD ),\newD }$ and 
$h_{2}$ on $ C([0,T],\RR)$ and for any $0 \leq s \leq t \leq T$,
\begin{equation*}
\EE \Bigl[  h_{1} \bigl(W,\mu,\overline{S}^n_{0}\bigr) h_{2}\bigl(M^n_{\cdot \wedge s}(\phi)\bigr) \Bigl(M^n_{t}(\phi) - M^n_{s}(\phi)\Bigr) \Bigr] = 0,
\end{equation*}
and similarly for the other two martingales. Letting $n$ tend to $\infty$, we get
\begin{equation*}
\EE \Bigl[  h_{1} (\varpi,\nu,\theta_{0}) h_{2}\bigl(\zeta_{\cdot \wedge s}(\phi)\bigr) \Bigl(\zeta_{t}(\phi)-\zeta_{s}(\phi)\Bigr)\Bigr] = 0.
\end{equation*}
Arguing similarly for the other two processes, we deduce that, conditional on $(\varpi,\nu,\vartheta_{0})$, the processes $(\zeta_t(\phi))_{t \in [0,T]}$, $(\zeta_t(\psi))_{t \in [0,T]}$, and (the limit of \eqref{def:mtg-crossterm})
\begin{align*}
\zeta_t(\phi)\zeta_t(\psi) - \int_{0}^{t} \langle \nu_s, \sigma\sigma^\top D_x  \phi \cdot D_x\psi \rangle ds, \quad t \in [0,T]
\end{align*}
are martingales. So, conditional on $(\varpi,\nu,\vartheta_0)$, $(\zeta_t(\phi))_{t \in [0,T]}$ and $(\zeta_t(\psi))_{t \in [0,T]}$ are Gaussian with covariation process $\bigl(\int_{0}^t\langle \nu_{s},\vert \sigma\sigma^\top D_x \phi \cdot D_x \psi\vert^2 \rangle ds\bigr)_{t \in [0,T]}$. Since this covariation process is $\nu$-measurable, and since $(\varpi,\nu)$ and $\vartheta_{0}$ are independent, we deduce in particular that  $(\varpi,\nu,\zeta)$ and $\vartheta_{0}$ are independent. 
Moreover, we have found that, conditionally on $(\varpi,\nu)$, $\zeta$ is a Gaussian process on $H^{-(2+2\newD ),\newD }$ with covariance
\begin{equation}
\label{eq:covariance:nu:zeta}
\begin{split}
\forall &\phi_{1},\phi_{2} \in \cH^{2+2\newD ,\newD },
\ \forall s,t \in [0,T], 
\\
&{\mathbb E} \bigl[ \zeta_{t}(\phi_{1}) \zeta_{s}(\phi_{2}) \, \vert \, \varpi,\nu \bigr] =
 \int_{0}^{s \wedge t}
\bigl\langle \nu_{r}, 
\sigma \sigma^\top
D_x \phi_{1}
\cdot
D_x \phi_{2} \bigr\rangle dr.
\end{split}
 \end{equation}
Since $(W,\mu)$ and $(\varpi,\nu)$ have the same distribution, this shows that $(\varpi,\nu,\zeta)$ has the same law as $(W,\mu,\xi)$. 
\end{proof}

\subsubsection{Limit of $(m^n_{\overline{\boldsymbol X}},W,\overline{S}^n)_{n \geq 1}$}
We claim:
\begin{lemma}
Let $(\nu,\nu',\varpi,\zeta,S)$ be (weak) limit point of $(m^n_{\overline{\boldsymbol X}},\mu,W,M^n,\overline{S}^n)_{n \ge 1}$. Then, with probability 1,  we have $\nu'=\nu$, and 
$S$ solves the equation:
\begin{equation}
\label{eq:lim:spde:for:clt}
\begin{split}
d \bigl\langle S_{t},\phi
\bigr\rangle 
&= \Bigl[
    \bigl\langle
 S_{t},
 D_x \phi(\cdot) \cdot \widetilde b(t,\cdot,\nu_{t}) \bigr\rangle
 +  \bigl\langle 
 S_{t},\bigl[{\mathcal G}\bigl(t,\nu_{t},\nu_{t}\bigr) \phi \bigr](\cdot)
 \bigr\rangle \Bigr] dt
 \\ 
&\hspace{15pt} + \tfrac12 \bigl\langle S_{t},
\textrm{\rm Tr}[ ( \sigma \sigma^\top 
+ \sigma_{0} \sigma_{0}^\top ) D^2_x \phi ] \bigr\rangle  dt +
\bigl\langle S_{t}, \sigma_{0}^\top D_x \phi 
\bigr\rangle \cdot d \varpi_{t}
+ d \zeta_{t}(\phi),
\end{split}
\end{equation} 
for $t \in [0,T]$ and $\phi \in \cH^{2\newD +4,\newD }$, with 
${\mathcal G}$ as in 
 \eqref{eq:def:cal:G}. 
\end{lemma}

\begin{proof}
The fact that $\nu=\nu'$ is a mere consequence of the fact that $\sup_{t \in [0,T]}W_{2}(m^n_{\bm{\overline{X}}_t},\mu_t)$ tends to $0$ with probability 1.

In order to handle the second part of the statement, we come back to the
expansion \eqref{eq:expansion:1},
see also \eqref{eq:expansion:2}
and \eqref{eq:expansion:3}. 
We rewrite the expansion under the form
\begin{equation*}
\begin{split}
d\bigl\langle \overline{S}^n_{t},\phi \bigr\rangle 
&=  
\Bigl[
    \bigl\langle
\overline{S}^n_{t},
 D_x \phi(\cdot) \cdot \widetilde b(t,\cdot,m^n_{\overline{\boldsymbol X}_{t}}) \bigr\rangle
 +  \bigl\langle 
\overline{S}^n_{t},\bigl[{\mathcal G}\bigl(t,m^n_{\overline{\boldsymbol X}_{t}},\mu_{t}\bigr) \phi \bigr](\cdot)
 \bigr\rangle \Bigr] dt
 \\ 
&\hspace{15pt} +
 \tfrac12 \bigl\langle \overline S^n_{t},  
\mathrm{Tr}[(\sigma\sigma^\top+\sigma_0\sigma_0^\top)D^2_x\phi]
   \bigr\rangle dt 
 +
\bigl\langle \overline S^n_{t}, \sigma_{0}^\top D_x \phi
\rangle \cdot dW_{t}
+ dM^n_{t}(\phi), \quad t \in [0,T].  
\end{split}
\end{equation*}
We take $\phi$ in the space $\cH^{4+2\newD ,\newD }$ as in required in the statement. {Under the assumption of 
Theorem 
\ref{th:CLT:sec:3}, $\widetilde b$ is differentiable in $x$
up to the order $2+2\newD $, with bounded and continuous derivatives. Similarly, 
$\delta \widetilde b/\delta m$ is differentiable in 
the variable $v$ up the order $2+2\newD $, with joint continuous (in all the arguments) and bounded derivatives.} In particular, 
$\cG(t,m^n_{\overline{\boldsymbol X}_{t}},\mu_{t}) \phi$ has bounded derivatives 
up to the order $2+2\newD $. 
The key point is that a function with derivatives up the order $2+2\newD $ that are at bounded is in 
$\cH^{2+2\newD ,\newD }$, which follows from the fact that $\RR^d \ni x \mapsto 
1/(1+\vert x\vert^{2\newD })$ is integrable since $2\newD  \geq d+1$. Hence, 
$\cG(t,m^n_{\overline{\boldsymbol X}_{t}},\mu_{t}) \phi$ converges almost surely in 
$\cH^{2+2\newD ,\newD }$ to 
$\cG(t,\mu_{t},\mu_{t}) \phi$
and 
there is no real difficulty for passing to the limit in the above expansion.
We get 
\eqref{eq:lim:spde:for:clt}.

By a standard separability argument, \eqref{eq:lim:spde:for:clt} holds with probability 1 for all $\phi$ in the space $\cH^{4+2\newD ,\newD }$. 
\end{proof}

Observe that equation \eqref{eq:lim:spde:for:clt} may be regarded as an equation in the space $\cH^{-(4+2\newD ),\newD }$, while 
$S$ is known to take values in the smaller space 
$\cH^{-(2+2\newD ),\newD }$. 
In order to prove convergence of the sequence 
$(m^n_{\overline{\boldsymbol X}},W,M^n,\overline{S}^n)_{n \ge 1}$, one must prove uniqueness of the solution to this equation.

\subsubsection{Uniqueness to the limiting equation} \label{se:CLTlimitinguniqueness}

We shall prove in Subsection \ref{subse:proof:!:spde} the following uniqueness result:

\begin{theorem}
\label{thm:uniqueness:spde}
Fix a given filtered probability space
$(\Xi,{\mathcal G},{\mathbb G},{\mathbb Q})$ equipped with an adapted triple 
$(\varpi,\nu,\zeta)$ such that 
$\varpi=(\varpi_{t})_{t \in [0,T]}$ is a $d$-dimensional Brownian motion, $\nu=(\nu_t)_{t \in [0,T]}$ is a $\P^{2\newD}(\R^d)$-valued process,
and, conditional on $(\varpi,\nu)$, $\zeta$ is a Gaussian process with the same covariance structure as in \eqref{eq:covariance:nu:zeta}. 
Then,  for any square integrable initial condition $S_{0}$ with values in $\cH^{-(2+2\newD ),\newD }$
and independent of $(\varpi,\nu,\zeta)$, the equation \eqref{eq:lim:spde:for:clt} has at most one (adapted) square-integrable solution with paths in $ C([0,T],\cH^{-(2+2\newD ),\newD })$. 
\end{theorem}

Invoking 
Lemma 
\ref{lem:weak:input:cv}
and recalling that 
 strong uniqueness implies weak uniqueness in law, we have the following corollary.

\begin{corollary}
\label{cor:lim:spde}
The sequence $(W,m^n_{\bm{\overline{X}}},\overline{S}^n)_{n \geq 1}$ converges in law on the space $ C([0,T],\R^d) \times 
 C([0,T],\cP^{1}(\R^d)) \times  C([0,T],
\cH^{-(2+2\newD ),\newD })$ to the tuple 
$(W,\mu,S)$ where $S$ solves, for any 
$\phi \in \cH^{4+2\newD ,\newD }$ the equation:
\begin{equation}
\label{eq:spde:cor:lim}
\begin{split}
d \bigl\langle S_{t},\phi
\bigr\rangle 
&= \Bigl[
    \bigl\langle
 S_{t},
 D_x \phi(\cdot) \cdot \widetilde b(t,\cdot,\mu_{t}) \bigr\rangle
 +  \bigl\langle 
 S_{t},\bigl[{\mathcal G}\bigl(t,\mu_{t},\mu_{t}\bigr) \phi \bigr](\cdot)
 \bigr\rangle \Bigr] dt
 \\ 
&\hspace{10pt} + \tfrac12 \bigl\langle S_{t},
\textrm{\rm Tr}[ ( \sigma \sigma^\top 
+ \sigma_{0} \sigma_{0}^\top ) D^2_x \phi ] \bigr\rangle  dt +
\bigl\langle S_{t}, \sigma_{0}^\top D_x \phi 
\bigr\rangle \cdot d {W_{t}}
+ d \xi_{t}(\phi),
\quad t \in [0,T],
\end{split}
\end{equation} 
where $\xi$ is given by 
\eqref{eq:clt:covariance:limit}
and where $S_{0}$ is independent 
of $(W,\mu,\xi)$ and has the same law as $\theta_{0}$ in
\eqref{eq:clt:covariance:limit:initial:condition}. 
\end{corollary}

\subsubsection{Reformulation of the limiting equation}
In fact, we can provide another representation
of the term $\xi$, based on the fact that, almost surely, for any $t>0$, $\mu_{t}$ has a density and may be identified with a function, which is the purpose of the following lemma.

\begin{lemma}
\label{lem:density}
With probability 1, for any $t>0$, the probability measure $\mu_{t}$ has a density with respect to the Lebesgue measure on $\R^d$. 
\end{lemma}

\begin{proof}
Let $Y_t = X_t -\sigma_0 W_t$. Then
\begin{align*}
dY_t = \widetilde{b}(t,Y_t + \sigma_0 W_t,\mu_t)dt + \sigma dB_t.
\end{align*}
Let $\widetilde\mu_t =\cL(Y_t \, | \, (W_s)_{s \le t})={\cL(Y_t \, | \, W)}$ denote the flow of conditional laws of $Y$ given $W$. Given $W$, the process $Y$ solves a deterministic SDE with uniformly bounded drift. By Girsanov's theorem, the law $\widetilde\mu_t$ has strictly positive density with respect to Lebesgue measure, almost surely. Hence, so too does $\mu_t = \cL(X_t \, | \, (W_s)_{s \le t})$, which is simply a translation of $\widetilde\mu_t$.
\end{proof}

We then have the following result.

\begin{proposition}
Let $(\beta_{t}:=(\beta^1_{t},\ldots,\beta^{d_{0}}_{t}))_{t \in [0,T]}$ be a $W$-independent $d_{0}$-tuple of independent cylindrical Wiener processes with values in $L^2(\R^d)$. Then, the collection 
\begin{equation*}
\biggl( \biggl(
 \int_{0}^t 
\bigl( \sqrt{\mu}_{s} \sigma^{\top} D_x \phi \bigr) \cdot d\beta_{s}\biggr)_{\phi \in \cH^{2+2\newD ,\newD }}
 \biggr)_{t \in [0,T]}
\end{equation*}
defines a process with values in $ C([0,T],\cH^{-(2+2\newD ),\newD })$ which has the same joint law with $(W,\mu)$ as $\xi$ in 
\eqref{eq:clt:covariance:limit}.
\end{proposition}

In other words, the above statement says that, in \eqref{eq:lim:spde:for:clt}, the term 
$\zeta_{t}(\phi)$ may be written as 
$\int_{0}^t (\sqrt{\mu}_{s} \sigma^{\top} D_x \phi) \cdot
d \beta_{s}$.

\begin{proof}
The proof follows from a standard computation of  covariance. To do so, it suffices to observe 
from 
\eqref{eq:embedding:1}
  that, 
for any $\phi \in \cH^{2+2\newD ,\newD }$, 
the function 
$\R^d \ni x \mapsto \vert D_x \phi(x) \vert /
(1+ \vert x \vert^\newD )$ is bounded. As a result, 
the function 
$(t,x) \mapsto \sqrt{\mu_{t}(x)} D_x \phi(x)$ is square-integrable, since $p'$ 
in Assumption \textbf{\bf A}
is greater than $2\newD $.
\end{proof}  

\subsection{Proof of the fluctuation theorem for MFG} \label{se:MFGfluctuationproof}

Postponing the proof of the uniqueness Theorem \ref{thm:uniqueness:spde}, we are now in position to the complete the proof of Theorem \ref{th:CLT:sec:3}. 

There are two steps. The first one is to observe that the limiting SPDE \eqref{eq:spde:cor:lim} may be rewritten in terms of  $\A$ in \eqref{eq:cA:clt}, simply by expanding the definition of $\mathcal{G}$. Hence, the SPDE \eqref{eq:spde:cor:lim} may be easily identified with \eqref{eq:spde:clt:intro}. 

Then, it remains to show that $(D^n_t := \overline{S}^n_t - S^n_t)_{t \in [0,T]}$ converges to zero in law on the space $C([0,T],\cH^{-(2+2\newD ),\newD })$. In fact, we will show that $\E[\sup_{t \in [0,T]}\|D^n_t\|_{-(2+2\newD ),\newD }]$ tends to $0$ as $n$ tends to $\infty$. Notice first that
\begin{equation*}
\begin{split}
\E\Bigl[\sup_{t \in [0,T]}\|D^n_t\|_{-(2+2\newD ),\newD }\Bigr]
&= \E\Bigl[\sup_{t \in [0,T]}
\sup_{\| \phi \|_{2+2\newD ,\newD } \leq 1}
\langle D^n_{t},\phi \rangle \Bigr]
\\
&\leq \frac1{\sqrt{n}} \E\biggl[\sup_{t \in [0,T]}\sup_{\| \phi \|_{2+2\newD ,\newD } \leq 1}
\biggl\vert \sum_{i=1}^n \phi(X_{t}^i) - \phi(\overline X_{t}^i) \biggr\vert \biggr].
\end{split}
\end{equation*}
For $\phi$ satisfying $\| \phi \|_{2+2\newD ,\newD } \leq 1$, 
we know from Sobolev's embeddings that
the function $\R^d \ni x \mapsto \vert D_x \phi(x) \vert / (1+ \vert x \vert^\newD )$ is bounded, independently of $\phi$. Therefore, there exists a constant $C$, independent of $n$, such that:
\begin{equation*}
\begin{split}
\E\Bigl[\sup_{t \in [0,T]}\|D^n_t\|_{-(2+2\newD ),\newD }\Bigr]
&\leq \frac{C}{\sqrt{n}} \E\biggl[\sup_{t \in [0,T]}
\biggl\vert \sum_{i=1}^n 
\Bigl( 1+ \vert X_{t}^i \vert^\newD  + 
\vert \overline X_{t}^i) \vert^\newD 
\Bigr) \bigl\vert X_{t}^i - \overline X_{t}^i \bigr\vert \biggr].
\end{split}
\end{equation*}
By the Cauchy-Schwarz inequality, we get:
\begin{equation*}
\begin{split}
&\E\Bigl[\sup_{t \in [0,T]}\|D^n_t\|_{-(2+2\newD ),\newD }\Bigr]
\\
&\hspace{15pt} \leq C \sqrt{n} \biggl(  \E\biggl[ \frac1n \sum_{i=1}^n 
\Bigl( 1+ \| X^i \|_{{\infty}}^{2\newD } + \| \overline X^i \|_{{\infty}}^{2\newD } \Bigr) \biggr] \biggr)^{1/2}
\biggl(  \E\biggl[ \frac1n \sum_{i=1}^n 
 \bigl\| X^i - \overline X^i \bigr\|_{{\infty}}^2 \biggr] \biggr)^{1/2}.
\end{split}
\end{equation*}
The first expectation is bounded uniformly in $n$, and we  may use \eqref{def:estimate-|X-Xbar|} to conclude that the right-hand side tends to $0$. \hfill\qed

\subsection{Proof of the uniqueness Theorem \ref{thm:uniqueness:spde}}
\label{subse:proof:!:spde}
 
In order to prove uniqueness of the solution to 
\eqref{eq:lim:spde:for:clt}, we use the same argument as in \cite{kurtz-xiong}. To do so, we
 take two solutions $S$ and $S'$ of the limiting equation, with the same initial condition and with square-integrable and continuous paths with values in 
$\cH^{-(2+2\newD ),\newD }$ and we consider an orthonormal basis $(f_{k})_{k \geq 1}$ of the space $\cH^{-(4+2\newD ),\newD }$. Letting 
$(\delta S_{t} = S_{t}- S_{t}')_{t \in [0, T]}$, we know that, for any $k \geq 1$, 
\begin{equation*}
\begin{split}
d \bigl\langle \delta S_{t},f_{k}
\bigr\rangle 
&= \Bigl[
    \bigl\langle
 \delta S_{t},
 D_x f_{k}(\cdot) \cdot \widetilde b(t,\cdot,\nu_{t}) \bigr\rangle
 +  \bigl\langle 
 \delta S_{t},\bigl[\cG(t,\nu_{t},\nu_{t}) f_{k} \bigr](\cdot)
 \bigr\rangle \Bigr] dt
 \\ 
&\hspace{15pt} +  \tfrac12 \bigl\langle \delta S_{t}, 
\textrm{\rm Tr}\bigl[ \bigl( \sigma \sigma^\top 
+ \sigma_{0} \sigma_{0}^\top \bigr) D^2_x  f_k   \bigr\rangle \bigr] \bigr\rangle dt + 
\bigl\langle \delta S_{t}, \sigma_{0}^\top D_x f_{k} 
\rangle \cdot d\varpi_{t}, \quad t \in [0,T],
\end{split}
\end{equation*} 
with $\delta S_{0}=0$. 
Then, recalling that $\varpi$ has dimension $d_0$,
\begin{equation}
\label{eq:uniqueness}
\begin{split}
d  
\bigl\langle \delta S_{t},f_{k}\bigr\rangle^2
&= 
2
\bigl\langle \delta S_{t},f_k\bigr\rangle
\Bigl(
    \bigl\langle
 \delta S_{t},
 D_x f_{k}(\cdot) \cdot \widetilde b(t,\cdot,\nu_{t}) \bigr\rangle
 +  \bigl\langle 
 \delta S_{t},\bigl[\cG(t,\nu_{t},\nu_{t}) f_{k} \bigr](\cdot)
 \bigr\rangle \Bigr) dt
\\
&\hspace{5pt} +
\bigl\langle \delta S_{t},f_{k}\bigr\rangle
\bigl\langle \delta S_{t}, \textrm{\rm Tr} \bigl[ \bigl( \sigma \sigma^\top +
\sigma_{0} \sigma_{0}^\top \bigr) 
 D^2_x f_{k} \bigr] \bigr\rangle  dt + 
\sum_{i=1}^{d_{0}}
\bigl\langle \delta S_{t},
\bigl( \sigma_{0}^\top D_x f_{k} \bigr)_{i}\bigr\rangle^2
 dt
\\
&\hspace{5pt} +
2
\bigl\langle \delta S_{t}, f_{k} 
\rangle
\bigl\langle \delta S_{t}, \sigma_{0}^\top D_x f_{k} 
\rangle \cdot d\varpi_{t}, \quad t \in [0,T]. 
\end{split}
\end{equation}
At this stage, the problem is the same as in \cite{kurtz-xiong}, namely the right hand side depends on higher derivatives of the basis function $f_{k}$. To handle these higher derivatives, we proceed as follows. 
\vskip 4pt
 
\noindent\textbf{Step 1:}
As a first step, we tackle the second order term in 
\eqref{eq:uniqueness}. 
Defining the operator
$\L_{0} = \textrm{\rm Tr}[ (\sigma \sigma^\top + 
\sigma_{0} \sigma_{0}^\top) D^2_x ( \cdot) ]$
and summing over $k \geq 1$, we get
\begin{equation*}
\begin{split}
&\sum_{k \geq 1} \bigl\langle \delta S_{t},f_{k}
\bigr\rangle \bigl\langle \delta S_{t},\L_{0}   f_{k} \bigr\rangle
= \bigl\langle \delta S_{t},\L_{0}^* \delta S_{t}
\bigr\rangle_{{-(4+2\newD ),\newD }}, 
\end{split}
\end{equation*}
where $\langle \cdot,\cdot \rangle_{{-(4+2\newD ),\newD }}$ stands for the 
inner product in $\cH^{-(4+2\newD ),\newD }$ and 
$\cL_{0}^* \delta S_{t}$ is the element of 
$\cH^{-(4+2\newD ),\newD }$ defined as 
$\langle \cL_{0}^* \delta S_{t},\phi \rangle 
= \langle  \delta S_{t},\cL_{0} \phi \rangle$
for all $\phi \in \cH^{4+2\newD ,\newD }$ (which makes sense since $\delta S_{t} \in \cH^{-(2+2\newD ),\newD }$).
We then make use of Riesz' representation theorem. For an element $\varrho$ in    
$\cH^{-(2+2\newD ),\newD } \subset \cH^{-(4+2\newD ),\newD }$, we call $\widetilde{\varrho}$ its representative element in 
$\cH^{4+2\newD ,\newD }$, namely
\begin{equation*}
\forall \phi \in \cH^{4+2\newD ,\newD }, \quad 
\langle \phi , 
\varrho \rangle = 
\bigl\langle \phi, \widetilde{\varrho} \bigr\rangle_{{4+2\newD ,\newD }},
\end{equation*}
where $\langle \cdot, \cdot \rangle$ denotes the duality bracket while $\langle \cdot,\cdot \rangle_{{4+2\newD ,\newD }}$ stands for the 
inner product in $\cH^{4+2\newD ,\newD }$.
 We then take an element $\varrho \in \cH^{-(2+2\newD ),\newD }$ such that $\widetilde{\varrho}$ belongs to $\cH^{6+2\newD ,\newD }$. 
 Since $\cL_{0}^* \varrho \in \cH^{-(4+2\newD ),\newD }$, we get
\begin{align*}
\bigl\langle \varrho,\cL_{0}^* \varrho \bigr\rangle_{{-(4+2\newD ),\newD }}
&=\bigl\langle \widetilde \varrho,\cL_{0}^* \varrho \bigr\rangle
= \bigl\langle \cL_{0} \widetilde \varrho, \varrho \bigr\rangle
= \bigl\langle \cL_{0} \widetilde \varrho,\widetilde \varrho  \bigr\rangle_{{4+2\newD ,\newD }} \\
&=  \sum_{i,j=1}^d 
(\sigma \sigma^\top
+
\sigma_{0} \sigma_{0}^\top
)_{i,j}
\sum_{\vert {\boldsymbol k} \vert \leq 4+2\newD } \int_{\R^d} \frac{D^{\boldsymbol k} \widetilde \varrho(x) 
D^2_{x_i,x_j} D^{\boldsymbol k} \widetilde \varrho(x) }{1+ \vert x \vert^{2 \newD }} dx
\\
&= - \sum_{i,j=1}^d 
(\sigma \sigma^\top
+
\sigma_{0} \sigma_{0}^\top
)_{i,j}
\sum_{\vert {\boldsymbol k} \vert \leq 4+2\newD } \int_{\R^d} \frac{D_{x_i}D^{\boldsymbol k} \widetilde \varrho(x) 
D_{x_j} D^{\boldsymbol k} \widetilde \varrho(x) }{1+ \vert x \vert^{2 \newD }} dx
\\
&\hspace{15pt} + \frac12 \sum_{{i,j}=1}^d 
(\sigma \sigma^\top
+
\sigma_{0} \sigma_{0}^\top
)_{i,j}
\sum_{\vert {\boldsymbol k} \vert \leq 4+2\newD } \int_{\R^d} 
D_{x_i,x_j}^2 \bigl( (1 + \vert x \vert^{2\newD })^{-1}
\bigr)  (D^{\boldsymbol k} \widetilde \varrho)^2  (x)  dx
\\
&= - \bigl( \| \sigma^\top D_x \widetilde{\varrho} \|^2_{{4+2\newD ,\newD }} 
+ \| \sigma_{0}^\top D_x \widetilde{\varrho} \|^2_{{4+2\newD ,\newD }}
\bigr)  
+ O \bigl(  
\| \widetilde{\varrho} \|^2_{{4+2\newD ,\newD }}
\bigr),
\end{align*}
where $\| \sigma^\top D_x \widetilde{\varrho} \|^2_{{4+2\newD ,\newD }}$ is a shortened notation for 
$\sum_{i=1}^{d} \| (\sigma^\top D_x\widetilde{\varrho})_{i}\|^2_{{4+2\newD ,\newD }}$ 
(and similarly for the term 
$\| \sigma^\top_{0} D_x \widetilde{\varrho} \|^2_{{4+2\newD ,\newD }}$)
and $O(\cdot)$ is the Landau notation.
\vskip 4pt

\noindent\textbf{Step 2:}
We now handle the quadratic term in 
\eqref{eq:uniqueness}. Proceeding as before, we have
\begin{equation*}
\sum_{k \geq 1}
\sum_{i=1}^{d_{0}}
\bigl\langle \delta S_{t}, \bigl( \sigma_{0}^\top 
D_x f_{k}\bigr)_{i}\bigr\rangle^2
= \sum_{i=1}^{d_{0}} \| (\sigma_{0}^\top D_x)_{i}^* \delta S_{t} \|^2_{{-(4+2\newD ),\newD }},
\end{equation*}
where $\langle (\sigma_{0}^\top D_x)_{i}^* \delta S_{t},\phi \rangle 
= \langle \delta S_{t}, (\sigma_{0}^\top D_x)_{i} \phi \rangle$ for 
all $\phi \in \cH^{4+2\newD ,\newD }$. 

Choose now $\varrho \in \cH^{-(2+2\newD ),\newD }$ such that 
$\widetilde \varrho \in \cH^{6+2\newD ,\newD }$. Then,
for any smooth function $\phi$ with compact support, 
\begin{equation*}
\begin{split}
\bigl\langle (\sigma_{0}^\top D_x)_{i}^* \varrho, \phi \bigr\rangle 
&= \bigl\langle \widetilde{\varrho}, (\sigma_{0}^\top D_x)_{i}
\phi \bigr\rangle_{{4+2\newD ,\newD }}
\\
&=
\sum_{\vert {\boldsymbol k} \vert \leq 4+2\newD } \int_{\R^d} \frac{D^{\boldsymbol k} \widetilde \varrho(x) 
(\sigma_{0}^\top D_x)_{i} D^{\boldsymbol k} \phi(x) }{1+ \vert x \vert^{2 \newD }} dx
\\
&= - \sum_{\vert {\boldsymbol k} \vert \leq 4+2\newD } \int_{\R^d} \frac{D^{\boldsymbol k} [(\sigma_{0}^\top D_x)_{i} \widetilde \varrho](x) 
D^{\boldsymbol k} \phi(x) }{1+ \vert x \vert^{2 \newD }} dx 
\\
&\hspace{15pt}
- 
\sum_{\vert {\boldsymbol k} \vert \leq 4+2\newD } \int_{\R^d} 
(\sigma_{0}^\top D_x)_{i} \bigl( (1+ \vert x \vert^{2 \newD })^{-1} \bigr) 
D^{\boldsymbol k}  \widetilde \varrho(x) 
D^{\boldsymbol k} \phi(x) 
dx.
\end{split}
\end{equation*}
Therefore,
\begin{equation*}
\bigl\langle (\sigma_{0}^\top D_x)_{i}^* \varrho, \phi \bigr\rangle
= - \bigl\langle (\sigma_{0}^\top D_x)_{i} \widetilde \varrho, \phi 
\bigr\rangle_{{4+2\newD ,\newD }}
+ O \bigl( \| \widetilde \varrho \|_{{4+2\newD ,\newD }} \| \phi \|_{{4+2\newD ,\newD }}
\bigr), 
\end{equation*}
which shows that $(\sigma_{0}^\top D_x)_{i}^* \varrho$ belongs to 
$\cH^{-(4+2\newD ),\newD }$ and that the above equality holds for any $\phi \in \cH^{4+2\newD ,\newD }$. Hence,
\begin{equation}
\label{eq:2:ibp:clt}
\begin{split}
&\| (\sigma_{0}^\top D_x)_{i}^* \varrho\|_{{-(4+2\newD ),\newD }}
\\
&= \sup_{\| \phi \|_{{4+2\newD ,\newD }} \leq 1}
\Bigl[ - \bigl\langle (\sigma_{0}^\top D_x)_{i} \widetilde \varrho, \phi 
\bigr\rangle_{{4+2\newD ,\newD }}
+ O \bigl( \| \widetilde \varrho \|_{{4+2\newD ,\newD }} \| \phi \|_{{4+2\newD ,\newD }}
\bigr) \Bigr]
\\
&= \| (\sigma_{0}^\top D_x)_{i} \widetilde \varrho\|_{{4+2\newD ,\newD }}
+ O \bigl( \| \widetilde \varrho \|_{{4+2\newD ,\newD }}
\bigr). 
\end{split}
\end{equation}

\noindent\textbf{Step 3:} Lastly, we handle the first term in the right-hand side of 
\eqref{eq:uniqueness}. We have
\begin{equation*}
\begin{split}
\sum_{k \geq 1}
\bigl\langle  \delta S_{t},f_{k}\bigr\rangle
    \bigl\langle
 \delta S_{t},
 D_x f_{k}(\cdot) \cdot \widetilde b(t,\cdot,\nu_{t}) \bigr\rangle
&= 
\sum_{i=1}^d
\sum_{k \geq 1}
\bigl\langle  \delta S_{t},f_{k}\bigr\rangle
    \bigl\langle D_{x_i}^*
    \bigl( \Pi_{\widetilde b_{i}(t,\cdot,\nu_{t})}
    \delta S_{t} \bigr), 
 f_{k}
 \bigr\rangle
 \\
& = 
\sum_{i=1}^d
\bigl\langle  \delta S_{t},D_{x_i}^*
    \bigl( \Pi_{\widetilde b_{i}(t,\cdot,\nu_{t})}
    \delta S_{t} \bigr)
 \bigr\rangle , 
\end{split}
\end{equation*}
where, for $\varrho \in \cH^{-(2+2\newD ),\newD }$, $\Pi_{\widetilde b_{i}(t,\cdot,\nu_{t})} \varrho$ is the element of $\cH^{-(4+2\newD ),\newD }$ defined by
\begin{equation*}
\begin{split}
\bigl\langle \Pi_{\widetilde b_{i}(t,\cdot,\nu_{t})} \varrho,
\phi \bigr\rangle 
&= \bigl\langle \varrho, \widetilde b_{i}(t,\cdot,\nu_{t})
\phi \bigr\rangle, \quad \text{ for } \phi \in \cH^{4+2\newD ,\newD }.
\end{split}
\end{equation*}

As above, we investigate the above term 
when $\varrho \in \cH^{-(2+2\newD ),\newD }$
satisfies $\widetilde \varrho \in \cH^{6+2\newD ,\newD }$. 
Since the derivatives of $\widetilde b_{i}$ are bounded, it is  easily verified that $\widetilde b_{i}(t,\cdot,\nu_{t})
\phi$ belongs to $\cH^{4+2\newD ,\newD }$ which guarantees that the above right-hand side is well defined. In fact, when $\phi \in \cH^{5+2\newD ,\newD }$, 
\begin{equation*}
\begin{split}
\bigl\langle \Pi_{\widetilde b_{i}(t,\cdot,\nu_{t})} \varrho,
D_{x_i}\phi \bigr\rangle 
&= \bigl\langle \widetilde \varrho, \widetilde b_{i}(t,\cdot,\nu_{t})
D_{x_i} \phi \bigr\rangle_{{4+2\newD ,\newD }}.
\end{split}
\end{equation*}
Since the derivatives of $\widetilde b_{i}$ up to the order $5+2\newD $ are bounded, the right-hand side is bounded by $O( \| D_{x_i} \widetilde{\varrho} \|_{{4+2\newD ,\newD }} + \| \widetilde{\varrho} \|_{{4+2\newD ,\newD }}) \| \phi \|_{{4+2\newD ,\newD }}$. In particular, $D_{x_i}^* ( \Pi_{\widetilde b_{i}(t,\cdot,\nu_{t})} \varrho)$ belongs to the space 
$\cH^{-(4+2\newD ),\newD }$ and 
\begin{equation*}
\bigl\| 
D_{x_i}^* ( \Pi_{\widetilde b_{i}(t,\cdot,\nu_{t})} \varrho) \bigr\|_{{-(4+2\newD ),\newD }}
= O \bigl( \| D_{x_i} \widetilde{\varrho} \|_{{4+2\newD ,\newD }} + \| \widetilde{\varrho} \|_{{4+2\newD ,\newD }}\bigr). 
\end{equation*}
Then,
\begin{equation}
\label{eq:3:ibp:clt}
\begin{split}
 \Bigl\vert
\bigl\langle  \varrho, 
D_{x_i}^*
\bigl( \Pi_{\widetilde{b}_{i}(t,\cdot,\nu_{t})} \varrho\bigr)
\bigr\rangle_{\cH^{-(4+2\newD ),2\newD }}
\Bigr\vert
&= O \bigl( \| D_{x_i} \widetilde{\varrho} \|_{{4+2\newD ,\newD }}\| \widetilde{\varrho} \|_{{4+2\newD ,\newD }} + \| \widetilde{\varrho} \|_{{4+2\newD ,\newD }}^2\bigr)
. 
\end{split}
\end{equation}
\vskip 4pt

\noindent\textbf{Step 4:}
We next consolidate the first three steps into an estimate on the $dt$ term of the right-hand side of \eqref{eq:uniqueness}, or rather the sum thereof over $k \in \N$, which can be rewritten as ${\mathcal F}(\delta S_{t})$, where
\begin{equation*}
\begin{split}
{\mathcal F}(\varrho) &=   \Bigr[2
\sum_{i=1}^d
\bigl\langle \varrho ,D_{x_i}^* \bigl( \Pi_{\widetilde{b}_{i}(t,\cdot,\nu_{t})}
\varrho \bigr) \bigr\rangle_{{-(4+2\newD ),\newD }}
\\
&\hspace{15pt}
+ 2 \bigl\langle \varrho,\cG(t,\nu_{t},\nu_{t})^* \varrho\bigr\rangle_{{-(4+2\newD ),\newD }}
+  \bigl\langle \varrho,\cL_{0}^* \varrho \bigr\rangle_{{-(4+2\newD ),\newD }} \Bigr] 
 +
\sum_{i=1}^{d_{0}} \|
(\sigma_{0}^\top D_x)_{i}^*
 \varrho \|^2_{{-(4+2\newD ),\newD }}. 
\end{split}
\end{equation*}
Here, following the proof of 
\eqref{eq:conclusion:3rdstep}, 
$\cG(t,\nu_{t},\nu_{t})^* \varrho$
is in $\cH^{-(4+2\newD ),\newD }$ and 
$\| \cG(t,\nu_{t},\nu_{t})^* \varrho\|_{-(4+2\newD ),\newD } \leq c 
( 1 +\int_{\RR^d} \vert y \vert^{2\newD } d\nu_{t}(y))
\|  \varrho\|_{-(4+2\newD ),\newD }$. Returning to 
\eqref{eq:uniqueness} and
 collecting 
\eqref{eq:2:ibp:clt}
and
\eqref{eq:3:ibp:clt}, the second point is that 
${\mathcal F}(\varrho)$ 
satisfies, for $\varrho \in \cH^{-(2+2\newD ),\newD }$
such that $\widetilde \varrho \in \cH^{6+2\newD ,\newD }$,
\begin{equation*}
\begin{split}
{\mathcal F}(\varrho) &= - \| \sigma^\top D_x \widetilde \varrho\|^2_{{4+2\newD ,\newD }}
+ O \bigl( \| \widetilde \varrho \|_{{4+2\newD ,\newD }}
\| D_x \widetilde \varrho\|_{{4+2\newD ,\newD }}
+
\| \widetilde \varrho \|_{{4+2\newD ,\newD }}^2
\bigr)
\\
&\hspace{30pt} + 2  \bigl\langle 
 \varrho,   \cG(s,\nu_{s},\nu_{s})^* \varrho 
 \bigr\rangle_{{-(4+2\newD ),\newD }}. 
\end{split}
\end{equation*}
Using the fact that $\sigma$ is non-degenerate
together with the corresponding form 
of 
\eqref{eq:conclusion:3rdstep}, this yields
\begin{equation}
\label{eq:conclusion:4thstep}
\begin{split}
&{\mathcal F}(\rho) \leq C   \| \varrho \|_{{-(4+2\newD ),\newD }}^2 + 
C \biggl[
\| \varrho \|_{{-(4+2\newD ),\newD }}^2
\biggl( \int_{\RR^d} (1+ |y|^{2\newD }) d\nu_{s}(y) 
\biggr)^{1/2}
  \biggr]. 
\end{split}
\end{equation}
\vskip 4pt

\noindent\textbf{Step 5:} We now consider the general case when $\varrho \in  \cH^{-(2+2\newD ),\newD } \subset \cH^{-(4+2\newD ),\newD }$ and 
$\widetilde{\varrho}$ just belongs to $\cH^{4+2\newD ,\newD }$. Let $(\eta_{n})_{n \geq 1}$ be a sequence of even mollifiers, and let
$(\varrho_{n})_{n \geq 1}$ be the sequence of elements in $\cH^{-(2+2\newD ),\newD }$ defined by
\begin{equation*}
\langle \varrho_{n},\phi \rangle 
= \bigl\langle \varrho, \phi * \eta_{n}
\bigr\rangle, 
\end{equation*}
where $*$ denotes the usual convolution product.
Let $\widehat \varrho$ and $\widehat \varrho_{n}$ denote the elements of $\cH^{2+2\newD ,\newD }$ such that, for all $\phi \in \cH^{2+2\newD ,\newD }$,
\begin{equation*}
\bigl\langle 
\widehat{\varrho},\phi
\bigr\rangle_{\cH^{2+2\newD ,\newD }}
= \langle 
{\varrho},\phi
\rangle, \quad 
\bigl\langle 
\widehat{\varrho}_{n},\phi
\bigr\rangle_{{2+2\newD ,\newD }}
= \langle 
{\varrho}_{n},\phi
\rangle.
\end{equation*}
Then,
\begin{equation*}
\langle \rho_{n}, \phi \rangle
= \langle  \varrho, \phi * \eta_{n} \rangle
= \bigl\langle \widehat \varrho , 
\phi * \eta_{n} \bigr\rangle_{{2+2\newD ,\newD }}.
\end{equation*}
We observe that 
\begin{equation*}
\begin{split}
\|\phi * \eta_{n} \|_{{2+2\newD ,\newD }}^2
&=
\sum_{\vert {\boldsymbol k} \vert 
\leq 2 + 2\newD }
\int_{\R^d} \frac{\vert D^{\boldsymbol k} 
(\phi * \eta_{n})(x) \vert^2}{1+ \vert x \vert^{2 \newD }} dx
\\
&= \sum_{\vert {\boldsymbol k} \vert 
\leq 2 + 2\newD }
\int_{\R^d} \frac{\vert (D^{\boldsymbol k} 
\phi * \eta_{n})(x) \vert^2}{1+ \vert x \vert^{2 \newD }} dx
\\
&\leq \sum_{\vert {\boldsymbol k} \vert 
\leq 2 + 2\newD }
\int_{\R^d} \frac{1}{1+ \vert x \vert^{2\newD }}
\biggl( \int_{\R^d}
\vert D^{\boldsymbol k} \phi
(y) \vert^2 \eta_{n}(x-y) dy 
\biggr) dx
\\
&= \sum_{\vert {\boldsymbol k} \vert 
\leq 2 + 2\newD }
\int_{\R^d} \frac{\vert 
 D^{\boldsymbol k} \phi
(y)
\vert^2}{1+ \vert y \vert^{2\newD }}
\biggl( \int_{\R^d}
\frac{\eta_{n}(x-y) (1 + \vert y \vert^{2\newD })}{1+ \vert x \vert^{2\newD }} dx 
\biggr) dy \\
	&\leq C \| \phi \|^2_{\cH^{2+2\newD ,\newD }},
\end{split}
\end{equation*}
for a constant $C$ independent of $n$. For a given $\epsilon >0$, we choose 
$\widehat{\varrho}'$ in $ C^{\infty}_{c}(\RR^d)$ such that 
$\| \widehat{\varrho}' - 
\widehat{\varrho} \|_{\cH^{2+2\newD ,\newD }}
\leq \epsilon$ and we get 
\begin{equation}
\label{eq:widehat:rho:eps:n}
\bigl\vert \langle \widehat\varrho',\phi * \eta_{n} \rangle_{2+2\newD ,\newD } - 
\bigl\langle
\widehat{\varrho},\phi * \eta_{n}\bigr\rangle_{2+2\newD ,\newD }
\bigr\vert \leq C \epsilon \| \phi \|_{\cH^{2+2\newD ,\newD }}. 
\end{equation}
Now,
\begin{equation*}
\begin{split}
&\bigl\langle
\widehat{\varrho}',\phi * \eta_{n} - \phi\bigr\rangle 
\\
&= 
\sum_{\vert {\boldsymbol k} \vert \leq 4+2\newD } \int_{\R^d}
\frac{1}{1+ \vert y \vert^{2\newD }}
D^{\boldsymbol k} \phi(y) 
\biggl( \int_{\RR^d}
\eta_{n}(y-x)
\Bigl[
\frac{1 + \vert y \vert^{2\newD }}{1 + \vert x \vert^{2\newD }}
 D^{\boldsymbol k} \widehat{\varrho}'(x)
-
 D^{\boldsymbol k} \widehat{\varrho}'(y)\Bigr]dx 
 \biggr) dy.
\end{split}
\end{equation*}
Since $\widehat{\rho}'$ has a compact support, we deduce that there exists a constant $C_{\epsilon}$ such that 
\begin{equation}
\label{eq:widehat:rho:eps:n:2}
\bigl\vert \bigl\langle
\widehat{\varrho}',\phi * \eta_{n} - \phi\bigr\rangle 
\bigr\vert \leq \frac{C_{\epsilon}}{n}
\| \phi \|_{2+2\newD ,\newD }.
\end{equation}
By 
\eqref{eq:widehat:rho:eps:n}
and
\eqref{eq:widehat:rho:eps:n:2}, 
we deduce that $\| \varrho_{n} -  \varrho \|^2_{{-(2+2\newD ),\newD }} \rightarrow 0$ as
$n \rightarrow \infty$. As a consequence, we have 
$\| \widetilde \varrho_{n} - \widetilde \varrho \|^2_{{4+2\newD ,\newD }} \rightarrow 0$ as
$n \rightarrow \infty$. And then, 
$\| D^2_{x_i,x_j} \widetilde \varrho_{n} - D^2_{x_i,x_j}   \widetilde \varrho \|^2_{{2+2\newD ,\newD }} \rightarrow 0$ as
$n \rightarrow \infty$, for all $i,j \in \{1,\ldots,d\}$. Also, 
$\| D_{x_i}^* \varrho_{n} -  D_{x_i}^* \varrho \|^2_{{-(4+2\newD ),{\newD }}} \rightarrow 0$ and 
$\| D_{x_i}^* 
( \Pi_{\widetilde{b}_{i}(t,\cdot,\nu_{t})}
\varrho_{n})
 -  D_{x_i}^* 
 ( \Pi_{\widetilde{b}_{i}(t,\cdot,\nu_{t})}
\varrho) \|^2_{{-(4+2\newD ),{\newD }}} \rightarrow 0$
 as
$n \rightarrow \infty$.

When $\varrho$ is replaced by $\varrho_{n}$, 
 \eqref{eq:conclusion:4thstep} holds true. Passing to the limit over $n \rightarrow \infty$, 
 we deduce that 
 \eqref{eq:conclusion:4thstep} remains true for a general $\varrho$. 
 \vskip 4pt

\noindent\textbf{Conclusion of proof:} We finally return to \eqref{eq:uniqueness}, using \eqref{eq:conclusion:4thstep} to get
\begin{equation*}
\begin{split}
d  \bigl[
\| \delta S_{t} \|_{{-(4+2\newD ),\newD }}^2
\bigr]
&\leq C   \| \delta S_{t} \|_{{-(4+2\newD ),\newD }}^2 \left[ 1 + \biggl(\int_{\RR^d} |y|^{2\newD } d\nu_{t}(y) 
\biggr)^{1/2}  \right] dt
\\
&\hspace{15pt} +
\sum_{i=1}^d 
\bigl\langle \delta S_{t}, (\sigma_{0}^\top D_x)^*_{i} \delta S_{t}\bigr\rangle \cdot d\varpi_{t}^i.
\end{split}
\end{equation*}
By expanding 
$d  [
\| \delta S_{t} \|_{{-(4+2\newD ),\newD }}^2
\exp ( - C ( t + \int_{0}^t 
(\int_{\RR^d}|y|^{2\newD } d\nu_{s}(y) 
)^{1/2}ds))
]$, we get
\begin{equation*}
\begin{split}
{\mathbb E}
\biggl[
\exp \biggl( - C \biggl( t + \int_{0}^t 
\biggl(\int_{\RR^d} |y|^{2\newD } d\nu_{s}(y) 
\biggr)^{1/2}ds\biggr) \biggr)
\| \delta S_{t} \|_{{-(4+2\newD ),\newD }}^2
\biggr] = 0, \quad t \in [0,T]. 
\end{split}
\end{equation*}
\hfill\qed

\subsection{Auxiliary lemma: Rates of convergence for smooth functions of an empirical sample} 
\label{se:empiricalconvergencerate}

We prove here the remarkable fact that the rate of convergence 
of the empirical distribution of a sample of i.i.d. random vectors with values $\RR^d$
can be made  
dimension-free when it is tested against test functions that are sufficiently smooth.
This is in contrast with existing results 
 in the literature that address the rate of convergence in the $p$-Wasserstein distance, such as \cite{fournier-guillin}.

\begin{lemma}
\label{lem:lln:l4:appendix}
Assume that $\phi$ is a $C^2$ function on $\cP^4(\RR^d)$, as defined in Section  
\ref{subse:derivatives:m:P2}, with bounded derivatives $\delta \phi/\delta m$ and 
$\delta^2 \phi/\delta m^2$, and that 
$\delta^2 \phi/\delta m^2$ is Lipschitz continuous in $m$ with respect to $\cW_{1}$, uniformly in the other variables. Then, there exists a constant $c$, depending on the bounds on the first and second-order derivatives
and on the Lipschitz constant of the second-order derivative such that, for any $m_0 \in \P^4(\R^d)$ and any i.i.d.\ sample $\bm{\zeta}=(\zeta_1,\ldots,\zeta_n)$ from $m_0$,
\begin{equation*}
{\mathbb E} \Bigl[ \bigl\vert \phi(m^n_{\bm{\zeta}}) - 
\phi(m_0) \bigr\vert^{4} \Bigr]^{1/4} \le c\left(1 + \left(\int_{\R^d}|x|\,m_0(dx)\right)^{1/4}\right)n^{-1/2}.
\end{equation*}
\end{lemma}

\begin{proof}
Since $\phi$ is differentiable, we have the following expansion:
\begin{eqnarray*}
\phi\bigl(m^n_{\bm{\zeta}}\bigr) - \phi(m_0)
&= & \int_{0}^1 d\lambda \int_{\RR^d} \frac{\delta \phi}{\delta m}
\bigl( \lambda m^n_{\bm{\zeta}} + (1-\lambda) m_0,v\bigr) d \bigl( m^n_{\bm{\zeta}} - m_0 \bigr)(v)
\\
&= & S_1  + S_2, 
\end{eqnarray*}
where 
\begin{eqnarray*}
S_1 & = & \int_{\RR^d} \frac{\delta \phi}{\delta m}
( m_0,v) d \bigl( m^n_{\bm{\zeta}} - m_0 \bigr)(v), 
\\
S_2 & = & \int_{0}^1 d\lambda \int_{\RR^d}
\Bigl[\frac{\delta \phi}{\delta m}
\bigl( 
 \lambda m^n_{\bm{\zeta}} + (1-\lambda) m_0,v\bigr)
-
 \frac{\delta \phi}{\delta m}
( m_0,v)
\Bigr]
 d \bigl( m^n_{\bm{\zeta}} - m_0 \bigr)(v). 
\end{eqnarray*}

\noindent\textbf{Step 1:}
We first deal with the term $S_1$  This may be rewritten as
\begin{equation*}
\begin{split}
&S_1
= \frac{1}{n}
\sum_{i=1}^n \left(
\frac{\delta \phi}{\delta m}
\bigl( m_0,\zeta_i\bigr) - {\mathbb E}
\left[
\frac{\delta \phi}{\delta m}
\bigl( m_0,\zeta_1)
\right]
\right).
\end{split} 
\end{equation*}
 Since $\delta \phi/\delta m$ is bounded, we get 
${\mathbb E}[ \vert S_1 \vert^4 ]^{1/4} \leq 
 c n^{-1/2}$, for a constant $c < \infty$ independent of $n$. The value of $c$ is allowed to vary from line to line. 
 \vskip 4pt
 
\noindent\textbf{Step 2:} Now, we turn to the term $S_2$, which we first rewrite in the form
\begin{align*}
S_2 &= \frac{1}{n} \sum_{i=1}^n\int_0^1\varphi_{\lambda}^i\,d\lambda,
\end{align*}
where we define, for $i=1,\ldots,n$ and $\lambda \in [0,1]$,
\begin{equation}
\label{eq:appendix:lln:varphii}
\begin{split}
\varphi_{\lambda}^i &:= 
\Bigl[\frac{\delta \phi}{\delta m}
\bigl( 
 \lambda m^n_{\bm{\zeta}} + (1-\lambda) m_0, \zeta_i
\bigr)
-
 \frac{\delta \phi}{\delta m}
( m_0,\zeta_i
)\Bigr] 
 \\
 &\hspace{60pt}- 
\widetilde\EE \biggl[ \frac{\delta \phi}{\delta m}
\bigl( 
 \lambda m^n_{\bm{\zeta}} + (1-\lambda) m_0, \widetilde \zeta\bigr)
-
 \frac{\delta \phi}{\delta m}
\bigl( m_0, \widetilde \zeta
\bigr)
\biggr],
\end{split}
\end{equation}
and where the expectation $\widetilde \EE$ is taken over the random variable $\widetilde{\zeta}$, which has the same distribution as $\zeta_1$ and which is independent of $\bm{\zeta}$.
Now, we have 
\begin{align*}
|S_2|^4 &\le n^{-4}\int_0^1\left(\sum_{i=1}^n\varphi_{\lambda}^i\right)^4\,d\lambda,
\end{align*}
so we focus on estimating the integrand for a fixed $\lambda$. We have
\begin{equation}
\label{eq:l4:lem}
\begin{split}
\biggl( \sum_{i=1}^n \varphi_{\lambda}^{i}
\biggr)^4 &= 
\sum_{i=1}^n (\varphi_{\lambda}^i)^4 +
3 \sum_{i \not = j} (\varphi_{\lambda}^i)^2 (\varphi_{\lambda}^j)^2 
 + 4 \sum_{i \not = j}
\varphi_{\lambda}^{i} (\varphi_{\lambda}^{j})^3
\\
&\hspace{15pt} + 6 \sum_{i_{1},i_{2},i_{3} \ \textrm{distinct}}
 \varphi_{\lambda}^{i_{1}} \varphi_{\lambda}^{i_{2}} (\varphi_{\lambda}^{i_{3}})^2
 +
 \sum_{i_{1},i_{2},i_{3},i_{4} \ \textrm{distinct}} 
  \varphi_{\lambda}^{i_{1}} \varphi_{\lambda}^{i_{2}} \varphi_{\lambda}^{i_{3}} \varphi_{\lambda}^{i_{4}}. 
  \end{split}
\end{equation}
Noting that $\varphi_{\lambda}^i$ is bounded, uniformly in $i$ and $\lambda$, we bound the first three terms on the right-hand side by $cn^2$.
The only difficulty is to upper bound the last two terms. 
\vskip 4pt

\noindent\textbf{Step 3:} We start with the sum over 
$i_{1},i_{2},i_{3}$ distinct. We let, for any $(i_{1},i_{2},i_{3}) \in \{1,\ldots,n\}^3$, $
m^{n,-(i_{1},i_{2},i_{3})}_{\bm{\zeta}}
= \frac{1}{n-3}
\sum_{\ell \not = i_{1},i_{2},i_{3}}
\delta_{\zeta_\ell}$. We then observe that 
\begin{equation*}
\begin{split}
&\Bigl[\frac{\delta \phi}{\delta m}
\bigl( 
 \lambda m^n_{\bm{\zeta}} + (1-\lambda) m_0,\zeta_i
\bigr)
-
 \frac{\delta \phi}{\delta m}
\bigl( m_0,\zeta_i
\bigr)\Bigr]
\\
&\hspace{15pt}
- 
\Bigl[\frac{\delta \phi}{\delta m}
\bigl( 
 \lambda m^{n,-(i_{1},i_{2},i_{3})}_{\bm{\zeta}}  + (1-\lambda) m_0,\zeta_i
\bigr)
-
 \frac{\delta \phi}{\delta m}
\bigl( m_0,\zeta_i
\bigr)\Bigr]
\\
&=
\Bigl[\frac{\delta \phi}{\delta m}
\bigl( 
 \lambda m^n_{\bm{\zeta}} + (1-\lambda) m_0, \zeta_i
\bigr)
-
\frac{\delta \phi}{\delta m}
\bigl( 
 \lambda {m}^{n,-(i_{1},i_{2},i_{3})}_{\bm{\zeta}} + (1-\lambda) m_0, \zeta_i
\bigr) \Bigr]
 \\
 &= \int_{0}^1 d s 
 \int_{\RR^d}
\frac{\delta^2 \phi}{\delta m^2}
\Bigl( 
 s\lambda m^n_{\bm{\zeta}} 
 +(1-s)
 \lambda 
 m^{n,-(i_{1},i_{2},i_{3})}_{\bm{\zeta}}
  + (1-\lambda) m_0, \zeta_i,v
\Bigr) 
d \bigl( 
m^n_{\bm{\zeta}}
-m^{n,-(i_{1},i_{2},i_{3})}_{\bm{\zeta}} \bigr)(v).
\end{split}
\end{equation*}
Now,
\begin{equation}
\label{eq:lem:appendix:diff:dirac:mass}
m^n_{\bm{\zeta}}
-m^{n,-(i_{1},i_{2},i_{p})}_{\bm{\zeta}}
= \frac{1}{n} \sum_{j = i_{1},i_{2},i_{3}}
\delta_{\zeta_j}
- \frac{3}{n (n-3)} \sum_{j \not = i_{1},i_{2},i_{3}} \delta_{\zeta_j},
\end{equation}
and then, since $\delta^2 \phi/\delta m^2$ is bounded, we get
\begin{equation*}
\begin{split}
&\biggl\vert \Bigl[\frac{\delta \phi}{\delta m}
\bigl( 
 \lambda m^n_{\bm{\zeta}} + (1-\lambda) m_0,\zeta_i
\bigr)
-
 \frac{\delta \phi}{\delta m}
\bigl( m_0,\zeta_i
\bigr)\Bigr]
\\
&\hspace{15pt}
- 
\Bigl[\frac{\delta \phi}{\delta m}
\bigl( 
 \lambda m^{n,-(i_{1},i_{2},i_{3})}_{\bm{\zeta}}  + (1-\lambda) m_0,\zeta_i
\bigr)
-
 \frac{\delta \phi}{\delta m}
\bigl( m_0,\zeta_i
\bigr)\Bigr] \biggr\vert \leq \frac{c}{n}. 
\end{split}
\end{equation*}
Proceeding in the same way 
for the term containing $\widetilde{\zeta}$ in 
\eqref{eq:appendix:lln:varphii}, we get 
\begin{equation*}
\vert \varphi_{\lambda}^{i} - \varphi_{\lambda}^{i,-(i_{1},i_{2},i_{3})} \vert \leq \frac{c}{n},
\end{equation*}
where
\begin{equation*}
\begin{split}
\varphi_{\lambda}^{i,-(i_{1},i_{2},i_{3})} &= 
\Bigl[\frac{\delta \phi}{\delta m}
\bigl( 
 \lambda m^{n,-(i_{1},i_{2},i_{3})}_{\bm{\zeta}} + (1-\lambda) m_0, \zeta_i
\bigr)
-
 \frac{\delta \phi}{\delta m}
( m_0,\zeta_i
)\Bigr] 
 \\
 &\hspace{60pt}- 
\widetilde\EE \biggl[ \Bigl[\frac{\delta \phi}{\delta m}
\bigl( 
 \lambda  m^{n,-(i_{1},i_{2},i_{3})}_{\bm{\zeta}} + (1-\lambda) m_0, \widetilde \zeta\bigr)
-
 \frac{\delta \phi}{\delta m}
\bigl( m_0, \widetilde \zeta
\bigr)\Bigr]
\biggr].
\end{split}
\end{equation*}
Therefore,
\begin{align*}
&{\mathbb E} \left[ \sum_{i_{1},i_{2},i_{3} \ \textrm{\rm distinct}} \varphi_{\lambda}^{i_{1}}
\varphi_{\lambda}^{i_{2}} (\varphi_{\lambda}^{i_{3}})^2 
\right] \\
	&\quad\quad\quad\quad\quad\leq c n^2 + {\mathbb E} \biggl[ \sum_{i_{1},i_{2},i_{3} \ \textrm{\rm distinct}} \varphi_{\lambda}^{i_{1},-(i_{1},i_{2},i_{3})}
\varphi_{\lambda}^{i_{2},-(i_{1},i_{2},i_{3})} (\varphi_{\lambda}^{i_{3},-(i_{1},i_{2},i_{3})})^2 
\biggr],
\end{align*}  
and it is clear, by a conditional independence argument, that the last sum is in fact zero. This shows that the left-hand side is less than $c n^2$.   
\vskip 4pt

\noindent\textbf{Step 4:} It now remains to study the sum over distinct $i_{1},\cdots,i_{4}$ in
\eqref{eq:l4:lem}.  To do so, we may define in a similar manner $\varphi_{\lambda}^{i,-(i_{1},\ldots,i_{4})}$ for $(i_{1},\ldots,i_{4}) \in \{1,\cdots,n\}^4$. 
Notice first that, for any $(i_{1},\ldots,i_{4})$,
\begin{align}
\varphi_{\lambda}^{i_{1}}
\varphi_{\lambda}^{i_{2}}\varphi_{\lambda}^{i_{3}}\varphi_{\lambda}^{i_{4}} = \,&\varphi_{\lambda}^{i_{1},-(i_{1},\ldots,i_{4})}
\varphi_{\lambda}^{i_{2},-(i_{1},\ldots,i_{4})}\varphi_{\lambda}^{i_{3},-(i_{1},\ldots,i_{4})}\varphi_{\lambda}^{i_{4},-(i_{1},\ldots,i_{4})} \label{eq:l4:lln-0} \\
	&+ \sum_{j=1}^4(\varphi_{\lambda}^{i_j} - \varphi_{\lambda}^{i_j,-(i_{1},\ldots,i_{4})})\prod_{\stackrel{k=1}{k \neq j}}^4\varphi_{\lambda}^{i_k,-(i_{1},\ldots,i_{4})} + S_3, \nonumber
\end{align}
where  $S_3$  is a sum of terms that involve a  product of four terms, at least two of which are of the form  $(\varphi_{\lambda}^{i_j} - \varphi_{\lambda}^{i_j,-(i_{1},\ldots,i_{4})})$ for some $j$, and the rest are of the form $\varphi_{\lambda}^{i_j}$ or $\varphi_{\lambda}^{i_j,-(i_{1},\ldots,i_{4})}$. But since $|\varphi_{\lambda}^{i_j} - \varphi_{\lambda}^{i_j,-(i_{1},\ldots,i_{4})}| \le c/n$, and since the first term on the right-hand side of \eqref{eq:l4:lln-0} has mean zero, we find
\begin{align}
&{\mathbb E}
\biggl[ \sum_{i_{1},i_{2},i_{3},i_{4} \ \textrm{\rm distinct}} \varphi_{\lambda}^{i_{1}}
\varphi_{\lambda}^{i_{2}}\varphi_{\lambda}^{i_{3}}\varphi_{\lambda}^{i_{4}}
\biggr] \leq  c n^2 + {\mathbb E}
\biggl[ 
\sum_{i_{1},\ldots,i_{4} \ \textrm{\rm distinct}} 
\sum_{j=1}^4(\varphi_{\lambda}^{i_j} - \varphi_{\lambda}^{i_j,-(i_{1},\ldots,i_{4})})\prod_{\stackrel{k=1}{k \neq j}}^4\varphi_{\lambda}^{i_k,-(i_{1},\ldots,i_{4})}
\biggr]. \label{eq:l4:lln}
\end{align}
Returning to the proof of the third step but using in addition the $\W_1$-Lipschitz property of $\delta^2 \phi/\delta m^2$ with respect to the measure argument, we can write:
\begin{equation}
\begin{split}
&\varphi_{\lambda}^{i} 
-
\varphi_{\lambda}^{i,-(i_{1},\ldots,i_{4})} 
\\
&=
 \int_{0}^1 d s 
 \int_{\RR^d}
\frac{\delta^2 \phi}{\delta m^2}
\Bigl( 
\lambda 
 m^{n,-(i_{1},\ldots,i_{4})}_{\bm{\zeta}}
  + (1-\lambda) m_0, \zeta_i,v
\Bigr) 
d \bigl( 
m^n_{\bm{\zeta}}
-m^{n,-(i_{1},\ldots,i_{4})}_{\bm{\zeta}} \bigr)(v)
\\
&\hspace{15pt} -
\widetilde{\EE}
 \int_{0}^1 d s 
 \int_{\RR^d}
\frac{\delta^2 \phi}{\delta m^2}
\Bigl( 
\lambda 
 m^{n,-(i_{1},\ldots,i_{4})}_{\bm{\zeta}}
  + (1-\lambda) m_0,\widetilde \zeta,v
\Bigr) 
d \bigl( 
m^n_{\bm{\zeta}}
-m^{n,-(i_{1},\ldots,i_{4})}_{\bm{\zeta}} \bigr)(v)
\\
&\hspace{15pt} + \varepsilon_{n}^{i_{1},\ldots,i_{4}}, 
\end{split}  \label{eq:l4:lln-1}
\end{equation}
where 
\begin{equation*}
\vert \varepsilon_{n}^{i_{1},\ldots,i_{4}} \vert \leq 
\frac{c}{n^2} \bigl( \vert \zeta_{i_{1}} \vert + \ldots + 
\vert \zeta_{i_{4}} \vert \bigr)
+ \frac{c}{n^3} \sum_{j=1}^n \vert \zeta_j \vert.
\end{equation*}
Following \eqref{eq:lem:appendix:diff:dirac:mass}, we can write the difference 
$\varphi_{\lambda}^{i} -\varphi_{\lambda}^{i,-(i_{1},\ldots,i_{4})}$ in the form
\begin{equation*}
\begin{split}
\varphi_{\lambda}^{i} 
-
\varphi_{\lambda}^{i,-(i_{1},\ldots,i_{4})}
= \frac1n \sum_{j=1}^4 F \bigl( m^{n,-(i_{1},\ldots,i_{4})}_{\bm{\zeta}},\zeta_i,\zeta_{i_j} \bigr) 
+ \frac1n G\bigl( m^{n,-(i_{1},\ldots,i_{4})}_{\bm{\zeta}},\zeta_i \bigr) + \varepsilon_{n}^{i_{1},\ldots,i_{4}},
\end{split}
\end{equation*}
for bounded measurable functions $F$ and $G$. Now note that $F \bigl( m^{n,-(i_{1},\ldots,i_{4})}_{\bm{\zeta}},\zeta_{i_k},\zeta_{i_j} \bigr)$ is measurable with respect to $(\zeta_\ell : \ell \in \{i_1,\ldots,i_4\}^c \cup \{i_k,i_j\})$, for $k,j \in \{1,\ldots,4\}$, and similarly for the $G$ term. On the other hand, for $k=1,\ldots,4$,
\[
\E[\varphi_{\lambda}^{i_k,-(i_{1},\ldots,i_{4})} \, | \,  (\zeta_\ell : \ell \neq i_k)] = 0.
\]
Hence, for each $j \in \{1,\ldots,4\}$ and each $i_1,\ldots,i_4$ distinct, we find that the corresponding term in \eqref{eq:l4:lln} simplifies in the sense that
\begin{align*}
\E\left[\sum_{j=1}^4(\varphi_{\lambda}^{i_j} - \varphi_{\lambda}^{i_j,-(i_{1},\ldots,i_{4})} - \epsilon_n^{i_{1},\ldots,i_{4}})\prod_{\stackrel{k=1}{k \neq j}}^4\varphi_{\lambda}^{i_k,-(i_{1},\ldots,i_{4})}\right] = 0.
\end{align*}
This finally lets us bound the $(i_1,\ldots,i_4)$-term in \eqref{eq:l4:lln} by $c\E[\epsilon_n^{i_{1},\ldots,i_{4}}] \le cn^{-2}\E[|\zeta_1|]$, which completes the proof.
\end{proof}

\section{A linear-quadratic example} \label{se:examples}

This section discusses an explicitly solvable model that does not satisfy Assumption \ref{assumption:A}. Nonetheless, 
we show that our strategy for deriving the central limit theorem by comparison with a more classical McKean-Vlasov system is still successful.

\subsection{Description of the  model}
\label{subs-lqmodel}

Consider the mean field game model of systemic risk proposed in \cite{carmona-fouque-sun}.
Here, $d=1$, $\sigma$ and $\sigma_0$ are positive constants, the action space  $A = \R$, and 
for some  $\cc, \epsilon, \bb > 0$ and $0 \le q^2 \le \epsilon$ we have
\begin{align*} 
b(x,m,a) &= \bb(\overline{m} - x) + a, \\
f(x,m,a) &= \frac{1}{2}a^2 - qa(\overline{m}-x) + \frac{\epsilon}{2}(\overline{m}-x)^2, \\
g(x,m) &= \frac{\cc}{2}(\overline{m} - x)^2,
\end{align*}
where $\overline{m} = \int_{\R}y\,dm(y)$. 
Both the drift and cost functions induce a herding behavior toward the population average; see \cite{carmona-fouque-sun} for a thorough discussion.

It was shown in \cite[(3.24)]{carmona-fouque-sun}
 that the unique closed loop Nash equilibrium dynamics is given by: 
\begin{align}
\alpha^i_t = \left[q + \varphi^n_t\left(1-\frac{1}{n}\right)\right](\overline{X}_t - X^i_t), \quad t \in [0,T], \label{def:Nash-open}
\end{align}
where $\overline{X}_t = \frac{1}{n}\sum_{i=1}^nX^i_t$, and where $\varphi^n$ is the unique solution to the Riccati equation:
\begin{align*}
\dot{\varphi}^n_t = 2(\bb+q)\varphi^n_t + \left(1-{\frac{1}{n^2}}\right)|\varphi^n_t|^2 - (\epsilon - q^2), \quad\quad \varphi^n_T = \cc.
\end{align*}
The explicit solution takes the form 
\begin{align}
\varphi^n_t = \frac{-(\epsilon-q^2)\left(e^{(\delta_n^+-\delta_n^-)(T-t)}-1\right) - \cc\left(\delta^+_n e^{(\delta_n^+-\delta_n^-)(T-t)}-\delta_n^-\right)}{\left(\delta_n^-e^{(\delta_n^+-\delta_n^-)(T-t)} - \delta_n^+\right) - \cc\left(1-\frac{1}{n^2}\right)\left(e^{(\delta_n^+-\delta_n^-)(T-t)}-1\right)}, \label{def:Riccati1}
\end{align}
where
\begin{align}
\delta_n^{\pm} = -(\bb+q) \pm \sqrt{(\bb+q)^2 + \left(1-\frac{1}{n^2}\right)(\epsilon-q^2)}. \label{def:Riccati2}
\end{align}
In particular, the Nash equilibrium state process is given by the solution $\bm{X}=(X^1,\ldots,X^n)$ of the SDE system:
\begin{align}
\label{eq:CFS:1}
dX^i_t = \left(\bb + q + \varphi^n_t\left(1-\frac{1}{n}\right)\right)(\overline{X}_t - X^i_t)dt + \sigma dB^i_t + \sigma_0 dW_t, \quad t \in [0,T]. 
\end{align}

It is straightforward to show that $\varphi^n_t \rightarrow \varphi^\infty_t$ as $n\rightarrow\infty$, uniformly in $t \in [0,T]$, where $\varphi^\infty$ is the unique solution to the Riccati equation
\begin{align*}
\dot{\varphi}^{\infty}_t = 2(\bb+q)\varphi^{\infty}_t + |\varphi^{\infty}_t|^2 - (\epsilon - q^2), \quad\quad \varphi^{\infty}_T = \cc.
\end{align*}
The explicit solution is of the same form given by \eqref{def:Riccati1} and \eqref{def:Riccati2}, with $n=\infty$.
It follows that $\bm{X}=(X^1,\ldots,X^n)$ should be ``close'' in some sense to the solution $\bm{Y}=(Y^1,\ldots,Y^n)$ of the auxiliary SDE system:
\begin{align}
\label{eq:CFS:2}
dY^i_t = (\bb + q + \varphi^\infty_t)(\overline{Y}_t - Y^i_t)dt + \sigma dB^i_t + \sigma_0 dW_t,
\end{align}
initialized at the same points $Y^i_0=X^i_0$.
Of course, it should be noted that 
the process ${\boldsymbol Y}$ plays here the same role as the process $\overline{\boldsymbol X}$ in 
\eqref{eq:sec4:overlineX}, the solution 
$U$ to the master equation being given in the current framework by:
\begin{equation*}
U(t,x,m) = \frac{\varphi_{t}^\infty}2 \bigl( \overline m - x \bigr)^2.  
\end{equation*}
In this regard, the fact that ${\boldsymbol X}$ and ${\boldsymbol Y}$ should be ``close'' is completely analogous to the statements of Theorems \ref{th:mainestimate} and \ref{th:mainestimate2}.

In fact, there are two ways to compare ${\boldsymbol X}$ and ${\boldsymbol Y}$. One is to use the explicit expressions for their solutions. 
Another  strategy, which we prefer,  is to compare \eqref{eq:CFS:1} and \eqref{eq:CFS:2} and  use the fact that $\bm{X}_0 = \bm{Y}_0$,  applying Gronwall's inequality to  deduce that there exists a constant $ C < \infty$ such that
\begin{align}
\label{eq:sec:8:gronwall}
\frac{1}{n}\sum_{i=1}^n\|X^i-Y^i\|_{{\infty}} \le C\left\|
\left(1-\frac{1}{n}\right)
\varphi^n - \varphi^\infty\right\|_{{\infty}} \frac{1}{n}\sum_{i=1}^n\|X^i\|_{{\infty}}, \ a.s.,
\end{align}
where, as usual, $\|\cdot\|_\infty$ denotes the supremum norm on $[0,T]$.  
On the other hand, the equation \eqref{eq:CFS:1} and  Gronwall's inequality  together yield
\begin{align*}
\frac{1}{n}\sum_{i=1}^n\|X^i\|_{{\infty}} \le C\left(1 + \frac{1}{n}\sum_{i=1}^n|X^i_0| + \frac{1}{n}\sum_{i=1}^n\|B^i\|_{{\infty}} + \|W\|_{{\infty}} \right), \ a.s.
\end{align*}
As soon as $X^i_0$ are i.i.d. and subgaussian (e.g., $\E[\exp(\kappa|X^1_0|^2)] < \infty$ for some $\kappa > 0$), we find a uniform subgaussian bound on these averages; that is, there exist constants $C < \infty, \delta > 0$, independent of $n$, such that 
\begin{align*}
\PP\left(\frac{1}{n}\sum_{i=1}^n\|X^i\|_{{\infty}} > a\right) \le \exp(-\delta^2 a^2), \text{ for all } a \ge C, \ n \in \N.
\end{align*}
Assuming without any loss of generality that the constant $C$ in the last display coincides with the one in  \eqref{eq:sec:8:gronwall}, and letting $r_n = C\left\|
\left(1-\frac{1}{n}\right)
\varphi^n - \varphi^\infty\right\|_{{\infty}}$, we find that, for $a \ge Cr_n$,
\begin{align}
\PP\left(\W_{{1,\C^d}}(m^n_{\bm{X}},m^n_{\bm{Y}}) > a\right) &\le \PP\left(\frac{1}{n}\sum_{i=1}^n\|X^i-Y^i\|_{{\infty}} > a\right) \le \PP\left(\frac{r_n}{n}\sum_{i=1}^n\|X^i\|_{{\infty}} > a\right) \nonumber \\
	&\le \exp \left(-\frac{\delta^2a^2}{r_n^2}\right). \label{def:CFSmodel-keyestimate}
\end{align}
It is straightforward to check that $r_n = O(1/n)$, which implies in particular \eqref{def:estimate-|X-Xbar|}.

\subsection{Derivation of the CLT} 
Following our study in Section \ref{se:CLT-proofs}, the key estimate \eqref{def:CFSmodel-keyestimate} is sufficient to transfer a central limit theorem from the classical McKean-Vlasov particle system $\bm{Y}$ over to the Nash system $\bm{X}$.
Still, 
the derivation of the CLT suffers from the fact that the drift is unbounded, which prevents us from directly applying the results of Section 
\ref{se:CLT-proofs}.
We explain below how to adapt the arguments.

The first step in the analysis is to check that the tightness property, as stated in Proposition \ref{thm:clt:tightness}, 
remains true. In fact, the main point is to revisit 
\eqref{eq:bound:2}.
For $\newDD \geq 1$, we have
\begin{equation*}
\begin{split}
&\bigl\|
 \phi'' 
\bigr\|_{1+\newDD,2\newDD} 
+
\bigl\|
 \phi' 
\bigr\|_{1+\newDD,2\newDD}
\leq C \bigl\| \phi \bigr\|_{2+2\newDD,\newDD}, \quad \phantom{\int}
\\
&\bigl\|
  \phi'  \, \widetilde b(t,\cdot,\mu)
\bigr\|_{1+\newDD,2\newDD} \leq C \bigl\| \phi \bigr\|_{2+2\newDD,\newDD},\phantom{\int}
\end{split}
\end{equation*}
where $\widetilde b(t,x,\mu)$ is here given by $(\bb+q+\varphi_{t}^{\infty})(x- \overline{\mu})$. 

Similar to 
 \eqref{eq:def:cal:G}, we let
 \begin{equation*}
 \bigl[ {\mathcal G}(t,m,\mu) \phi
 \bigr](v) = 
- (b+q+\varphi_{t}^{\infty})
 \bigl\langle \mu, \phi'  \bigr\rangle v.  
 \end{equation*}			
 It is then easily  checked that the first line in 
\eqref{eq:bound:2} remains true.  
The rest of the proof of tightness, see the statement of Proposition \ref{thm:clt:tightness}, is similar.
Importantly, it holds true for values of $\newDD$ greater than or equal to $\lfloor d/2 \rfloor +1$ (which is here equal to 1), and not only for $\newDD=1$, provided the constraint $p' > 12 \newDD$ in the statement 
of Theorem
\ref{th:CLT:sec:3} is modified accordingly. 

The second step, which is in fact more difficult, is the identification of the limit. The first point is to write down the corresponding equation 
\eqref{eq:lim:spde:for:clt}. The difficulty is 
twofold: first, 
the function $v \mapsto [{\mathcal G}(t,m,\mu) \phi](v)$ is not in ${\mathcal H}^{2\newDD+4,\newDD}$
when $\newDD=1$; second, when the test function 
$\phi$ is in the space ${\mathcal H}^{2\newDD+4,\newDD}$, the function 
$\phi' \widetilde{b}(t,\cdot,\mu)$ may not belong to ${\mathcal H}^{2\newDD+2,\newDD}$ because of the term 
$x\phi'(x)$ in 
$\phi'(x) \widetilde{b}(t,x,\mu)$. 
The first difficulty may be easily circumvented by choosing $\newDD>3/2$, in which case the identity function 
$\textrm{\rm Id} : v \mapsto v$ indeed belongs to ${\mathcal H}^{2\newDD+4,\newDD}$; as already explained, this mostly requires us to change $p'$ accordingly. 

It is more challenging to deal with the fact that the function $x \mapsto x\phi'(x)$ does not belong to ${\mathcal H}^{2\newDD+2,\newDD}$.
The key point is to modify the space of test functions. Indeed, if we assume that $\phi$ is in the space ${\mathcal H}^{2\newDD+4,\newDD-1}$, 
then the function $\phi' \widetilde{b}(t,\cdot,\mu)$ is ${\mathcal H}^{2\newDD+2,\newDD}$.
For this, it  suffices to write down the analogue of
\eqref{eq:spde:cor:lim}, but in a smaller space of test functions.  Specifically, this takes the form
\begin{equation}
\begin{split}
d \bigl\langle S_{t},\phi
\bigr\rangle 
&=  (\bb + q + \varphi_{t}^{\infty}) \Bigl[
    \bigl\langle
 S_{t},
 \phi' \times
 (\textrm{\rm Id} - \bar \mu_t) \bigr\rangle 
 -  \langle 
 S_{t}, \textrm{\rm Id} \rangle
 \bigl\langle \mu_t, \phi' \bigr\rangle \Bigr] dt
 \\ 
&\hspace{15pt} + \tfrac12 \bigl\langle S_{t},
( \sigma^2 
+ \sigma_{0}^2 ) \phi''  \bigr\rangle  dt +
\sigma_{0} \bigl\langle S_{t},  \phi' 
\bigr\rangle dW_t
+ d {\xi_{t}}(\phi),
\end{split} \label{def:LQSPDE}
\end{equation} 
for $\phi$ in the space ${\mathcal H}^{2\newDD+4,\newDD-1}$.

Of course, the difficulty is then to prove uniqueness of the limiting equation.
By linearity of the equation, it suffices to prove 
that any solution $(S_t)_{t \in [0,T]}$ in ${\mathcal H}^{-(2\newDD+2),\newDD}$
to the above equation, with a null initial condition and with $\xi$ set to $0$, is null. 
To do so, we assume that $\newDD>5/2$ (say $\newDD=3$ to simplify). Choosing $\phi \equiv 1$ as test function, it is clear that $\langle S_t,1\rangle =0$. 
Then, choosing $\phi$ as the identity $\textrm{\rm Id} : x \mapsto \phi(x)=x$ as test function, we obtain {$d\langle S_t, \mathrm{Id}\rangle = 0$}
The next step is to take advantage of the linear structure of the drift. Indeed, for any 
$t \in [0,T]$ and any test function $\phi \in {\mathcal H}^{2\newDD+4,\newDD-1}$, we may consider the new test function $\phi_{t} : \RR \ni x \mapsto 
\phi( \ell_{t} x)$, where 
\begin{equation*}
\ell_{t} = \exp \biggl( - \int_{0}^t \bigl(\bb + q+ \varphi_{s}^{\infty} \bigr) ds \biggr).
\end{equation*}
We then compute
$ d \langle S_{t},\phi_{t}  \rangle$ as
\begin{equation}
\label{eq:St:phit}
\begin{split}
d \bigl\langle S_{t},\phi_{t}\rangle
&= 
- \ell_{t} (\bb + q + \varphi_{t}^{\infty}) 
  \overline \mu_t
   \bigl\langle
 S_{t},
 \phi'(\ell_{t} \cdot)
 \bigr\rangle 
 dt
 \\ 
&\hspace{15pt} + \frac{\ell_{t}^2}2 \bigl\langle S_{t},
( \sigma^2 
+ \sigma_{0}^2 ) \phi''(\ell_{t} \cdot) \bigr\rangle  dt +
\sigma_{0}
\ell_{t} \bigl\langle S_{t},  \phi'(\ell_{t} \cdot) 
\bigr\rangle d {W_{t}}, \quad t \in [0,T].
\end{split}
\end{equation}
In order to conclude, we call $S_{t}'$ the element of ${\mathcal H}^{-(2\newDD+2),\newDD}$ defined by
\begin{equation*}
\bigl\langle S_{t}',\psi \bigr\rangle 
= \bigl\langle S_{t}, \psi( \ell_{t} \cdot) 
\bigr\rangle. 
\end{equation*}
With this notation in hand, we can write
\begin{equation*}
\begin{split}
d \bigl\langle S_{t}',\phi\rangle
&= 
-\ell_{t} (\bb + q + \varphi_{t}^{\infty}) 
  \overline \mu_t
   \bigl\langle
 S_{t}',
 \phi'
 \bigr\rangle 
 dt
 \\ 
&\hspace{15pt} + \frac{\ell_{t}^2}2 \bigl\langle S_{t}',
( \sigma^2 
+ \sigma_{0}^2 ) \phi''  \bigr\rangle  dt +
\sigma_{0}
\ell_{t} \bigl\langle S_{t}',  \phi' 
\bigr\rangle d {W_{t}}, \quad t \in [0,T].
\end{split}
\end{equation*}
Obviously, $S_{0}'=0$. Once again, the above holds true for any test function 
$\phi \in {\mathcal H}^{2\newDD+4,\newDD-1}$, but, since the drift no longer involves the function 
$x \mapsto \phi'(x) x$, it also holds true for any test function $\phi \in {\mathcal H}^{2\newDD+4,\newDD}$. 
Now, we can use the same argument in the proof of uniqueness in Section 
\ref{se:CLT-proofs} to show that, necessarily, 
$S'$ is null.

\subsection{Explicit solution of the SPDE}
We now provide an explicit expression for the solution to 
\eqref{def:LQSPDE}.
{To do so, we first notice that, in the above framework,  
the limiting McKean-Vlasov equation takes the form
\begin{equation}
\label{MKV:LQ}
d {\mathcal X}_{t} = (\bb + q + \varphi^\infty_t)( \overline \mu_{t} - {\mathcal X}_t)dt + \sigma dB^1_t + \sigma_0 dW_t,
\quad t \in [0,T],
\end{equation}
where $\mu_{t} = {\mathcal L}({\mathcal X}_{t} \vert W)$ and 
$\overline \mu_{t} = {\mathbb E}[{\mathcal X}_{t} \vert W]$. 
Equation \eqref{MKV:LQ} is the analogue of 
\eqref{def:MKV-conditional}.} Taking the conditional mean with respect to 
$W$ in 
\eqref{MKV:LQ}, we deduce that $\overline \mu_{t} = \overline \mu_{0} + \sigma_{0} W_{t}$, 
for $t \in [0,T]$. Also, we have 
\begin{align*}
d \bigl( {\mathcal X}_t - \overline \mu_{t} \bigr) = - (\bb + q + \varphi^\infty_t)( {\mathcal X}_t - \overline \mu_{t})dt + \sigma dB^1_t,
\end{align*}
that is, 
\begin{equation*}
{\mathcal X}_{t} - \overline \mu_{t} = 
\ell_{t} \biggl( {\mathcal X}_{0} - \overline{\mu}_{0} + \sigma \int_{0}^t \ell_{s}^{-1}  dB_{s}^1 \biggr),
\end{equation*}
which shows that, conditional on $({\mathcal X}_{0},W)$, ${\mathcal X}_{t}$ has a Gaussian distribution. We then get 
\begin{equation*}
\mu_{t} = 
\int_{\RR}
{\mathcal N}\biggl(\ell_{t} x+ (1-\ell_{t}) \overline \mu_{0} + \sigma_{0} W_{t}, \sigma^2 \ell_{t}^2 \int_{0}^t 
\ell_{s}^{-2} ds \biggr) d\mu_{0}(x). 
\end{equation*}
In particular, $\mu_{t}$ has a smooth density, with which we identify it. 

 {Recall from Section \ref{se:CLT-statements} the expression of the Gaussian process $\xi$
in terms of a cylindrical white noise $\beta$ with values in $L^2(\RR)$, independent of $(\theta_0,W)$.
 Taking } $\phi = \textrm{\rm Id}$ in 
 \eqref{def:LQSPDE}, we get as in the previous paragraph that $d\bigl\langle S_{t},\textrm{\rm Id} \bigr\rangle  = \sigma \langle \sqrt{\mu_{t}} , d \beta_{t} \rangle$.
That is, 
 \begin{equation}
 \label{eq:St:Id}
 \begin{split}
 \bigl\langle S_{t},\textrm{\rm Id} \bigr\rangle
 &=  \bigl\langle \theta_{0},\textrm{\rm Id} \bigr\rangle   +   \sigma \int_{0}^t  \langle \sqrt{\mu_{s}} , d \beta_{s} \rangle.
 \end{split}
 \end{equation}
Observe that, in the last term, the process 
$(\int_{0}^t \langle \sqrt{\mu}_{s},d \beta_{s} \rangle)_{t \in [0,T]}$
is a standard Brownian motion.
 
 Now, 
 performing the same computation as in 
\eqref{eq:St:phit},
we have for any $\phi \in \cH^{2\lambda+4,\lambda-1}$, 
\begin{equation*}
\begin{split}
d \bigl\langle S_{t},\phi_{t}\rangle
&= 
- \ell_{t} (\bb + q + \varphi_{t}^{\infty}) 
  \overline \mu_t
   \bigl\langle
 S_{t},
 \phi'(\ell_{t} \cdot)
 \bigr\rangle 
 dt
- \ell_{t} (\bb + q + \varphi_{t}^{\infty}) 
   \bigl\langle
 S_{t},
\textrm{\rm Id} \bigr\rangle 
   \bigl\langle
 \mu_{t},
\phi'(\ell_{t} \cdot) \bigr\rangle dt
 \\ 
&\hspace{15pt} + \frac{\ell_{t}^2}2 \bigl\langle S_{t},
( \sigma^2 
+ \sigma_{0}^2 ) \phi''(\ell_{t} \cdot) \bigr\rangle  dt +
\sigma_{0}
\ell_{t} \bigl\langle S_{t},  \phi'(\ell_{t} \cdot) 
\bigr\rangle d W_{t} + 
{\sigma}
 \ell_{t}\bigl\langle \sqrt{\mu_{t}(\cdot)} \phi'(\ell_{t} \cdot), d\beta_{t}(\cdot) \bigr\rangle,
\end{split}
\end{equation*}
where 
$\phi_{t}(\cdot)
= \phi(\ell_{t} \cdot)$.
For $t \in [0,T]$, let $\widetilde{S}_{t}$ denote the element of ${\mathcal H}^{-(2\lambda+2),\lambda}$ given by 
$\langle \widetilde{S}_{t},\phi \rangle  = \langle S_{t},\phi(\ell_{t} \cdot) \rangle$, 
for $\phi \in {\mathcal H}^{2\lambda+2,\lambda}$. Then,
for $\phi \in {\mathcal H}^{2\lambda+4,\lambda-1}$,
\begin{equation}
\label{eq:ce:quil:faut:verifier}
\begin{split}
d \bigl\langle \widetilde S_{t},\phi \rangle
&= 
- \ell_{t} (\bb + q + \varphi_{t}^{\infty}) 
  \overline \mu_t
   \bigl\langle
\widetilde S_{t},
 \phi' 
 \bigr\rangle 
 dt
 - \ell_{t} (\bb + q + \varphi_{t}^{\infty}) 
   \bigl\langle
 S_{t},
\textrm{\rm Id} \bigr\rangle 
   \bigl\langle
 \mu_{t},
\phi'(\ell_{t} \cdot) \bigr\rangle dt
 \\ 
&\hspace{15pt} + \frac{\ell_{t}^2}2 \bigl\langle \widetilde S_{t},
( \sigma^2 
+ \sigma_{0}^2 ) \phi''  \bigr\rangle  dt +
\sigma_{0}
\ell_{t} \bigl\langle \widetilde S_{t},  \phi' 
\bigr\rangle d W_{t} + {\sigma}
 \ell_{t}\bigl\langle \sqrt{\mu_{t}(\cdot)} \phi'(\ell_{t} \cdot), d\beta_{t}(\cdot) \bigr\rangle
 \\
 &= - \ell_{t} (\bb + q + \varphi_{t}^{\infty}) 
  \overline \mu_t
   \bigl\langle
\widetilde S_{t},
 \phi' 
 \bigr\rangle 
 dt
 -  (\bb + q + \varphi_{t}^{\infty}) 
   \bigl\langle
 S_{t},
\textrm{\rm Id} \bigr\rangle 
   \bigl\langle
 \mu_{t}\bigl( \frac{\cdot}{\ell_{t}} \bigr),
\phi'(\cdot) \bigr\rangle dt
 \\ 
&\hspace{15pt} + \frac{\ell_{t}^2}2 \bigl\langle \widetilde S_{t},
( \sigma^2 
+ \sigma_{0}^2 ) \phi''  \bigr\rangle  dt +
\sigma_{0}
\ell_{t} \bigl\langle \widetilde S_{t},  \phi' 
\bigr\rangle d W_{t} + {\sigma}
 \sqrt{\ell_{t}} \Bigl\langle \sqrt{\mu_{t}(\frac{\cdot}{\ell_{t}})} \phi'( \cdot), d\widetilde \beta_{t}(\cdot) \Bigr\rangle,
\end{split}
\end{equation}
where $\widetilde \beta$ is a cylindrical Brownian motion with values in $L^2(\RR)$, independent of $(\theta_{0},W)$, given by
\begin{equation*}
\begin{split}
\langle \phi(\cdot),d\widetilde \beta_{t}(\cdot)\rangle
= \sqrt{\ell_{t}}\langle \phi(\ell_{t} \cdot),d \beta_{t}(\cdot)\rangle.
\end{split}
\end{equation*}
In particular, the stochastic term $\langle \sqrt{\mu}_{t},d \beta_{t} \rangle 
$ in the representation formula \eqref{eq:St:Id} may be rewritten as
$\langle \sqrt{\mu}_{t},d \beta_{t} \rangle 
= \sqrt{\ell_{t}} \langle 
\sqrt{\mu_{t}(\cdot/\ell_{t})}, d \widetilde{\beta}_{t}(\cdot)
\rangle$.

Letting 
$\psi_{t} = - \int_{0}^t \ell_{s} \bigl( \bar b + q + \varphi_{s}^\infty \bigr) \bar \mu_{s} ds + 
\int_{0}^t \sigma_{0} \ell_{s} dW_{s}$ for $t \in [0,T]$, this prompts us to let, as
a candidate for solving 
\eqref{eq:ce:quil:faut:verifier},
\begin{equation*}
\begin{split}
\widetilde{S}_{t}(x) &=
\bigl\langle 
\theta_{0}, g_{0,t}(\cdot - x)  
\bigr\rangle 
- \int_{0}^t 
\bigl( \bar b + q + \varphi_{s}^\infty
\bigr)
\bigl\langle S_{s},\textrm{\rm Id}
\bigr\rangle
\bigl\langle g_{s,t}'(\cdot -x),
\mu_{s}\bigl(\frac{\cdot}{\ell_{s}}
\bigr)
\bigr\rangle
 ds
 \\
&\hspace{15pt} + \int_{0}^t {\sigma} \sqrt{\ell_{s}} \Bigl\langle g_{s,t}' (\cdot - x) 
\sqrt{\mu_{s}\bigl( \frac{\cdot}{\ell_{s}} \bigr)} , d \widetilde \beta_{s}(\cdot)
\Bigr\rangle,
\end{split}
\end{equation*}
which makes sense in a distributional sense, 
with
\begin{equation*}
g_{s,t}(x) = \frac{1}{\sqrt{\sigma^2\int_{s}^t \ell^2_{r} dr}} \Phi \biggl( \frac{x + \psi_{t} - \psi_{s}}{\sqrt{\sigma^2\int_{s}^t \ell^2_{r} dr}}
\biggr) 
\end{equation*}
where $\Phi$ is the standard Gaussian density.
To make it clear, for $\phi \in {\mathcal H}^{2\lambda+2,\lambda}$, 
\begin{equation}
\label{eq:widetilde:St}
\begin{split}
\bigl\langle \widetilde{S}_{t},\phi \bigr\rangle &=
\bigl\langle 
\theta_{0}, g_{0,t} * \phi
\bigr\rangle 
- \int_{0}^t 
\bigl( \bar b + q + \varphi_{s}^\infty
\bigr)
\bigl\langle S_{s},\textrm{\rm Id}
\bigr\rangle
\bigl\langle g_{s,t} * \phi',
\mu_{s}\bigl(\frac{\cdot}{\ell_{s}}
\bigr)
\bigr\rangle
 ds
\\
&\hspace{15pt} + \int_{0}^t {\sigma} \sqrt{\ell_{s}} \Bigl\langle \bigl( g_{s,t} * \phi' 
 \bigr)(\cdot) 
\sqrt{\mu_{s}\bigl( \frac{\cdot}{\ell_{s}} \bigr)} , d \widetilde \beta_{s}(\cdot)
\Bigr\rangle,
\end{split}
\end{equation}
where $*$ is the usual convolution operator. 
We prove below that this defines a 
square-integrable process with values in 
${\mathcal C}([0,T];{\mathcal H}^{-(2\lambda+2),\lambda})$.
Then, by standard arguments in stochastic calculus, 
$((g_{s,t} * \phi)(x))_{t \geq s}$ is a semi-martingale for any $x \in \RR^d$ and 
any $\phi \in {\mathcal H}^{2,\lambda+1}$ and, with probability 1, 
\begin{equation*}
\begin{split}
d_{t}\bigl( g_{s,t} * \phi\bigr)(x)
&= 
- \ell_{t} (\bb + q + \varphi_{t}^{\infty}) 
  \overline \mu_t
 \bigl( g_{s,t} * \phi'\bigr)(x)
 \\ 
&\hspace{15pt} + \frac{\ell_{t}^2}2 
{\sigma^2 }   \bigl( g_{s,t} * \phi''\bigr)(x)   dt +
\sigma_{0}
\ell_{t} 
 \bigl( g_{s,t} * \phi'\bigr)(x)
 d W_{t}, \quad t \geq s,
 \end{split}
 \end{equation*}
 with $\phi(x)$ as initial condition.
 Following the proof of the It\^o-Wentzell formula (see, for instance, \cite{Kunita.book}), 
 we deduce that \eqref{eq:ce:quil:faut:verifier} holds true for 
 $\phi \in {\mathcal H}^{2 \lambda + 4,\lambda-1}$. By 
 \eqref{eq:St:phit}, this suffices to identify the limit in the central limit theorem.  
 
It then remains to check that 
\eqref{eq:widetilde:St} defines a square-integrable process with values in 
${\mathcal C}([0,T];{\mathcal H}^{2\lambda+2,\lambda})$. 
This is easily checked for the first two terms in the right-hand side, but 
the {verification for the } last term is more involved. 
To prove that it generates a square-integrable process with values in 
${\mathcal C}([0,T];{\mathcal H}^{-(2\lambda+2),\lambda})$, it suffices to prove that 
\begin{equation*}
{\mathbb E}
\biggl[ \sup_{0 \le t \le T} 
\biggl\vert \int_{0}^t \sqrt{\ell_{s}} \Bigl\langle \bigl( g_{s,t} * \phi' 
 \bigr)(\cdot) 
\sqrt{\mu_{s}\bigl( \frac{\cdot}{\ell_{s}} \bigr)} , d \widetilde \beta_{s}(\cdot)
\Bigr\rangle
\biggr\vert^2 \biggr] 
\leq C \| \phi \|_{2,\lambda+1}^2,
\end{equation*}
and then to use the fact the embedding from 
${\mathcal H}^{2+2\lambda,\lambda}$
into
${\mathcal H}^{2,\lambda+1}$ is Hilbert Schmidt.  
To do so, we 
use the \textit{factorization method}
for SPDEs, see 
\cite[Chapter 7]{MR3236753}. Notice indeed that, for $\varepsilon >0$ small enough and for $\phi \in {\mathcal H}^{2,\lambda+1}$,
\begin{equation*}
\begin{split}
&\int_{0}^t \sqrt{\ell_{s}} \Bigl\langle \bigl( g_{s,t} * \phi' 
 \bigr)(\cdot) 
\sqrt{\mu_{s}\bigl( \frac{\cdot}{\ell_{s}} \bigr)} , d \widetilde \beta_{s}(\cdot)
\Bigr\rangle
\\
&= c_{\varepsilon} \int_{0}^t (t-r)^{-1+\varepsilon} 
\biggl[ \int_{\RR}
 g_{r,t}(x)
\biggl( \int_{0}^{r}
\Bigl\langle 
(r-s)^{-\varepsilon}\bigl(g_{s,r} * \phi' \bigr)(\cdot-x) 
\sqrt{\ell_{s}} \sqrt{ \mu_{s}\bigl( \frac{\cdot}{\ell_{s}} \bigr)},  d
 \widetilde \beta_{s}(\cdot) \Bigr\rangle \biggr)  dx \biggr] dr,
\end{split}
\end{equation*}
where $c_{\varepsilon}$ is a normalization constant such that 
$c_{\varepsilon} \int_{s}^t (t- r)^{-1+\varepsilon} (r-s)^{-\varepsilon} d r=1$. 
By H\"older's inequality, we notice that, for $q>1$, the right-hand side is less than
\begin{equation*}
\begin{split}
c_{\varepsilon}
&\biggl( \int_{0}^t 
 \int_{\RR}
(t-r)^{(-1+\varepsilon)q/(q-1)}
\bigl( 1+ \vert x \vert^{\lambda+3}
\bigr)^{q/(q-1)}
 g_{r,t}^{q/(q-1)}(x)
\, dr \, dx \biggr)^{(q-1)/q} \\
&\times \biggl(
 \int_{0}^{T}
\int_{\RR}
\bigl( 1+ \vert x \vert^{\lambda+3}
\bigr)^{-q}
\biggl\vert \int_{0}^r \Bigl\langle 
(r-s)^{-\varepsilon}\bigl(g_{s,r} * \phi' \bigr)(\cdot-x) 
\sqrt{\ell_{s}} \sqrt{ \mu_{s}\bigl( \frac{\cdot}{\ell_{s}} \bigr)},  d
 \widetilde \beta_{s}(\cdot) \Bigr\rangle \biggr\vert^q   dx \biggr)^{1/q} dr.
\end{split} 
 \end{equation*}
Choosing $q$ large enough, we can find a constant $C$ (the value of which is allowed to change from line to line), independent of $t \in [0,T]$, such that 
\begin{equation*}
\begin{split}
\int_{0}^t \int_{\RR} & (t-r)^{(-1+\varepsilon)q/(q-1)}
\bigl( 1+ \vert x \vert^{\lambda+3}
\bigr)^{q/(q-1)}
 g_{r,t}^{q/(q-1)}(x)
 dr dx  
 \\
 &\leq C \bigl( 1 + \sup_{s \in [0,T]} \vert \psi_{s} \vert \bigr)^{(\lambda+3)q/(q-1)} \int_{0}^t (t-r)^{-1+\varepsilon/2}
\\
&\leq C\bigl( 1 + \sup_{s \in [0,T]} \vert \psi_{s} \vert \bigr)^{(\lambda+3)q/(q-1)}.
\end{split}
\end{equation*}
Hence, it suffices to prove that, for any $q \geq 2$,
\begin{equation*}
\sup_{r \in [0,T]}
{\mathbb E} \biggl[ \int_{\RR}
\bigl( 1+ \vert x \vert^{\lambda+3}
\bigr)^{-q}
\biggl\vert \int_{0}^r \Bigl\langle 
(r-s)^{-\varepsilon}\bigl(g_{s,r} * \phi' \bigr)(\cdot-x) 
\sqrt{\ell_{s}} \sqrt{ \mu_{s}\bigl( \frac{\cdot}{\ell_{s}} \bigr)},  d
 \widetilde \beta_{s}(\cdot) \Bigr\rangle \biggr\vert^q   dx \biggr] < \infty. 
\end{equation*}
For $r \in [0,T]$, {the Burkholder-Davis-Gundy inequality and boundedness of $\ell$ yield}
\begin{equation*}
\begin{split}
&{\mathbb E} \biggl[ \int_{\RR}
\bigl( 1+ \vert x \vert^{\lambda+3}
\bigr)^{-q}
\biggl\vert \int_{0}^r \Bigl\langle 
(r-s)^{-\varepsilon}\bigl(g_{s,r} * \phi' \bigr)(\cdot-x) 
\sqrt{\ell_{s}} \sqrt{ \mu_{s}\bigl( \frac{\cdot}{\ell_{s}} \bigr)},  d
 \widetilde \beta_{s}(\cdot) \Bigr\rangle \biggr\vert^q   dx \biggr] 
 \\
 &\leq C 
 {\mathbb E} \biggl[ \int_{\RR}
\bigl( 1+ \vert x \vert^{\lambda+3}
\bigr)^{-q}
\biggl\vert 
\int_{0}^r 
\int_{\R}
(r-s)^{-2 \varepsilon} \bigl\vert 
\bigl(g_{s,r} * \phi' \bigr)(y-x) 
\bigr\vert^2 \mu_{s}\bigl( \frac{y}{\ell_{s}} \bigr) ds \, dy \biggr\vert^{q/2}   dx \biggr] 
 \\
 &\leq C 
 {\mathbb E} \biggl[
 \int_{0}^r 
\int_{\R}
 \int_{\RR}
\bigl( 1+ \vert x \vert^{\lambda+3}
\bigr)^{-q}
(r-s)^{-2 \varepsilon}  \bigl\vert 
\bigl(g_{s,r} * \phi' \bigr)(y-x) 
\bigr\vert^{q} \mu_{s}\bigl( \frac{y}{\ell_{s}} \bigr) ds \, dx \,   dy \biggr]. 
 \end{split}
\end{equation*}
Noticing from 
\eqref{eq:embedding:1}
that 
$\vert (g_{s,r} * \phi')(x) \vert \leq C \| \phi \|_{2,\lambda+1}
( 1+ \vert x \vert + \sup_{t \in [0,T]} \vert \psi_{t} \vert)^{\lambda+1}$, the proof
is easily completed.

\appendix 

\section{On derivatives on Wasserstein space} \label{ap:derivatives}

In this section we state two simple but useful technical results on the derivative $D_m$, and we prove two results from the body of the paper.

\begin{lemma} \label{le:commute}
For any  function $\bU = \bU(x,m) : \R^d \times \P^{\newexp}(\R^d) \rightarrow \R$, we have
\[
D_xD_m\bU(x,m,v) = D_mD_x\bU(x,m,v),  \qquad \mbox{ for all }  v \in \R^d,
\]
as long as the derivatives on both sides exist and are jointly continuous. More precisely, assume that $\bU$ is ${{\mathscr C}^1}$ in the sense of Section \ref{subse:derivatives:m:P2}, that the derivatives $D_x(\delta \bU/\delta m)$, $D_x\bU$, and $\delta(D_x\bU)/\delta m$ exist, and that for every compact set $K \subset \P^{\newexp}(\R^d)$ there exists $c < \infty$ such that $\sup_{m \in K}|D_x\frac{\delta \bU}{\delta m}(x,m,v)| \le c(1 + |x|^\newexp + |v|^\newexp)$ for all $x,v \in \R^d$.
\end{lemma}
\begin{proof}
Fix $x \in \R^d$, $m,m ' \in \P^{\newexp}(\R^d)$, and for $t \in [0,1]$ let $m^t = (1-t)m + tm'$. 
Then by \eqref{def:measure-derivative}, 
\begin{align*}
\bU(x,m') - \bU(x,m) &= \int_0^1\int_{\R^d}\frac{\delta \bU}{\delta m}(x,m^t,v)(m'-m)(dv) dt. 
\end{align*}
Apply $D_x$ to both sides and use dominated convergence to interchange the order of $D_x$ and the integral to get
\begin{align*}
D_x \bU(x,m') - D_x\bU(x,m) &= \int_0^1\int_{\R^d}D_x\frac{\delta \bU}{\delta m}(x,m^t,v)(m'-m)(dv) dt.
\end{align*} 
Another application of \eqref{def:measure-derivative}  shows that for every $v \in \R^d$,
\[
\frac{\delta D_x\bU}{\delta m}(x,m,v) = D_x\frac{\delta \bU}{\delta m}(x,m,v).
\]
To conclude the proof,  apply $D_v$ to both sides, commute $D_v$ and $D_x$ on the right-hand side, and use the definition of $D_m$ from \eqref{intrinsic}. 
\end{proof}

\begin{lemma} \label{le:lineargrowth}
Suppose a function $V : [0,T] \times \R^d \times \P^{\newexp}(\R^d) \rightarrow \R$ is continuous. Suppose also that $D_xV$ and $D_mV$ exist and are continuous and bounded.
Then there exists $C <  \infty$ such that, for all $(t,x,m)$,
\[
|V(t,x,m)| \le C\left(1 + |x| + \W_1(m,\delta_0)\right), 
\]
where $\W_1$ is the Wasserstein distance on ${\mathcal P}(\R^d)$ defined in  Section \ref{subs-notation}. 
\end{lemma}
\begin{proof}
Begin by writing
\begin{align*}
V&(t,x,m) - V(t,0,\delta_0) \\
	&= \int_0^1\left(x \cdot D_xV(t,ux,(1-u)\delta_0 + um) + \int_{\R^d}\frac{\delta V}{\delta m}(t,ux,(1-u)\delta_0 + um,v)(m-\delta_0)(dv)\right)du.
\end{align*}
Since $|D_mV|$ is bounded, say by a constant $L <  \infty$, the map $v \mapsto \frac{\delta V}{\delta m}(t,x,m,v)$ is $L$-Lipschitz, for each $(t,x,m)$. Hence, Kantorovich duality implies
\begin{align*}
\left|\int_{\R^d}\frac{\delta V}{\delta m}(t,ux,(1-u)\delta_0 + um,v)(m-\delta_0)(dv)\right| \le L \W_1(m,\delta_0).
\end{align*}
If $C < \infty$ satisfies $|D_xV| \le C$ pointwise, then we conclude
\[
|V(t,x,m)| \le |V(t,0,\delta_0)| + C|x| + L\W_1(m,\delta_0).
\]
To complete the proof, note that the continuous function $V(\cdot,0,\delta_0)$ is bounded on $[0,T]$.
\end{proof}

\subsection*{Proof of Proposition \ref{pr:empiricalmeasure}} Let $m \in \P^{\newexp}(\R^d)$ and $\bm{x} = (x_1,\ldots,x_n) \in (\R^d)^n$.
By continuity, it suffices to prove the claims assuming the points $x_1,\ldots,x_n \in \R^d$ are distinct.  Fix an index $j \in \{1,\ldots,n\}$ and a bounded continuous function $\phi : \R^d \rightarrow \R^d$, to be specified later. We claim that, under the assumptions of part (i), 
\begin{align}
\lim_{h\downarrow 0 }\frac{\bU(m \circ (\mathrm{Id} + h\phi)^{-1} )- \bU(m)}{h} = \int_{\R^d}D_m\bU(m,v) \cdot \phi(v)m(dv) \label{pf:Uderivphi}
\end{align}
holds, where $\mathrm{Id}$ denotes the identity map on $\R^d$.
 Once \eqref{pf:Uderivphi} is proven, we complete the proof as follows. For a fixed vector $v \in \R^d$ we may choose $\phi$ such that $\phi(x_j) = v$ while $\phi(x_i)=0$ for $i \neq j$. Let $\bm{v} \in (\R^d)^n$ have $j^\text{th}$ coordinate equal to $v$ and $i^\text{th}$ coordinate zero for $i \neq j$.  
Then $u_n(\bm{x}) = \bU (m_{\bm{x}}^n)$ satisfies
\begin{align*}
\lim_{h \downarrow 0}\frac{u_n(\bm{x} + h\bm{v}) - u_n(\bm{x})}{h} &= \lim_{h\downarrow 0 }\frac{\bU(m^n_{\bm{x}} \circ (\mathrm{Id} + h\phi)^{-1} )- \bU(m^n_{\bm{x}})}{h} \\
	&= \frac{1}{n}\sum_{k=1}^nD_m\bU(m^n_{\bm{x}},x_k) \cdot \phi(x_k) \\
	&= \frac{1}{n}D_m\bU(m^n_{\bm{x}},x_j) \cdot v.
\end{align*}
This proves (i). Under the additional assumptions,  (ii) follows by applying (i) again.

It remains to prove \eqref{pf:Uderivphi}.
For $h > 0$, $t \in [0,1]$, and $m \in \P^{\newexp}(\R^d)$, let $m_{h,t} = t m \circ (\mathrm{Id} + h\phi)^{-1} + (1-t)m$.  Then, using 
\eqref{def:measure-derivative} and \eqref{intrinsic}, respectively, in the first and third equalities below, we obtain 
\begin{align*}
\bU(m \circ (\mathrm{Id} + h\phi)^{-1} - \bU(m) &= \int_0^1\int_{\R^d}\frac{\delta \bU}{\delta m}(m_{h,t},v)\,(m \circ (\mathrm{Id} + h\phi)^{-1} - m)(dv)\,dt \\
	&= \int_0^1\int_{\R^d}\left(\frac{\delta \bU}{\delta m}(m_{h,t},v + h\phi(v)) - \frac{\delta \bU}{\delta m}(m_{h,t},v)\right)m(dv)\,dt \\
	&= h\int_0^1\int_{\R^d} \int_0^1 D_m\bU(m_{h,t},v + sh\phi(v)) \cdot \phi(v)\,ds\,m(dv)\,dt.
\end{align*}
As $h\downarrow 0$ we have $m_{h,t} \rightarrow m$  and $sh \phi(v) \rightarrow 0$,   and we deduce \eqref{pf:Uderivphi} from the bounded convergence theorem and continuity of $D_m$.
\hfill \qed

\subsection*{Proof of Proposition \ref{pr:master-nearly-nash}}
We follow the proof of Proposition 6.3 in \cite{cardaliaguet-delarue-lasry-lions}.  For fixed $n \in \N$ and $x \in \R^d$,  recall that $u^{n,i} (t, \bm{x}) = U (t, x_i, m_{\bm{x}}^n)$, and use Proposition \ref{pr:empiricalmeasure} and Lemma \ref{le:commute}   to derive, for distinct $i, j, k \in \{1, \ldots, n\}$
\begin{align}
D_{x_i}u^{n,i}(t,\bm{x}) &= D_xU(t,x_i,m^n_{\bm{x}}) + \frac{1}{n}D_mU(t,x_i,m^n_{\bm{x}},x_i) \label{pf:deriv1.1} \\
D_{x_j}u^{n,i}(t,\bm{x}) &= \frac{1}{n}D_mU(t,x_i,m^n_{\bm{x}},x_j) \label{pf:deriv1} \\
D^2_{x_i,x_i}u^{n,i}(t,\bm{x}) &= D_x^2U(t,x_i,m^n_{\bm{x}}) + \frac{1}{n}D_xD_mU(t,x_i,m^n_{\bm{x}},x_i) \nonumber \\
	&\quad  + \frac{1}{n}D_vD_mU(t,x_i,m^n_{\bm{x}},x_i) + \frac{1}{n^2}D_m^2U(t,x_i,m^n_{\bm{x}},x_i,x_i) \label{pf:deriv2.1} \\
D^2_{x_i,x_j}u^{n,i}(t,\bm{x}) &= \frac{1}{n}D_xD_mU(t,x_i,m^n_{\bm{x}},x_j) + \frac{1}{n^2}D_m^2U(t,x_i,m^n_{\bm{x}},x_i,x_j) \label{pf:deriv2} \\
D^2_{x_j,x_j}u^{n,i}(t,\bm{x}) &= \frac{1}{n}D_vD_mU(t,x_i,m^n_{\bm{x}},x_j) + \frac{1}{n^2}D^2_mU(t,x_i,m^n_{\bm{x}},x_j,x_j) \label{pf:deriv3} \\
D^2_{x_j,x_k}u^{n,i}(t,\bm{x}) &= \frac{1}{n^2}D^2_mU(t,x_i,m^n_{\bm{x}},x_j,x_k). \label{pf:deriv4}
\end{align}
Use the master equation \eqref{def:masterequation} to find
\begin{align*}
0 &= \partial_tu^{n,i}(t,\bm{x})+ \Phi^{n,i}(t,\bm{x}) + \Psi^{n,i}(t,\bm{x})  + \Xi^{n,i}(t,\bm{x}), 
\end{align*}
where we define
\begin{align*}
\Psi^{n,i}(t,\bm{x}) &= H(x,m^n_{\bm{x}},D_xU(t,x_i,m^n_{\bm{x}})) + \int_{\R^d}\widehat{b}\left(v,m^n_{\bm{x}},D_xU(t,v,m^n_{\bm{x}})\right) \cdot D_mU(t,x_i,m^n_{\bm{x}},v) \,dm^n_{\bm{x}}(v), \\
\Phi^{n,i}(t,\bm{x}) &= \frac{1}{2}\int_{\R^d}\mathrm{Tr}\left[\sigma\sigma^\top D_vD_mU(t,x_i,m^n_{\bm{x}},v)\right]\,dm^n_{\bm{x}}(v) + \frac{1}{2}\mathrm{Tr}\left[\sigma\sigma^\top D_x^2U(t,x_i,m^n_{\bm{x}})\right],
\end{align*}
and
\begin{align*}
\Xi^{n,i}(t,\bm{x}) = &\frac{1}{2}\int_{\R^d}\mathrm{Tr}\left[\sigma_0\sigma_0^\top D_v D_mU(t,x_i,m^n_{\bm{x}},v)\right]\,dm^n_{\bm{x}}(v) \\
	&+\frac{1}{2}\int_{\R^d}\int_{\R^d}\mathrm{Tr}\left[\sigma_0\sigma_0^\top D^2_mU(t,x_i,m^n_{\bm{x}},v,v')\right]\,dm^n_{\bm{x}}(v)\,dm^n_{\bm{x}}(v') \\
&+ \int_{\R^d}\mathrm{Tr}\left[\sigma_0 \sigma_0^\top  D_xD_mU(t,x_i,m^n_{\bm{x}},v)\right]\,dm^n_{\bm{x}}(v)  + \frac{1}{2}\mathrm{Tr}\left[\sigma_0\sigma_0^\top D_x^2U(t,x_i,m^n_{\bm{x}})\right].
\end{align*}
Now use \eqref{pf:deriv1}, \eqref{pf:deriv1.1}  and \eqref{exp-ham} to get
\begin{align*}
\Psi^{n,i}(t,\bm{x}) &= \frac{1}{n}\sum_{j = 1}^n\widehat{b}\left(x_j,m^n_{\bm{x}},D_xU(t,x_j,m^n_{\bm{x}})\right) \cdot D_mU(t,x_i,m^n_{\bm{x}},x_j) \\
	&\quad + H(x,m^n_{\bm{x}},D_xU(t,x_i,m^n_{\bm{x}})) \\
	&= \sum_{j = 1, j \neq i}^n\widehat{b}\left(x_j,m^n_{\bm{x}},D_xU(t,x_j,m^n_{\bm{x}})\right) \cdot D_{x_j}u^{n,i}(t,\bm{x}) \\
	&\quad + \widehat{b}\left(x_i,m^n_{\bm{x}},D_xU(t,x_i,m^n_{\bm{x}})\right) \cdot \left(D_{x_i}u^{n,i}(t,\bm{x}) - D_xU(t,x_i,m^n_{\bm{x}})\right) \\
	&\quad + \widehat{b}\left(x_i,m^n_{\bm{x}},D_xU(t,x_i,m^n_{\bm{x}})\right) \cdot D_xU(t,x_i,m^n_{\bm{x}}) + \widehat{f}(x_i,m^n_{\bm{x}},D_xU(t,x_i,m^n_{\bm{x}})) \\
	&= \sum_{j = 1}^n\widehat{b}\left(x_j,m^n_{\bm{x}},D_xU(t,x_j,m^n_{\bm{x}})\right) \cdot D_{x_j}u^{n,i}(t,\bm{x}) + \widehat{f}(x_i,m^n_{\bm{x}},D_xU(t,x_i,m^n_{\bm{x}})).
\end{align*}

Now, recall from Assumption \ref{assumption:A}(5) that $D_mU$ and $D_xU$ are uniformly bounded. 
Thus,  since $u^{n,j}(t,\bm{x}) = U(t,x_j,m^n_{\bm{x}})$, we use \eqref{pf:deriv1.1} and \eqref{pf:deriv1} to obtain  $\|D_{x_j}u^{n,i}\|_\infty \le C/n$ for $j \neq i$ and $\|D_{x_i}u^{n,i}\|_\infty \le C$.  
Moreover, 
$\widehat{b}(x,m,y)$ and $\widehat{f}(x,m,y)$ are locally Lipschitz in $y$ by respectively 
Assumptions \ref{assumption:A}(1) and either Assumption \ref{assumption:B} or \ref{assumption:B'}(3).   
 Hence,  using \eqref{pf:deriv1.1}, we have 
\begin{align*}
&\left|\Psi^{n,i}(t,\bm{x}) - \sum_{j = 1}^n\widehat{b}\left(x_j,m^n_{\bm{x}},D_{x_j}u^{n,j}(t,\bm{x})\right) \cdot D_{x_j}u^{n,i}(t,\bm{x}) - \widehat{f}\left(x_i,m^n_{\bm{x}},D_{x_i}u^{n,i}(t,\bm{x})\right)\right| \\
	&\le C \sum_{j=1}^n\left|D_{x_j}u^{n,i}(t,\bm{x})\right|\left|D_xU(t,x_j,m^n_{\bm{x}}) - D_{x_j}u^{n,j}(t,\bm{x})\right|  \\
	&\quad\quad + C\left(1+|D_xU(t,x_i,m^n_{\bm{x}})| + |D_{x_i}u^{n,i}(t,\bm{x})|\right)\left|D_xU(t,x_i,m^n_{\bm{x}}) - D_{x_i}u^{n,i}(t,\bm{x})\right| \\
	&\le \frac{C}{n}\sum_{j=1}^n\left|D_{x_j}u^{n,i}(t,\bm{x})\right|\left|D_mU(t,x_j,m^n_{\bm{x}},x_j)\right| + \frac{C}{n}\left|D_mU(t,x_i,m^n_{\bm{x}},x_i)\right| \\
	&\le \frac{C}{n}.
\end{align*}

The argument for the volatility terms is similar, but  somewhat more involved. Expand the integrals and use   \eqref{pf:deriv2.1} and \eqref{pf:deriv3} to get
\begin{align*}
\Phi^{n,i}(t,\bm{x}) 
= &\frac{1}{n}\sum_{j = 1}^n\frac{1}{2}\mathrm{Tr}\left[\sigma\sigma^\top  D_vD_mU(t,x_i,m^n_{\bm{x}},x_j)\right]  + \frac{1}{2}\mathrm{Tr}\left[\sigma\sigma^\top D_x^2U(t,x_i,m^n_{\bm{x}})\right] \\
	= &\frac{1}{2}\mathrm{Tr}\left[\sigma\sigma^\top \left(D_x^2U(t,x_i,m^n_{\bm{x}}) + \frac{1}{n}D_v D_mU(t,x_i,m^n_{\bm{x}}, x_i)\right)\right] \\
	&+ \frac{1}{n}\sum_{j=1,j \neq i}^n \frac{1}{2}\mathrm{Tr}\left[\sigma\sigma^\top  D_v D_mU(t,x_i,m^n_{\bm{x}},x_j)\right] \\
	= &\frac{1}{2}\mathrm{Tr}\left[\sigma\sigma^\top \left(D_{x_i,x_i}^2u^{n,i}(t,\bm{x}) - \frac{1}{n}D_xD_mU(t,x_i,m^n_{\bm{x}},x_i) - \frac{1}{n^2}D_m^2U(t,x_i,m^n_{\bm{x}},x_i,x_i)\right)\right] \\
	&+ \sum_{j=1, j \neq i}^n \frac{1}{2}\mathrm{Tr}\left[\sigma\sigma^\top  \left(D_{x_j,x_j}^2u^{n,i}(t,\bm{x}) - \frac{1}{n^2}D_m^2U(t,x_i,m^n_{\bm{x}},x_j,x_j)\right)\right].
\end{align*}
Hence, using the bounds from Assumption \ref{assumption:A}(5), we obtain  
\begin{align*}
&\left|\Phi^{n,i}(t,\bm{x}) - \sum_{j = 1}^n\frac{1}{2}\mathrm{Tr}\left[\sigma\sigma^\top D_{x_j,x_j}^2u^{n,i}(t,\bm{x})\right]\right| \\
	&\quad\quad\quad\quad\le \frac{1}{2n}\left|\sigma\sigma^\top \left(D_xD_mU(t,x_i,m^n_{\bm{x}},x_i) + \frac{1}{n}D_m^2U(t,x_i,m^n_{\bm{x}},x_i,x_i)\right)\right| \\
	&\quad\quad\quad\quad\quad\quad + \frac{1}{2n^2}\sum_{j=1,j \neq i}^n \left|\sigma\sigma^\top D_m^2U(t,x_i,m^n_{\bm{x}},x_j,x_j)\right| \\
	&\quad\quad\quad\quad \le \frac{C}{n}.
\end{align*}
Finally, we use \eqref{pf:deriv2.1}, \eqref{pf:deriv2}, \eqref{pf:deriv3}, and \eqref{pf:deriv4} to get
\begin{align*}
\Xi^{n,i}(t,\bm{x}) 
= \ &\frac{1}{n}\sum_{j = 1}^n\frac{1}{2}\mathrm{Tr}\left[\sigma_0\sigma_0^\top  D_vD_mU(t,x_i,m^n_{\bm{x}},x_j)\right]  + \frac{1}{n^2}\sum_{j = 1}^n\sum_{k = 1}^n\frac{1}{2}\mathrm{Tr}\left[\sigma_0 \sigma_0^\top D^2_mU(t,x_i,m^n_{\bm{x}},x_j,x_k)\right] \\
	& + \frac{1}{n}\sum_{j =1}^n\mathrm{Tr}\left[\sigma_0 \sigma_0^\top  D_xD_mU(t,x_i,m^n_{\bm{x}},x_j)\right] + \frac{1}{2}\mathrm{Tr}\left[\sigma_0\sigma_0^\top D_x^2U(t,x_i,m^n_{\bm{x}})\right] \\
	= \ &\frac{1}{2}\mathrm{Tr}\left[\sigma_0\sigma_0^\top \left(\frac{1}{n}D_vD_mU(t,x_i,m^n_{\bm{x}},x_i) + \frac{1}{n}D_xD_mU(t,x_i,m^n_{\bm{x}},x_i) \right.\right. \\
	&\quad\quad\quad\quad\quad\quad\quad\quad\quad \left.\left. \vphantom{\frac{1}{n}}+ D_x^2U(t,x_i,m^n_{\bm{x}}) + \frac{1}{n^2}D^2_mU(t,x_i,m^n_{\bm{x}},x_i,x_i)\right)\right] \\
	&+ \frac{1}{2}\sum_{j=1, j \neq i}^n \mathrm{Tr}\left[\sigma_0\sigma_0^\top  \left(\frac{1}{n}D_vD_mU(t,x_i,m^n_{\bm{x}},x_j) + \frac{1}{n^2}D^2_mU(t,x_i,m^n_{\bm{x}},x_j,x_j)\right)\right] \\
	&+ \sum_{j=1, j \neq i}^n \mathrm{Tr}\left[\sigma_0 \sigma_0^\top \left(\frac{1}{n}D_xD_mU(t,x_i,m^n_{\bm{x}},x_j) + \frac{1}{n^2}D^2_mU(t,x_i,m^n_{\bm{x}},x_i,x_j) \right)\right] \\
	&+ \frac{1}{2n^2}\sum_{j=1, j \neq i}^n \sum_{k = 1, k \notin \{i,j\}}^n \mathrm{Tr}\left[\sigma_0 \sigma_0^\top D^2_mU(t,x_i,m^n_{\bm{x}},x_j,x_k)\right],
\end{align*}
and then
\begin{align*}	
\Xi^{n,i}(t,\bm{x}) 
 	= \ &\frac{1}{2}\mathrm{Tr}\left[\sigma_0\sigma_0^\top  D_{x_i,x_i}^2u^{n,i}(t,\bm{x})\right] + \frac{1}{2}\sum_{j \neq i}\mathrm{Tr}\left[\sigma_0\sigma_0^\top D_{x_j,x_j}^2u^{n,i}(t,\bm{x})\right] \\
	&+ \sum_{j \neq i}\mathrm{Tr}\left[\sigma_0 \sigma_0^\top D_{x_i,x_j}^2u^{n,i}(t,\bm{x})\right] + \frac{1}{2}\sum_{j=1,j \neq i}^n \sum_{k \notin \{i,j\}}\mathrm{Tr}\left[\sigma_0 \sigma_0^\top D_{x_j,x_k}^2u^{n,i}(t,\bm{x})\right] \\
	= \ &\frac{1}{2}\sum_{j = 1}^n\sum_{k=1}^n\mathrm{Tr}\left[\sigma_0 \sigma_0^\top D_{x_j,x_k}^2u^{n,i}(t,\bm{x})\right].
\end{align*}
Complete the proof by combining the above results.
\hfill \qed

\bibliographystyle{amsplain}
\bibliography{MFG-CLT-LDP-bib}

\end{document}